\documentclass[11pt,twoside]{amsart}
\usepackage{latexsym}
\usepackage{amssymb}
\usepackage{amsfonts}
\usepackage[T1]{fontenc}
\usepackage{amstext}
\usepackage{multicol}
\usepackage{mathtools}
\usepackage{amscd,amssymb,amsfonts,amsmath,amstext,amsthm,verbatim,mathdots}
\usepackage{comment,array,blkarray,multirow,tikz,graphicx,ifthen}
\usepackage [margin = 1.3in] {geometry}
\usepackage{youngtab}
\usepackage{ytableau}
\usepackage{cancel}
\numberwithin{equation}{section}

\newtheorem{thm}{Theorem}[section]
\newtheorem{lem}[thm]{Lemma}
\newtheorem{cor}[thm]{Corollary}
\newtheorem{prop}[thm]{Proposition}

\theoremstyle{definition}
\newtheorem{defn}[thm]{Definition}
\newtheorem{conv}[thm]{Convention}
\newtheorem{ex}[thm]{Example}
\newtheorem{remark}[thm]{Remark}

\newcommand{\bb}[1]{\mathbb{#1}}
\newcommand{\bbZ}{\bb{Z}}
\newcommand{\bbC}{\bb{C}}
\newcommand{\bbQ}{\bb{Q}}

\newcommand{\bbX}{\bb{X}}
\newcommand{\tw}[1]{\widetilde{#1}}
\newcommand{\wh}[1]{\widehat{#1}}
\newcommand{\ov}[1]{\overline{#1}}

\newcommand{\Gr}{{\rm Gr}}
\newcommand{\GL}{{\rm GL}}
\newcommand{\Id}{{\rm Id}}
\newcommand{\ds}{\displaystyle}

\newcommand{\Cx}{\mathbb{C}^\times}
\newcommand{\Y}[2]{\Gr(#1,#2) \times \Cx}
\newcommand{\X}[1]{\bbX_{#1}}
\newcommand{\Zn}{\mathbb{Z}/n\mathbb{Z}}

\newcommand{\lp}{\lambda} 
\newcommand{\bT}{\mathbb{T}}
\newcommand{\ve}{\varepsilon}
\newcommand{\vp}{\varphi}
\newcommand{\tep}{\tw{\varepsilon}}
\newcommand{\tph}{\tw{\varphi}}
\newcommand{\te}{\tw{e}}
\newcommand{\tf}{\tw{f}}
\newcommand{\tsi}{\tw{\sigma}}
\newcommand{\tS}{\tw{S}}
\newcommand{\tg}{\tw{\gamma}}

\newcommand{\twf}{\tw{f}}
\newcommand{\twg}{\tw{g}}

\DeclareMathOperator{\PR}{PR}
\DeclareMathOperator{\pr}{pr}
\DeclareMathOperator{\wt}{wt}
\DeclareMathOperator{\Trop}{\text{Trop}}
\DeclareMathOperator{\sh}{sh}
\DeclareMathOperator{\fl}{fl}
\DeclareMathOperator{\rot}{rot}

\DeclareMathOperator{\refl}{refl}
\DeclareMathOperator{\inv}{inv}
\DeclareMathOperator{\sgn}{sgn}
\DeclareMathOperator{\trot}{\tw{\rot}}

\DeclareMathOperator{\tsh}{\tw{\sh}}
\DeclareMathOperator{\trefl}{\tw{\refl}}
\DeclareMathOperator{\tpr}{\tw{\pr}}
\DeclareMathOperator{\adj}{adj}

\title{Affine type $A$ geometric crystal on the Grassmannian}
\author{Gabriel Frieden}
\date{\today}
\thanks{The author was supported in part by NSF grants DMS-1464693 and DMS-0943832.}
\address{Department of Mathematics, University of Michigan, Ann Arbor, MI 48109-1043, USA}
\email{gfrieden@umich.edu}

\begin{document}
\maketitle

\begin{abstract}
We construct a type $A_{n-1}^{(1)}$ geometric crystal on the variety $\Gr(k,n) \times \Cx$, and show that it tropicalizes to the disjoint union of the Kirillov-Reshetikhin crystals corresponding to rectangular tableaux with $n-k$ rows. A key ingredient in our construction is the $\Zn$ symmetry on the Grassmannian coming from cyclically shifting the basis of the underlying vector space. We show that a twisted version of this symmetry tropicalizes to combinatorial promotion. Additionally, we use the loop group $\GL_n(\mathbb{C}(\lp))$ to define a unipotent crystal which induces our geometric crystal. We use this unipotent crystal to study the geometric analogues of two symmetries of rectangular tableaux.
\end{abstract}


\section{Introduction}
\label{sec_intro}

Kashiwara's theory of crystal bases provides a combinatorial model for the representation theory of semisimple Lie algebras, and more generally of Kac-Moody algebras. In type $A_{n-1}$, this theory brings to light an intimate connection between the representation theory of $\mathfrak{sl}_n$ and the combinatorics of semistandard Young tableaux. The operations on tableaux that arise in this theory, such as promotion, evacuation, the crystal operators, the Lascoux-Sch\"{u}tzenberger symmetric group action, and the Robinson-Schensted-Knuth correspondence, are traditionally defined in terms of combinatorial algorithms involving the individual entries of a tableau, such as bumping, sliding, or bracketing rules. When these operations are transferred from tableaux to Gelfand-Tsetlin patterns, they are given by piecewise-linear formulas (\cite{KirBer}, \cite{Kir}, \cite{NouYam}). This suggests that there should be a way to lift the operations to subtraction-free rational functions on some algebraic variety, in such a way that the behavior of the rational functions parallels that of the combinatorial operations. Certain properties of the combinatorial maps may become more transparent in the geometric setting, and information gained at the rational level can be pushed down to the combinatorial level via tropicalization.

Berenstein and Kazhdan's theory of geometric crystals (\cite{BKI}, \cite{BKII}) provides a framework for lifting the crystal combinatorics of semisimple Lie algebras to the rational setting. Nakashima \cite{Nak} extended the theory to the Kac-Moody setting, and in particular to affine Lie algebras. One goal of this theory is to find, for a given family of combinatorial crystals, a geometric crystal which tropicalizes to that family (with respect to a suitable parametrization). For finite-dimensional representations of semisimple Lie algebras, there is a general construction of Berenstein and Kazhdan \cite{BKII} which accomplishes this goal. In the affine case, there has been an ongoing effort to construct geometric crystals corresponding to the crystals of Kirillov-Reshetikhin modules, an important class of finite-dimensional representations of the quantum affine algebras (\cite{KNO1}, \cite{KNO2}).

In type $A_{n-1}^{(1)}$, Kirillov-Reshetikhin modules correspond to rectangular partitions, and their crystal bases are modeled by semistandard Young tableaux of rectangular shape. In addition to the ``classical'' crystal operators defined on tableaux of all shapes, there is an ``affine'' crystal operator $\tw{e}_0$ defined on rectangular tableaux, corresponding to the action of the additional simple root of the Lie algebra $\wh{\mathfrak{sl}_n}$. There are also two crystal-theoretic operations on tensor products of rectangular tableaux---the combinatorial $R$-matrix and the energy function---which have no analogue in the classical setting (\cite{Shim}). Ideally, an affine geometric crystal should come with rational lifts of these operations as well.

In the one-row case, such an affine geometric crystal has been constructed. The underlying variety is $(\Cx)^n$, and a point $(z_1, \ldots, z_n) \in (\Cx)^n$ is the rational analogue of a vector $(b_1, \ldots, b_n) \in (\bbZ_{\geq 0})^n$ which specifies the number of $1's, 2's, \ldots, n's$ in a one-row tableau. The affine geometric crystal structure on this variety was described by Kuniba, Okado, Takagi, and Yamada \cite{KOTY}, and Yamada \cite{Yam} found a rational lift of the combinatorial $R$-matrix for tensor products of one-row tableaux. Lam and Pylyavskyy \cite{LPenergy} showed that a certain loop Schur function provides a rational lift of the energy function for tensor products of one-row tableaux.

Let $\Gr(k,n)$ be the Grassmannian of $k$-dimensional subspaces in $\mathbb{C}^n$, and let $B^{n-k,L}$ denote the Kirillov-Reshetikhin crystal corresponding to rectangular tableaux with $n-k$ rows and $L$ columns. The main contribution of this article is the construction of an affine geometric crystal on the variety $\Gr(k,n) \times \Cx$ which tropicalizes to the disjoint union of the crystals $B^{n-k,L}$.

Our starting point is a type $A_{n-1}$ geometric crystal constructed by Berenstein and Kazhdan \cite{BKII}, which tropicalizes to the ``classical'' crystal structure on rectangular tableaux with $n-k$ rows. This geometric crystal is defined on a certain subvariety of the group of lower triangular matrices in $\GL_n$ (see Remark \ref{rmk_ZLP}), and there is a birational isomorphism from this subvariety to $\Gr(k,n) \times \Cx$. The advantage of using the Grassmannian is that it has a natural $\Zn$ symmetry coming from cyclically shifting the basis vectors of the underlying $n$-dimensional vector space. We define the geometric crystal operator $e_0$ by conjugating the geometric crystal operator $e_1$ by a twisted version of the cyclic shifting map. We then show (Theorem \ref{thm_trop_PR}) that under a suitable parametrization of the Grassmannian, the cyclic shifting map tropicalizes to promotion, which allows us to prove directly that the geometric crystal operators tropicalize to their combinatorial counterparts (Theorem \ref{thm_trop_cryst}).

The idea of relating the cyclic symmetry of the Grassmannian to promotion of rectangular tableaux is not new. In Rhoades' work on the cyclic sieving phenomenon \cite{Rho}, he showed that the map on the homogeneous coordinate ring of the Grassmannian induced by cyclic shifting maps the Kazhdan-Lusztig basis element corresponding to a rectangular tableau $T$ to the basis element corresponding to the promotion of $T$ (up to a sign). Theorem \ref{thm_trop_PR} was inspired by this result; we don't, however, know of any direct connection between the two. More recently, Grinberg and Roby \cite{GrinRob} used the Grassmannian to prove that birational rowmotion on the $k \times (n-k)$ rectangle has order $n$, a result equivalent to Theorem \ref{thm_trop_PR}.

Another important actor in this article is the loop group $\GL_n(\mathbb{C}(\lp))$, which is used to define an appropriate notion of unipotent crystal. A unipotent crystal, as introduced by Berenstein and Kazhdan \cite{BKI}, is a pair $(V,g)$, where $V$ is a variety, and $g$ is a rational map from $V$ to the lower Borel subgroup of a reductive group (satisfying certain axioms). Such an object induces a geometric crystal on $V$, and it transports symmetries of the reductive group to symmetries of the geometric crystal. In our affine setting, the loop group $\GL_n(\bbC(\lp))$ plays the role of the reductive group. We define a unipotent crystal $(\Y{k}{n}, g)$, where $g$ is a rational map from $\Y{k}{n}$ to the lower Iwahori subgroup of $\GL_n(\mathbb{C}(\lp))$. This map intertwines the cyclic symmetry of the loop group and the twisted cyclic shifting map on the Grassmannian. When the parameter $\lp$ is set to 0, we recover a certain type $A_{n-1}$ unipotent crystal from \cite{BKII} (see Proposition \ref{prop_g_props}\eqref{itm:g_0} and Remark \ref{rmk_ZLP}). By specializing $\lp$ to other values, we gain new information; this idea is exploited in various proofs in \S \ref{sec_symm}.

Between the announcement of our main results \cite{F0} and the completion of this article, Misra and Nakashima \cite{MisNakAll} presented an alternative construction of a type $A_{n-1}^{(1)}$ geometric crystal which tropicalizes to a certain limit of the crystals $B^{k,L}$, for fixed $k$. Their construction relies on a description of the affine crystal operator $\tw{e}_0$ in terms of lattice paths, rather than promotion.

\subsection*{Future directions}

We were originally motivated by the problem of finding a rational lift of the combinatorial $R$-matrix for tensor products of rectangular tableaux with an arbitrary number of rows. In the sequel to this paper \cite{F2}, we use the results and techniques developed here to construct such a rational lift.

Perhaps the methods of this paper could be used in other affine types. For instance, is it possible to define a type $C_n^{(1)}$ geometric crystal on the Lagrangian Grassmannian? If so, does this geometric crystal tropicalize to a family of Kirillov-Reshetikhin modules in type $D_{n+1}^{(2)}$, the Langlands dual\footnote{In general, geometric crystals of a given Kac-Moody type tropicalize to combinatorial crystals of the Langlands dual type. We can ignore the distinction in this article because the types $A_{n-1}$ and $A_{n-1}^{(1)}$ are self-dual.} of $C_n^{(1)}$? Is there a corresponding unipotent crystal on the type $C$ loop group?

\subsection*{Outline}

\subsubsection*{\S \ref{sec_combo}: Crystal combinatorics}
We review the combinatorics of the affine type $A$ crystal structure on rectangular tableaux. We translate this combinatorics into piecewise-linear maps on {\em $k$-rectangles}, the subset of Gelfand-Tsetlin patterns which correspond to rectangular tableaux with $k$ rows. We discuss two symmetries of rectangular tableaux: the Sch\"{u}tzenberger involution and the ``column complement'' map. These correspond to rotation and reflection of $k$-rectangles, respectively (Lemma \ref{lem_symms_tab}).

\subsubsection*{\S \ref{sec_geom_cryst}: Geometric crystals} We review the definition of a geometric crystal, and we define an affine geometric crystal on $\Gr(k,n) \times \Cx$ (Definition \ref{defn_maps}). The geometric crystal operators are given by the action of certain elements of one parameter subgroups of $\GL_n$ on $\Gr(k,n)$. For $i = 1, \ldots, n-1$, the one parameter subgroup corresponds to the $i^{th}$ simple root; for $i = 0$, it corresponds to the negative of the highest root. We introduce the twisted cyclic shifting map, and show that the geometric crystal structure is ``compatible'' with this map.

\subsubsection*{\S \ref{sec_param}: Parametrization} This section is devoted to the {\em Gelfand-Tsetlin parametrization} $\Theta_k$, a birational isomorphism from the set of {\em rational $(n-k)$-rectangles} to $\Gr(k,n) \times \Cx$. First we define $\Theta_k$ as a product of simple matrices (Definition \ref{defn_Phi}). Then we review the connection between planar networks and matrices, and we describe $\Theta_k$ in terms of a planar network. We use the Lindstr\"{o}m Lemma to obtain positive combinatorial formulas for the Pl\"{u}cker coordinates in terms of the entries of a rational $(n-k)$-rectangle (Lemma \ref{lem_pluc_formula}). These formulas allow us to prove Proposition \ref{prop_GT_pl}, which gives the inverse of $\Theta_k$.

\subsubsection*{\S \ref{sec_trop}: Tropicalization} This section is the heart of the paper. We start by reviewing {\em tropicalization}, the procedure which transforms a subtraction-free rational function into a piecewise-linear function by replacing multiplication by addition, division by subtraction, and addition by the operation $\min$. Next, we prove that the maps which make up the affine geometric crystal on $\Y{k}{n}$, expressed in terms of the Gelfand-Tsetlin parametrization, tropicalize to piecewise-linear formulas for the combinatorial crystal structure on rectangular tableaux (Theorem \ref{thm_trop_cryst}). The key step in the proof is to show that the twisted cyclic shifting map tropicalizes to promotion (Theorem \ref{thm_trop_PR}). We also show that the tropicalization of a function called the {\em decoration} defines a polyhedron whose integer points are the $k$-rectangles (Proposition \ref{prop_trop_dec}). In \S \ref{sec_trop_ex}, we derive the above-mentioned affine geometric crystal corresponding to one-row tableaux from the $\Gr(n-1,n)$ case of our construction; we also write down the explicit piecewise-linear formulas for promotion coming from our construction in the case $n=4, k=2$. In \S \ref{sec_pf_trop_PR}, we use a rational lift of the Bender-Knuth involutions to prove Theorem \ref{thm_trop_PR}.

\subsubsection*{\S \ref{sec_unip_cryst}: Unipotent crystals} We first explain how to view elements of the loop group as ``infinite periodic matrices.'' Then we define a notion of type $A_{n-1}^{(1)}$ unipotent crystal, and we show that a unipotent crystal induces a geometric crystal (Theorem \ref{thm_induces}). We define a unipotent crystal $(\Y{k}{n}, g)$ (Definition \ref{defn_g}) which induces the geometric crystal defined in \S \ref{sec_geom_Gr}. This reduces the claim that our geometric crystal satisfies the geometric crystal axioms to the claim that our unipotent crystal satisfies the unipotent crystal axioms, which we prove by means of the Grassmann-Pl\"{u}cker relations and the cyclic symmetry (Theorem \ref{thm_Gr_induces}). Finally, we prove several properties of the matrix $g(M,t) \in \GL_n(\bbC(\lp))$ (for $(M,t) \in \Y{k}{n}$), describing what happens when $\lp$ is specialized to certain values (Proposition \ref{prop_g_props}).

\subsubsection*{\S \ref{sec_symm}: Symmetries} We first consider the involution on the loop group given by reflection over the anti-diagonal. We define the {\em geometric Sch\"{u}tzenberger involution} $S$ to be the map on $\Y{k}{n}$ ``induced'' by this map (Definition \ref{defn_S}). Next, we show that the inverse of the matrix $g(M,t)$ is closely related to $g(M^\perp, t)$, where $M^\perp$ is the orthogonal complement of the subspace $M$ (Proposition \ref{prop_D}). By combining the map $M \mapsto M^\perp$ with the geometric Sch\"{u}tzenberger involution, we obtain the {\em duality map} $D : \Y{k}{n} \rightarrow \Y{n-k}{n}$ (Definition \ref{defn_D}). Using the machinery of unipotent crystals, we show that $S$ and $D$ are both ``compatible'' with the geometric crystal structure (Corollaries \ref{cor_S} and \ref{cor_D}). Finally, we show that $S$ and $D$ tropicalize to rotation and reflection of $k$-rectangles, respectively. This allows us to deduce that the corresponding tableau symmetries are compatible with the combinatorial crystal structure (Remark \ref{rmk_symms_cryst_ops}). In the case of the Sch\"{u}tzenberger involution, this compatibility was already known, but in the case of column complementation it seems to be new.

\subsection*{Notation}
\label{sec_notation}

Throughout this article, we fix integers $n \geq 2$ and $k \in \{1, \ldots, n-1\}$. For two integers $i$ and $j$, we write
\[
[i,j] = \{m \in \mathbb{Z} \, | \, i \leq m \leq j\}.
\]
We often abbreviate $[1,j]$ to $[j]$. We write ${[n] \choose k}$ for the set of $k$-element subsets of $[n]$, and $|J|$ for the cardinality of a set $J$. Given a matrix $A$ and two subsets $I,J$, we write $A_{I,J}$ to denote the submatrix using the rows in $I$ and the columns in $J$. If $|I| = |J|$, we write
\[
\Delta_{I,J}(A) = \det(A_{I,J}).
\]

\subsection*{Acknowledgements}

I would like to thank my adviser Thomas Lam for his guidance and encouragement as I worked on this project. My thanks also go to David Speyer for a conversation that helped me get started, and to Jake Levinson for his comments on the introduction.

\section{Crystal combinatorics}
\label{sec_combo}

\subsection{Background on crystals}
\label{sec_cryst}

A crystal is a combinatorial skeleton of a representation of a Kac-Moody Lie algebra $\mathfrak{g}$. It arises as the $q \rightarrow 0$ limit of a special basis of a module for the quantized universal enveloping algebra $U_q(\mathfrak{g})$ (\cite{Kash}). Note that what we call a crystal is sometimes called a ``crystal basis'' or a ``crystal graph'' in the literature.

A type $A_{n-1}^{(1)}$ crystal consists of a set $B$, together with
\begin{itemize}
\item a weight map $\tw{\gamma} : B \rightarrow (\mathbb{Z}_{\geq 0})^n$;
\item for each $i \in \Zn$, functions $\tep_i, \tph_i : B \rightarrow \mathbb{Z}_{\geq 0}$;
\item for each $i \in \Zn$, crystal operators $\tw{e}_i,\tw{f}_i: B \rightarrow B \sqcup \{0\}$.
\end{itemize}
We say that $\tw{e}_i(b)$ is undefined if $\tw{e}_i(b) = 0$ (and similarly for $\tw{f}_i$). These maps must satisfy various properties; here we list those which are relevant in this paper:
\begin{itemize}
\item $\tw{e}_i(b)$ is defined if and only if $\tep_i(b) > 0$, and when $\tw{e}_i(b)$ is defined, $\tep_i(\tw{e}_i(b)) = \tep_i(b)-1$;
\item $\tw{f}_i(b)$ is defined if and only if $\tph_i(b) > 0$, and when $\tw{f}_i(b)$ is defined, $\tph_i(\tw{f}_i(b)) = \tph_i(b)-1$;
\item $\tw{e}_i$ and $\tw{f}_i$ are ``partial inverses,'' i.e., if $\tw{e}_i(b)$ is defined, then $\tw{f}_i(\tw{e}_i(b)) = b$, and if $\tw{f}_i(b)$ is defined, then $\tw{e}_i(\tw{f}_i(b)) = b$;
\item $\tph_i(b) - \tep_i(b) = \tw{\alpha}_i(\tg(b))$, where $\tw{\alpha}_i(a_1, \ldots, a_n) = a_i - a_{i+1}$;
\item If $\tw{e}_i(b)$ is defined, then $\tg(\tw{e}_i(b)) = \tg(b) + \tw{\alpha}_i^\vee(1)$, and if $\tw{f}_i(b)$ is defined, then $\tg(\tw{f}_i(b)) = \tg(b) + \tw{\alpha}_i^\vee(-1)$, where $\tw{\alpha_i}^\vee(m) = m(v_i - v_{i+1})$, with $v_i$ the $i$th standard basis vector.
\end{itemize}

A type $A_{n-1}$ crystal consists of the same data as a type $A_{n-1}^{(1)}$ crystal, but without the maps associated to $i=0$.

\subsection{Tableaux and crystals}
\label{sec_tableaux}

\subsubsection{Classical crystal structure}

Let $\lambda$ be a partition with at most $n$ rows. A {\em semistandard Young tableau (SSYT) of shape $\lambda$}, is a filling of the Young diagram of $\lambda$ with entries in $[n]$, such that the rows are weakly increasing, and the columns are strictly increasing. (We will often omit the adjectives ``semistandard Young.'') We write $B(\lambda)$ to denote the set of SSYTs of shape $\lambda$.

For each partition $\lambda$, there is an irreducible $\mathfrak{sl}_n$-representation whose basis is indexed by $B(\lambda)$, and a corresponding type $A_{n-1}$ crystal on the vertex set $B(\lambda)$. The weight map $\tg$ is the content of a tableau, i.e., $\tg(T) = (a_1, \ldots, a_n)$, where $a_i$ is the number of $i$'s in $T$. The maps $\tep_i, \tph_i,\tw{e}_i$, and $\tw{f}_i$ are computed by the following algorithm.

\begin{defn}
For $i \in [n-1]$, the maps $\tep_i, \tph_i, \tw{e}_i,$ and $\tw{f}_i$ are defined on $T \in B(\lambda)$ as follows. To begin, let $w$ be the (row) reading word of $T$, i.e., the word formed by concatenating the rows of $T$, starting with the bottom row. Now apply the following algorithm to $w$:
\begin{enumerate}
\item[Step 1:] Cross out all letters not equal to $i$ or $i+1$.
\item[Step 2:] For each consecutive pair of (non-crossed out) letters of the form $i+1, i$, cross out both letters. \item[Step 3:] Repeat Step 2 until there are no more pairs to cross out.
\item[Step 4:] Let $w'$ be the resulting subword, which is necessarily of the form
\[
w' = i^\alpha \; (i+1)^\beta.
\]
\end{enumerate}
The functions $\tep_i$ and $\tph_i$ are defined by
\[
\tep_i(T) = \beta \quad\quad\quad\quad \tph_i(T) = \alpha.
\]
If $\beta = 0$, then the crystal operator $\tw{e}_i(T)$ is undefined; if $\beta > 0$, then $\tw{e}_i(T)$ is the tableau of shape $\lp$ whose reading word is obtained from $w$ by changing the left-most $i+1$ in $w'$ into an $i$. Similarly, if $\alpha = 0$, then $\tw{f}_i(T)$ is undefined, and if $\alpha > 0$, then $\tw{f}_i(T)$ is the tableau of shape $\lp$ whose reading word is obtained from $w$ by changing the right-most $i$ in $w'$ into an $i+1$.
\end{defn}

\begin{ex}
\label{ex_cryst_ops}
Let $T = \ytableaushort{111222333,2333} \,.$ The subword of 2's and 3's in $w$ is
\[
2 \quad 3 \quad 3 \quad 3 \quad 2 \quad 2 \quad 2 \quad 3 \quad 3 \quad 3
\]
and after (recursively) crossing out consecutive pairs of the form $3 \quad 2$, we are left with
\[
w' = 2 \quad \cancel{3} \quad \cancel{3} \quad \cancel{3} \quad \cancel{2} \quad \cancel{2} \quad \cancel{2} \quad 3 \quad 3 \quad 3 \quad = \quad 2 \quad 3 \quad 3 \quad 3.
\]
Thus, we have $\tep_2(T) = 3$, $\tph_2(T) = 1$, and
\[
\tw{e}_2(T) = \ytableaushort{111222{*(red)2}33,2333} \quad\quad \tw{f}_2(T) = \ytableaushort{111222333,{*(red)3}333} \,.
\]
\end{ex}

\subsubsection{Promotion and evacuation}
\label{sec_prom_evac}

For $T \in B(\lambda)$ and $i \in [n-1]$, define the {\em $i$th Bender-Knuth involution} $\tsi_i(T)$ to be the tableau obtained by applying the following algorithm to each row of $T$:

In the given row, suppose there are $\alpha$ boxes containing $i$ which are not directly above a box containing $i+1$, and $\beta$ boxes containing $i+1$ which are not directly below a box containing $i$. Thus, this row contains a consecutive subword of the form $i^\alpha \quad (i+1)^\beta$. Replace this subword with $i^\beta \quad (i+1)^\alpha$.

It's clear that $\tsi_i$ is an involution, and that it interchanges the numbers of $i$'s and $i+1$'s in $T$.

\begin{defn}
\label{defn_pr_S}
Define {\em promotion} by $\tpr = \tsi_1 \tsi_2 \cdots \tsi_{n-1}$. Define the {\em Sch\"{u}tzenberger involution} (also known as {\em evacuation}) by $\tS = (\tsi_{n-1}) (\tsi_{n-2} \tsi_{n-1}) \cdots (\tsi_2 \cdots \tsi_{n-1}) (\tsi_1 \cdots \tsi_{n-1})$. 
\end{defn}

\begin{remark}
It is well-known that promotion as defined here is equivalent to the following algorithm based on Sch\"{u}tzenberger's jeu-de-taquin: remove the $n$'s; slide the remaining entries outward (start by sliding into the left-most hole); fill the vacated boxes with 0; increase all entries by 1.
\end{remark}

\begin{ex}
\label{ex_pr}
If $T$ is the tableau in Example \ref{ex_cryst_ops}, then
\[
\tsi_2(T) = \ytableaushort{111222233,2233}
\]
and if $n=3$, we have
\[
\tw{\pr}(T) = \tsi_1(\tsi_2(T)) = \ytableaushort{111111233,2233} \,.
\]
\end{ex}

\subsubsection{Affine crystal structure on rectangular tableaux}

For $k \in [n-1]$ and $L \geq 0$, define $B^{k,L} := B(L^k)$, the set of SSYTs (with entries in $[n]$) whose shape is the $k \times L$ rectangle. (By convention, $B^{k,0}$ consists of a single ``empty tableau.'') The type $A_{n-1}$ crystal structure on $B^{k,L}$ can be extended to a type $A_{n-1}^{(1)}$ crystal structure. This affine crystal corresponds to a {\em Kirillov-Reshetikhin module}, a finite-dimensional representation of the quantum affine algebra of type $A_{n-1}^{(1)}$. The maps associated to $i=0$ can be computed using promotion.

\begin{defn}
On $B^{k,L}$, define
\begin{align*}
\tep_0 &= \tep_1 \circ \tw{\pr} & \tph_0 &= \tph_1 \circ \tw{\pr} \\
\tw{e}_0 &= \tw{\pr}^{-1} \circ \tw{e}_1 \circ \tw{\pr} & \tw{f}_0 &= \tw{\pr}^{-1} \circ \tw{f}_1 \circ \tw{\pr}
\end{align*}
where we set $\tw{e}_0(T) = 0$ if $\tw{e}_1 \circ \tw{\pr}(T) = 0$ (equivalently, if $\tep_0(T) = 0$), and $\tw{f}_0(T) = 0$ if $\tw{f}_1 \circ \tw{\pr}(T) = 0$ (equivalently, if $\tph_0(T) = 0$).
\end{defn}

\begin{prop}
\label{prop_cryst_ops_pr}
We have the following identities of maps on $B^{k,L}$:
\begin{enumerate}
\item $\tpr^n = \Id$;
\item $\tw{\gamma} \circ \tpr = \tsh \circ \, \tw{\gamma}$, where $\tsh(a_1, \ldots, a_n) = (a_n, a_1, \ldots, a_{n-1})$;
\item $\tep_i = \tep_{i+1} \circ \tpr$ and $\tph_i = \tph_{i+1} \circ \tpr$ for $i \in \Zn$;
\item $\te_i = \tpr^{-1} \circ \, \te_{i+1} \circ \tpr$ and $\tf_i = \tpr^{-1} \circ \tf_{i+1} \circ \tpr$ for $i \in \Zn$.
\end{enumerate}
\end{prop}

Part (1) is well-known (see, e.g., \cite{Shim}, \cite{Rho}). Part (2) is clear. Parts (3) and (4) are due to Shimozono \cite[\S 3.3]{Shim}.

\subsection{Piecewise-linear translation}
\label{sec_GT}

We now translate the combinatorial maps on tableaux from the previous section into piecewise-linear maps on arrays of integers subject to certain inequalities. In general, we will use the same notation for a combinatorial map and its piecewise-linear translation, and we'll rely on context to determine which is meant.

\subsubsection{Gelfand-Tsetlin patterns}

A {\em Gelfand-Tsetlin pattern} (GT pattern) is a triangular array of nonnegative integers $(A_{ij})_{1 \leq i \leq j \leq n}$ satisfying the inequalities
\begin{equation}
\label{eq_GT}
A_{i,j+1} \geq A_{ij} \geq A_{i+1,j+1}
\end{equation}
for $1 \leq i \leq j \leq n-1$. Gelfand-Tsetlin patterns can be represented pictorially as triangular arrays, where the $j$th row in the triangle lists the numbers $A_{ij}$ for $i \leq j$. For example, if $n = 3$, then a Gelfand-Tsetlin pattern looks like:
\[
\begin{array}{ccccccccccc}
&&& A_{11} & \\
&& A_{12} && A_{22}\\
&A_{13} && A_{23} && A_{33} \\
\end{array}
\]

There is a simple bijection between Gelfand-Tsetlin patterns and SSYTs with entries in $[n]$. Given a Gelfand-Tsetlin pattern $(A_{ij})$, the associated tableau $T$ is described as follows: the number of $j$'s in the $i$th row of $T$ is $A_{ij} - A_{i,j-1}$ (we use the convention that $A_{i,i-1} = 0$). Equivalently, the $j$th row of the pattern is the shape of $T_{\leq j}$, the part of $T$ obtained by removing numbers larger than $j$. In particular, the $n$th row of the pattern is the shape of $T$. Here is an example of a Gelfand-Tsetlin pattern, and the corresponding SSYT.
\begin{equation}
\label{eq_GT}
\begin{array}{ccccccccccc}
&&&& 2 & \\
&&& 4 && 2\\
&&6 && 3 && 1 \\
&6 && 6 && 1 && 0 \\
6 && 6 && 6 && 0 && 0
\end{array}
\quad\quad
\longleftrightarrow \quad\quad
\ytableaushort{{1}{1}{2}{2}{3}{3}, {2}{2}{3}{4}{4}{4}, {3}{5}{5}{5}{5}{5}}
\end{equation}

Many operations on tableaux are given by piecewise-linear formulas in the entries of the corresponding Gelfand-Tsetlin pattern. In general, if $\tw{f}$ is a function on tableaux, then we write $\tw{f}(A_{ij}) := \tw{f}(T)$, where $T$ is the tableau corresponding to $(A_{ij})$. If $\tw{f}$ is a map from one tableau to another, we write $\tw{f}(A_{ij}) = (A'_{ij})$, where $(A'_{ij})$ corresponds to $\tw{f}(T)$. Here is a simple example.


\begin{ex}
\label{ex_e1}
Here we describe how the maps $\tep_1, \tph_1, \te_1,$ and $\tf_1$ act on Gelfand-Tsetlin patterns. Let $(A_{ij})$ be a Gelfand-Tsetlin pattern with corresponding tableau $T$. When we apply the algorithm in the previous section, all 2's in the second row of $T$ ``cancel'' with 1's in the first row, so we have $$w' = 1^{A_{11} - A_{22}} \quad 2^{A_{12} - A_{11}}.$$ Thus, $\tep_1(A_{ij}) = A_{12} - A_{11}$, and when $\tep_1(A_{ij}) > 0$, $\te_1(A_{ij})$ is obtained by increasing $A_{11}$ by 1, and leaving the other entries unchanged. Similarly, $\tph_1(A_{ij}) = A_{11} - A_{22}$, and $\tf_1(A_{ij})$ is obtained by decreasing $A_{11}$ by 1 (if the result is still a GT pattern).
\end{ex}

There is also a piecewise-linear formula for the Bender-Knuth involutions.

\begin{lem}[Kirillov-Berenstein \cite{KirBer}]
\label{lem_PLBK}
Let $(A_{ij})$ be a Gelfand-Tsetlin pattern. For $r \in [n-1]$, we have $\tsi_r(A_{ij}) = (A'_{ij})$, where
\begin{equation}
\label{eq_PLBK}
A'_{ij} = \begin{cases}
\min(A_{i-1,r-1}, A_{i,r+1}) + \max(A_{i,r-1}, A_{i+1,r+1}) - A_{ir} & \text{ if } j = r \\
A_{ij} &\text{ if } j \neq r
\end{cases} 
\end{equation}
and we use the convention that $A_{0,j} = \infty$ and $A_{i,i-1} = 0$.
\end{lem}

Note that $\tsi_r$ changes only the $r$th row of the Gelfand-Tsetlin pattern, and for each $i$, $\sigma_r(A_{ir})$ depends only on $A_{ir}$ and the four entries diagonally adjacent to $A_{ir}$ in the Gelfand-Tsetlin pattern (some of which may be ``missing'' if $A_{ir}$ is on the upper boundary of the triangle):
\[
\begin{array}{ccccccccccc}
 A_{i-1,r-1} && A_{i,r-1}\\
& A_{ir}\\
A_{i,r+1} && A_{i+1,r+1} \\
\end{array}.
\]

Combining Lemma \ref{lem_PLBK} and Definition \ref{defn_pr_S}, we obtain recursive piecewise-linear descriptions of $\tpr$ and $\tS$.

\subsubsection{$k$-rectangles}
\label{sec_k_rectangles}

Gelfand-Tsetlin patterns parametrize SSYTs of arbitrary shape. Here we consider the restriction of this parametrization to the case of rectangular tableaux.

For $k \in [n-1]$, set
\begin{equation}
\label{eq_R_k}
R_k = \{(i,j) \: | \: 1 \leq i \leq k, \: i \leq j \leq i+n-k-1\},
\end{equation}
and define $\tw{\bT}_k = \mathbb{Z}^{R_k} \times \mathbb{Z} \cong \mathbb{Z}^{k(n-k)+1}$. We will denote a point of $\tw{\bT}_k$ by $b = (B_{ij},L)$, where $(i,j)$ runs over $R_k$.

Given $(B_{ij}, L) \in \tw{\bT}_k$, define a triangular array $(A_{ij})_{1 \leq i \leq j \leq n}$ by
\[
A_{ij} = \begin{cases}
B_{ij} & \text{ if } (i,j) \in R_k \\
L & \text{ if } j > i+n-k-1 \\
0 & \text{ if } j < i.
\end{cases}
\]
\begin{defn}
\label{defn_k_rect}
Define $B^k$ to be the set of $(B_{ij},L) \in \tw{\bT}_k$ such that $(A_{ij})$ is a Gelfand-Tsetlin pattern. We call an element of $B^k$ a {\em $k$-rectangle}, and we say that $(A_{ij})$ is the {\em associated Gelfand-Tsetlin pattern}.
\end{defn}

For example, if $n = 5$ and $k = 3$, then we may pictorially represent a 3-rectangle $(B_{ij},L)$ and its associated GT pattern as follows:
\begin{equation}
\label{eq_rect_GT}
\left(
\begin{array}{ccccccccccc}
& B_{11} & \\
B_{12} && B_{22}\\
& B_{23} && B_{33} \\
&& B_{34}
\end{array}
\; , \quad L \right)
\quad\quad \longleftrightarrow
\begin{array}{ccccccccccc}
&&&& B_{11} & \\
&&& B_{12} && B_{22}\\
&&L && B_{23} && B_{33} \\
&L && L && B_{34} && 0 \\
L && L && L && 0 && 0
\end{array}
\end{equation}

As \eqref{eq_GT} and \eqref{eq_rect_GT} illustrate, the bijection between GT patterns and SSYTs restricts to a bijection between $k$-rectangles and rectangular SSYTs with $k$ rows, with the parameter $L$ giving the number of columns in the tableau. Thus, we identify
\[
B^k = \bigsqcup_{L=0}^\infty B^{k,L}.
\]
We view the integers $(B_{ij},L)$ as coordinates on the set of $k$-row rectangular SSYTs. Sometimes it will be more convenient to work with the following alternative set of coordinates. For $1 \leq i \leq k$ and $i \leq j \leq i+n-k$, define
\begin{equation}
\label{eq_row_coords}
b_{ij} = B_{ij} - B_{i,j-1},
\end{equation}
where we use the convention that $B_{i,i-1} = 0$ and $B_{i,i+n-k} = L$ for all $i$. Thus, $b_{ij}$ is the number of $j$'s in the $i$th row of the $k$-row rectangular SSYT corresponding to $b = (B_{ij},L)$.

\subsubsection{Symmetries of $k$-rectangles}
\label{sec_comb_symmetries}

Throughout this section, fix $k \in [n-1]$ and $L \geq 0$.

\begin{defn}
Define {\em rotation} $\trot : B^k \rightarrow B^k$ by $\trot(B_{ij},L) = (B'_{ij},L)$, where
\[
B'_{ij} = L - B_{k-i+1,n-j}.
\]
Define {\em reflection} $\trefl : B^k \rightarrow B^{n-k}$ by $\trot(B_{ij},L) = (B''_{ij},L)$, where
\[
B''_{ij} = L - B_{j-i+1,j}.
\]
The first map rotates the rectangular Gelfand-Tsetlin pattern 180 degrees, and then replaces each entry $a$ with $L - a$; the second map reflects the rectangular Gelfand-Tsetlin pattern over a vertical axis, and then replaces each entry $a$ with $L - a$.
\end{defn}

The operations $\trot$ and $\trefl$ have simple effects on the corresponding rectangular tableaux.

\begin{lem}
\label{lem_symms_tab}
Suppose $b = (B_{ij},L) \in B^k$ and let $T,U,V$ be the SSYTs corresponding to $b, \trot(b), \trefl(b)$, respectively. Then
\begin{enumerate}
\item $U$ is obtained by rotating $T$ 180 degrees and replacing each entry $i$ with $n-i+1$.
\item $V$ is obtained by replacing each column of $T$ with the complement in $[n]$ of the entries in that column (arranged in increasing order), and then reversing the order of the columns.
\end{enumerate}
\end{lem}

\begin{proof}
Part (1) is straightforward. To prove (2), first consider the case $L = 1$. In this case, the tableau corresponding to $b$ is a single column of length $k$, or in other words, a subset $S = \{s_1 < \cdots < s_k\} \subset [n]$. We must show that if $b$ corresponds to $S$, then $\trefl(b)$ corresponds to $[n] \setminus S$.

Identify the $k$-rectangle $b = (B_{ij},1)$ with a partition $\lp$ inside the $k \times (n-k)$ rectangle by setting $\lp_i = | \{j \, | \, B_{ij} = 1\} |$ for $i = 1, \ldots, k$. The entries $s_1 < \cdots < s_k$ of the corresponding tableau are related to $\lp$ by
\[
\lp_i = i+n-k-s_i.
\]
Equivalently, $s_i$ is the position of the $i^{th}$ vertical step in $p_\lp$, the lattice path from the top-right corner of the $k \times (n-k)$ rectangle to the bottom-left corner which traces out the lower boundary of the Young diagram of $\lp$.

Now identify the $(n-k)$-rectangle $\tw{\refl}(b)$ with a partition $\tw{\lp}$ inside the $(n-k) \times k$ rectangle in the same manner. From the definition of $\tw{\refl}$, one sees that the positions of the vertical steps in $p_\lp$ are precisely the positions of the horizontal steps in $p_{\tw{\lp}}$, so $\tw{\lp}$ corresponds to the $(n-k)$-subset $[n] \setminus S$, as claimed. (See Figure \ref{fig_refl} for an example.)

\begin{figure}

\begin{center}
\begin{tikzpicture}[scale=0.5]

\draw (1.4,0) node {$b =$};

\draw (4.4,2) node {1};
\draw (5.2,1) node {0};
\draw (6,0) node {0};
\draw (3.6,1) node {1};
\draw (2.8,0) node {1};
\draw (4.4,0) node {1};
\draw (3.6,-1) node {1};
\draw (5.2,-1) node {1};
\draw (4.4,-2) node {1};
\draw (6.8,-1) node {0};
\draw (6,-2) node {0};
\draw (5.2,-3) node {1};

\pgfmathtruncatemacro{\x}{12};

\draw (12,0) node {$\trefl(b) =$};

\draw (4.4+\x,2) node {0};
\draw (5.2+\x,1) node {0};
\draw (6+\x,0) node {0};
\draw (3.6+\x,1) node {1};
\draw (2.8+\x,0) node {1};
\draw (4.4+\x,0) node {0};
\draw (3.6+\x,-1) node {0};
\draw (5.2+\x,-1) node {0};
\draw (4.4+\x,-2) node {0};
\draw (2+\x,-1) node {1};
\draw (2.8+\x,-2) node {1};
\draw (3.6+\x,-3) node {0};

\end{tikzpicture}
\end{center}

\caption{An example of $\trefl$ in the $L = 1$ case (with $n=7, k=4$). Here $b$ corresponds to the partition $(3,2,2,1)$ and the subset $\{1,3,4,6\}$; $\trefl(b)$ corresponds to $(3,1,0)$ and $\{2,5,7\}$.}
\label{fig_refl}
\end{figure}

Now suppose $L > 1$. The rectangle $b$ is equal to the entry-wise sum of the rectangles corresponding to the individual columns of $T$, and the same is true of $\trefl(b)$ and its corresponding tableau $V$. Let $V'$ be the array obtained by replacing each column of $T$ by its complement in $[n]$, and reversing the order of the columns. Using the $L = 1$ case, we see that $\trefl(b)$ is also equal to the entry-wise sum of the rectangles corresponding to the individual columns of $V'$. To conclude that $V = V'$, it remains to show that $V'$ is semistandard, i.e., that its rows are weakly increasing.

Let $S$ and $S'$ be the subsets of entries in two consecutive columns of $V'$. The condition for $V'$ to be semistandard is that $s_i \leq s'_i$ for $i = 1, \ldots, n-k$. If this condition holds, write $S \preccurlyeq S'$. Let $\lp, \lp'$ be the partitions associated to $S,S'$, respectively. From the proof of the $L = 1$ case, one sees that
\[
S \preccurlyeq S' \; \Longleftrightarrow \; \lp \supseteq \lp' \; \Longleftrightarrow \; \tw{\lp}' \supseteq \tw{\lp} \; \Longleftrightarrow \; [n] \setminus S' \preccurlyeq [n] \setminus S,
\]
where $\subseteq$ denotes inclusion of Young diagrams. Thus, $V'$ is semistandard because $T$ is semistandard.
\end{proof}


\begin{prop}
\label{prop_symms_cryst_ops}
We have the following identities of maps on rectangular tableaux:
\begin{enumerate}
\item $\trot = \tS$ (the Sch\"{u}tzenberger involution);
\item $\te_i = \trot \circ \tf_{n-i} \circ \trot$ and $\tf_i = \trot \circ \, \te_{n-i} \circ \trot$ for $i \in \Zn$;
\item $\te_i = \trefl \circ \tf_i \circ \trefl$ and $\tf_i = \trefl \circ \, \te_i \circ \trefl$ for $i \in \Zn$.
\end{enumerate}
\end{prop}

Part (1) follows from the theory of the plactic monoid (see, e.g., \cite{Ful}). We will prove parts (2) and (3) using geometric crystals in \S \ref{sec_symms_trop} (see Remark \ref{rmk_symms_cryst_ops}). Part (2) is known (see \cite[\S 3]{LLT}), but part (3) does not seem to appear in the literature.

\section{Geometric crystals}
\label{sec_geom_cryst}

\subsection{Pl\"{u}cker coordinates}
\label{sec_pluc_coords}

We write $\Gr(k,n)$ to denote the Grassmannian of $k$-dimensional subspaces in $\mathbb{C}^n$. We consider the Grassmannian in its Pl\"{u}cker embedding, and for $J \in {[n] \choose k}$, we write $P_J(M)$ for the $J^{th}$ Pl\"{u}cker coordinate of the subspace $M$. We will represent a point $M \in \Gr(k,n)$ as the row span of a (full-rank) $n \times k$ matrix $M^\circ$, so that $P_J(M)$ is the maximal minor of $M^\circ$ using the rows in $J$. When there is no danger of confusion, we treat a subspace and its matrix representatives interchangeably. For example, we may speak of the Pl\"{u}cker coordinates of a full-rank $n \times k$ matrix (these are only defined up to a common scalar multiple, of course).

There is a natural (left) action of $\GL_n = \GL_n(\mathbb{C})$ on $\Gr(k,n)$ given by matrix multiplication. We denote the action of $A \in \GL_n$ on $M \in \Gr(k,n)$ by $(A,M) \mapsto A \cdot M$; this is the subspace spanned by the columns of $A \cdot M^\circ$, where $M^\circ$ is an $n \times k$ matrix representative of $M$.

To reduce the number of special cases needed in various arguments, we make the following convention.

\begin{conv}
\label{conv_pluc}
Let $M$ be a full-rank $n \times k$ matrix representing a point in $\Gr(k,n)$.
\begin{enumerate}
\item
Unless otherwise indicated (see convention (4) below), we label Pl\"{u}cker coordinates of $M$ by \underline{sets}, not by ordered lists. That is, if $I \in {[n] \choose k}$, then $P_I(M)$ means the determinant of the $k \times k$ submatrix of $M$ using the rows indexed by the elements of $I$, taken in the order in which they appear in $M$. Thus, $P_{\{1,2\}}(M) = P_{\{2,1\}}(M)$. We will often write $P_{12}(M)$ or $P_{1,2}(M)$ instead of $P_{\{1,2\}}(M)$.
\item
If $I \subset [n]$ does not contain exactly $k$ elements, then we set $P_I(M) = 0$.
\item
If $I$ is any set of integers, we set $P_I(M) = P_{I'}(M)$, where $I'$ is the set consisting of the residues of the elements of $I$ modulo $n$, where we take the residues to lie in $[n]$.
\item
We use the notation $P_{<i_1, \ldots, i_k>}(M)$ for the determinant of the $k \times k$ matrix whose $j^{th}$ row is row $i_j$ of $M$. We will only use this notation when $i_1, \ldots, i_k$ are (not necessarily distinct) elements of $[n]$. Note that $P_{<1,2>}(M) = -P_{<2,1>}(M) = P_{12}(M)$.
\end{enumerate}
\end{conv}

The following classical result plays an important role in this article. For a proof, see, e.g., \cite{Ful}.

\begin{prop}[Grassmann-Pl\"{u}cker relations]
\label{prop_Gr_Pl_rel}
Suppose $i_1, \ldots, i_k, i_{k+1}, j_1, \ldots, j_{k-1} \in [n]$. Then for $M \in \Gr(k,n)$, we have
\begin{equation}
\label{eq_Gr_Pl_rel}
\sum_{r = 1}^{k+1} (-1)^r P_{<i_1, \ldots, i_{r-1}, i_{r+1}, \ldots, i_{k+1}>}(M) P_{<i_r, j_1, \ldots, j_{k-1}>}(M) = 0.
\end{equation}
\end{prop}

\begin{cor}[Three-term Pl\"{u}cker relation]
\label{cor_3term}
Fix $k \geq 2$. If $I \in {[n] \choose k-2}$ and $a,b,c,d$ are elements of $[n]$ satisfying $a \leq b \leq c \leq d$, then for $M \in \Gr(k,n)$, we have
\begin{equation}
\label{eq_3term}
P_{I \cup \{a,b\}}(M) P_{I \cup \{c,d\}}(M) + P_{I \cup \{a,d\}}(M) P_{I \cup \{b,c\}}(M) = P_{I \cup \{a,c\}}(M) P_{I \cup \{b,d\}}(M).
\end{equation}
\end{cor}

Note that the subscripts in \eqref{eq_Gr_Pl_rel} are ordered lists, whereas the subscripts in \eqref{eq_3term} are sets.

\subsection{Geometric crystal on the Grassmannian}
\label{sec_geom_Gr}

Before introducing our main object of study, we recall the definition of a type $A_{n-1}^{(1)}$ decorated geometric crystal (\cite{BKII}, \cite{Nak}).

\begin{defn}
\label{defn_geom_precryst}
A {\em geometric pre-crystal} (of type $A_{n-1}^{(1)}$) consists of
\begin{itemize}
\item an irreducible complex algebraic variety X
\item a rational map $\gamma : X \rightarrow (\Cx)^n$
\item for each $i \in \Zn$, two rational functions $\vp_i, \ve_i : X \rightarrow \Cx$ which are not identically zero.\footnote{In \cite{BKII}, some of the $\vp_i$ and $\ve_i$ are allowed to be zero, but we will not need this more general setting.}
\item for each $i \in \Zn$, a rational unital\footnote{This means that $e_i(1,x)$ is defined (and thus equal to $x$) for all $x \in X$.} action $e_i: \Cx \times X \rightarrow X$. We will usually denote the image $e_i(c,x)$ by $e_i^c(x)$.

\end{itemize}
These data must satisfy the following properties:
\begin{enumerate}
\item $\ve_i(x) = \vp_i(x)\alpha_i(\gamma(x))$, where
\[
\alpha_i(z_1, \ldots, z_n) = \dfrac{z_i}{z_{i+1}}
\]
\item $\gamma(e_i^c(x)) = \alpha_i^\vee(c)\gamma(x)$, where
\[
\alpha_i^\vee(c) = (1, \ldots, c, c^{-1}, \ldots, 1)
\]
with $c$ in the $i$th component and $c^{-1}$ in the $(i+1)$st component (mod $n$).
\item $\vp_i(e_i^c(x)) = c^{-1}\vp_i(x)$ and $\ve_i(e_i^c(x)) = c\ve_i(x)$.
\end{enumerate}
\end{defn}

\begin{defn}
\label{defn_geom_cryst}
A {\em geometric crystal} (of type $A_{n-1}^{(1)}$) is a geometric pre-crystal which satisfies the following ``geometric Serre relations'':

If $n \geq 3$, then for each pair $i,j \in \mathbb{Z}/n\mathbb{Z}$, and $c_1, c_2 \in \Cx$, the actions $e_i, e_j$ satisfy
\begin{equation*}
\begin{array}{ll}
\label{rels}
e_i^{c_1}e_j^{c_2} = e_j^{c_2}e_i^{c_1} & \text{if } | i - j | > 1 \\
e_i^{c_1}e_j^{c_1c_2}e_i^{c_2} = e_j^{c_2}e_i^{c_1c_2}e_j^{c_1} & \text{if } | i - j | = 1.
\end{array}
\end{equation*}
If $n = 2$, there is no Serre relation for $e_0$ and $e_1$, so a geometric pre-crystal is automatically a geometric crystal.
\end{defn}

\begin{defn}
A {\em decorated geometric crystal} (of type $A_{n-1}^{(1)}$) is a geometric crystal $(X,\gamma,\vp_i, \ve_i, e_i)$ equipped with a rational function $f : X \rightarrow \mathbb{C^\times}$ such that
\begin{equation}
\label{eq_f_prop}
f(e_i^c(x)) = f(x) + \frac{c-1}{\vp_i(x)} + \frac{c^{-1} - 1}{\ve_i(x)}
\end{equation}
for all $x \in X$ and $i \in \Zn$.
The function $f$ is called a {\em decoration}.
\end{defn}

For $k \in [n-1]$, let $\X{k}$ denote the variety $\Y{k}{n}$. The central object in this paper is a type $A_{n-1}^{(1)}$ decorated geometric crystal on $\X{k}$, which is defined as follows.

\begin{defn}
\label{defn_maps}
\
\begin{enumerate}
\item Define a rational map $\gamma : \X{k} \rightarrow (\Cx)^n$ by $\gamma(M,t) = (\gamma_1, \ldots, \gamma_n)$, where
\begin{equation*}
\gamma_i =
\begin{cases}
\dfrac{P_{[i-k+1,i]}(M)}{P_{[i-k,i-1]}(M)} \text{ if } 1 \leq i \leq k \bigskip \\
t \dfrac{P_{[i-k+1,i]}(M)}{P_{[i-k,i-1]}(M)} \text{ if } k+1 \leq i \leq n.
\end{cases}
\end{equation*}

\item For $i \in \Zn$, define rational functions $\vp_i, \ve_i : \X{k} \rightarrow \Cx$ by
\[
\vp_i(M,t) =
t^{-\delta_{i,0}} \dfrac{P_{[i-k+1,i-1] \cup \{i+1\}}(M)}{P_{[i-k+1,i]}(M)},
\]
\[
\ve_i(M,t) = t^{-\delta_{i,k}} \dfrac{P_{[i-k+1,i-1] \cup \{i+1\}}(M) P_{[i-k+1,i]}(M)}{P_{[i-k,i-1]}(M) P_{[i-k+2,i+1]}(M)}.
\]

\item For $i \in \Zn$, define a rational action $e_i^c : \X{k} \rightarrow \X{k}$ by $e_i^c(M,t) = (M',t)$, where
\[
M' = \begin{cases}
x_i\left(\dfrac{c-1}{\vp_i(M,t)}\right) \cdot M & \text{ if } i \neq 0 \smallskip \\
x_0\left(\dfrac{(-1)^{k-1}}{t} \cdot \dfrac{c-1}{\vp_0(M,t)}\right) \cdot M & \text{ if } i = 0.
\end{cases}
\]
Here $x_i(a) = \Id + aE_{i,i+1}$ for $i \in [n-1]$, and $x_0(a) = \Id + aE_{n1}$, where $E_{ij}$ is an $n \times n$ matrix unit.

\item Define a rational function $f : \X{k} \rightarrow \Cx$ by
\begin{equation}
\label{eq_dec}
f(M,t) = \sum_{i \neq k} \frac{P_{\{i-k\} \cup [i-k+2,j]}(M)}{P_{[i-k+1,i]}(M)} + t \frac{P_{[2,k] \cup \{n\}}(M)}{P_{[1,k]}(M)}.
\end{equation}
\end{enumerate}
\end{defn}

\begin{thm}
\label{thm_is_geom_cryst}
The data $(\X{k}, \gamma, \vp_i, \ve_i, e_i,f)$ define a decorated geometric crystal of type $A_{n-1}^{(1)}$.
\end{thm}

This result is proved by means of unipotent crystals in \S \ref{sec_unip_Gr}.

\subsection{Cyclic shifting}
\label{sec_cyclic_shift}

\begin{defn}
\label{defn_PR}
Define a map $\PR : \X{k} \rightarrow \X{k}$ by $\PR(M,t) = (M',t)$, where $M'$ is obtained from $M$ by shifting the rows down by 1 (mod $n$), and multiplying the new first row by $(-1)^{k-1} t$.
In terms of Pl\"{u}cker coordinates, the map is defined by
\begin{equation}
\label{eq_PR}
P_J(\PR(M,t)) =
\begin{cases}
P_{J - 1}(M) & \text{ if } 1 \not \in J \\
t \cdot P_{J - 1}(M) & \text{ if } 1 \in J
\end{cases}
\end{equation}
where $J - 1$ is obtained from $J$ by subtracting 1 from each element (0 is identified with $n$).
\end{defn}

\noindent We will often write $\PR_t$ to denote the map $M \mapsto M'$. For example, we have
\[
\left(\begin{array}{cc}
z_{11} & z_{12} \\
z_{21} & z_{22} \\
z_{31} & z_{32} \\
z_{41} & z_{42}
\end{array}\right)
\quad
\overset{\PR_t} {\longrightarrow}
\quad
\left(\begin{array}{cc}
-t \cdot z_{41} & -t \cdot z_{42} \\
z_{11} & z_{12} \\
z_{21} & z_{22} \\
z_{31} & z_{32}
\end{array}\right).
\]

We now show that the map $\PR$ is ``compatible'' with the geometric crystal structure on $\X{k}$.

\begin{lem}
\label{lem_pr_u_action}
For $i \in \Zn$, $a \in \bbC$, and $(M,t) \in \X{k}$, we have
\begin{equation*}
\PR_t^{-1}(x_i(a) \cdot \PR_t(M)) = \begin{cases}
x_{i-1}(a) \cdot M &\text{ if } i \neq 0,1 \\
x_{0}(\frac{(-1)^{k-1} a}{t}) \cdot M &\text{ if } i = 1 \\
x_{n-1}((-1)^{k-1} t a) \cdot M &\text{ if } i = 0 \\
\end{cases}.
\end{equation*}
\end{lem}

\begin{proof}
Let $v_i \in \bbC^k$ be the $i$th row of the matrix $M$. Left-multiplying $M$ by $x_i(a)$ replaces $v_i$ with $av_{i+1} + v_i$ if $i \neq 0$, and it replaces $v_n$ with $av_1 + v_n$ if $i = 0$. By definition, the map $\PR_t$ replaces $v_i$ with $v_{i-1}$ for $i \neq 1$, and it replaces $v_1$ with $(-1)^{k-1} t v_n$; the inverse map $\PR_t^{-1}$ replaces $v_i$ with $v_{i+1}$ for $i \neq n$, and it replaces $v_n$ with $\frac{(-1)^{k-1}}{t} v_1$. From this description, it's clear that the lemma holds when $i \neq 0,1$.

If $i = 1$, then we have
\[
\left(\begin{array}{c}
v_1 \\
v_2 \\
\vdots \\
v_n
\end{array}\right)
\,
\overset{\PR_t} {\longrightarrow}
\,
\left(\begin{array}{cc}
(-1)^{k-1} t v_n \\
v_1 \\
\vdots \\
v_{n-1}
\end{array}\right)
\,
\overset{x_1(a)} {\longrightarrow}
\,
\left(\begin{array}{cc}
av_1 + (-1)^{k-1}t v_n \\
v_1 \\
\vdots \\
v_{n-1}
\end{array}\right)
\,
\overset{\PR_t^{-1}} {\longrightarrow}
\,
\left(\begin{array}{cc}
v_1 \\
v_2 \\
\vdots \\
\frac{(-1)^{k-1}a}{t} v_1 + v_n
\end{array}\right)
\]
so we see that $\PR_t^{-1}(x_1(a) \cdot \PR_t(M)) = x_0\left(\frac{(-1)^{k-1}a}{t}\right) \cdot M$. The case $i = 0$ is verified similarly.

\end{proof}

\begin{lem}
\label{lem_pr_conj}
\
\begin{enumerate}
\item For each $i \in \Zn$, we have $\vp_i \circ \PR = \vp_{i-1}$ and $\ve_i \circ \PR = \ve_{i-1}$.
\item For each $i \in \Zn$, we have $\PR^{-1} \circ \, e_i^c \circ \PR = e_{i-1}^c$.
\item If $\gamma(M,t) = (\gamma_1, \gamma_2, \ldots, \gamma_n)$, then $(\gamma \circ \PR)(M,t) = (\gamma_n, \gamma_1, \ldots, \gamma_{n-1})$.
\item We have $f \circ \PR = f$.
\end{enumerate}
\end{lem}

\begin{proof}
If $i \neq 0,1$, then
\[
\vp_i \circ \PR(M,t) = \dfrac{P_{[i-k+1,i-1] \cup \{i+1\}}(\PR_t(M))}{P_{[i-k+1,i]}(\PR_t(M))} = \dfrac{P_{[i-k,i-2] \cup \{i\}}(M)}{P_{[i-k,i-1]}(M)} = \vp_{i-1}(M,t),
\]
where the middle equality holds because the sets $[i-k+1,i-1] \cup \{i+1\}$ and $[i-k+1,i]$ differ only in the elements $i$ and $i+1$, so either both sets contain 1, or neither set contains 1. The cases $i = 0,1$ are similar. This proves the first half of (1); the other half is similar.

Part (2) follows easily from part (1) and Lemma \ref{lem_pr_u_action}. Parts (3) and (4) are clear from the definitions of $\PR$, $\gamma$, and $f$.
\end{proof}

\section{Parametrization}
\label{sec_param}

\subsection{Rational $k$-rectangles and the Gelfand-Tsetlin parametrization}
\label{sec_GT_param}

Recall that a $k$-rectangle is an array of $k(n-k)+1$ nonnegative integers satisfying certain inequalities (Definition \ref{defn_k_rect}). We now define a ``rational version'' of $k$-rectangles, in which integers are replaced with nonzero complex numbers. Let
\[
\bT_k = (\Cx)^{R_k} \times \Cx
\]
where $R_k$ is the indexing set defined by \eqref{eq_R_k}. Denote a point of $\bT_k$ by $(X_{ij}, t)$, where $(i,j)$ runs over $R_k$. We call $(X_{ij},t)$ a {\em rational $k$-rectangle}. Set
\begin{equation}
\label{eq_x_ij}
x_{ij} = X_{ij}/X_{i,j-1}
\end{equation}
for $1 \leq i \leq k$ and $i \leq j \leq i+n-k$, where we use the convention $X_{i,i-1} = 1$ and $X_{i,i+n-k} = t$. The quantity $x_{ij}$ is the rational analogue of the number of $j$'s in the $i^{th}$ row of a tableau (c.f. \eqref{eq_row_coords}). Note that there are no inequality conditions on rational $k$-rectangles; we will see in \S \ref{sec_trop_dec} that the decoration is a ``rational proxy'' for the inequalities.

We will use the set of rational $(n-k)$-rectangles to parametrize the variety $\X{k} = \Y{k}{n}$; that is, we will define a birational isomorphism
\[
\Theta_k : \bT_{n-k} \rightarrow \X{k}.
\]

For $1 \leq i \leq j \leq n$ and $c_i, \ldots, c_j \in \Cx$, define
\begin{equation*}
M_{[i,j]}(c_i, \ldots, c_j) = x_{-i}(c_i) x_{-(i+1)}(c_{i+1}) \cdots x_{-(j-1)}(c_{j-1}) t_{j}(c_j),
\end{equation*}
where
\[
x_{-i}(c) = cE_{ii} + c^{-1}E_{i+1,i+1} + E_{i+1,i} + \sum_{j \neq i,i+1}E_{jj}, \quad\quad\quad\quad t_i(c) = cE_{ii} + \sum_{j \neq i}E_{jj}
\]
for $i \in [n-1]$, $i \in [n]$, respectively ($E_{ij}$ is an $n \times n$ matrix unit). For example, if $n = 5$, then we have
\begin{equation}
\label{eq_chev_gens}
M_{[2,4]}(c_2,c_3,c_4) = \left(
\begin{array}{ccccc}
1 &  &  &  &  \\
 & c_2 &  &  &  \\
 & 1 & c_3/c_2 &  &  \\
 &  & 1 & c_4/c_3 &  \\
 &  &  &  & 1
\end{array}
\right)
\end{equation}
where only nonzero entries are shown.

\begin{defn}
\label{defn_Phi}
\
\begin{enumerate}
\item
Define $\Phi_{n-k} : \bT_{n-k} \rightarrow \GL_n$ by
\[
\Phi_{n-k}(X_{ij}, t) = \prod_{i = n-k}^1 M_{[i,i+n-k-1]}(X_{ii}, X_{i,i+1}, \ldots, X_{i,i+n-k-1},t),
\]
where the terms in the product are arranged from left to right in decreasing order of $i$.

\item
Define $\Theta_k : \bT_{n-k} \rightarrow \X{k}$ by $\Theta_k(X_{ij}, t) = (M, t)$, where
\[
M = \pi^k \circ \Phi_{n-k}(X_{ij}, t)
\]
is the ``projection'' of the invertible matrix $\Phi_{n-k}(X_{ij},t)$ onto the subspace spanned by its first $k$ columns. We call $\Theta_k$ the {\em Gelfand-Tsetlin parametrization} of $\X{k}$.
\end{enumerate}
\end{defn}

\begin{ex}
\label{ex_2_5}
Suppose $n=5$ and $k=2$. For $(X_{ij},t) \in \bT_{3}$, let $(M,t) = \Theta_2(X_{ij},t) = \Y{2}{5}$. Then $M$ is spanned by the first two columns of the matrix
\begin{equation}
\label{eq_2_5}
\Phi_3(X_{ij}, t) = \left(
\begin{array}{ccccc}
x_{11} & 0 & 0 & 0 & 0 \\
x_{22} & x_{12}x_{22} & 0 & 0 & 0 \\
x_{33} & (x_{12} + x_{23})x_{33} & x_{13}x_{23}x_{33} & 0 & 0 \\
1 & x_{12} + x_{23} + x_{34} & x_{13}(x_{23} + x_{34}) & x_{24}x_{34} & 0 \\
0 & 1 & x_{13} & x_{24} & x_{35}
\end{array}
\right)
\end{equation}
(recall that $x_{ij} = X_{ij}/X_{i,j-1}$).
\end{ex}

\begin{prop}
\label{prop_GT_pl}
The map $\Theta_k$ is an open embedding of $\bT_{n-k}$ into $\X{k}$. The (rational) inverse is given by $(M,t) \mapsto (X_{ij}, t)$, where
\[
X_{ij} = \dfrac{P_{[i,j] \cup [n-k+j-i+2,n]}(M)}{P_{[i+1,j] \cup [n-k+j-i+1,n]}(M)}
\]
for $1 \leq i \leq n-k$ and $i \leq j \leq i+k-1$.
\end{prop}

This result is proved in \S \ref{sec_basic_pluc}.



\begin{remark}
\label{rmk_ZLP}
Berenstein and Kazhdan defined a geometric crystal on $UZ(L_P)\ov{w_P}U \cap B^-$ for any parabolic subgroup $P$ of a reductive group $G$, where $U$ is the upper unipotent subgroup, $B^-$ is the lower Borel subgroup, $Z(L_P)$ is the centralizer of the Levi subgroup $L_P$, and $\ov{w_P}$ is the ``standard representative'' of a certain Weyl group element $w_P$. The map $\Phi_{n-k}$ comes from the special case of their construction where $P$ is a maximal parabolic subgroup of $PGL_n$, and $w_P$ is the Grassmannian permutation
\[
w_P = \left(
\begin{array}{ccccccccc}
1 & \cdots & k & k+1 & \cdots & n \\
n-k+1 & \cdots & n & 1 & \cdots & n-k
\end{array}
\right)
\]
(\cite[\S 3.1]{BKII}).
\end{remark}

\subsection{Networks and formulas for Pl\"{u}cker coordinates}
\label{sec_net}

\subsubsection{Planar networks}

In this article, a {\em planar network} is a finite, directed, edge-weighted graph embedded in a disc, with no oriented cycles. The edges weights are nonzero complex numbers (or indeterminates which take values in $\Cx$). We assume there are $r$ distinguished source vertices, labeled $1, \ldots, r$, and $s$ distinguished sink vertices, labeled $1', \ldots, s'$. To each such network $N$, we associate an $r \times s$ matrix $M(N)$, as follows. Define the weight of a path to be the product of the weights of the edges in the path. The $(i,j)$-entry of $M(N)$ is the sum of the weights of all paths from source $i$ to sink $j'$, that is,
\[
M(N)_{ij} = \sum_{p \, : \, i \rightarrow j'} \wt(p).
\]
We say that $M(N)$ is the matrix associated to $N$, and that $N$ is a network representation of $M$. See Example \ref{ex_Lind} for an example.

Note that {\em gluing} of networks is compatible with matrix multiplication, in the sense that if a planar network $N$ is obtained by identifying the sinks of a planar network $N_1$ with the sources of a planar network $N_2$, then we have
\[
M(N) = M(N_1) \cdot M(N_2).
\]

Let $I = \{i_1 < \ldots < i_m\} \subset [r]$ and $J = \{j_1 < \ldots < j_m\} \subset [s]$ be two subsets of cardinality $m$. A {\em family of paths from $I$ to $J$} is a collection of $m$ paths $p_1, \ldots, p_m$, such that $p_a$ starts at source $i_a$ and ends at sink $j'_{\sigma(a)}$, for some permutation $\sigma \in S_m$. We denote such a family by $\mathcal{F} = (p_a; \sigma)$, and we define the weight of the family by $\wt(\mathcal{F}) = \prod_{a = 1}^m \wt(p_a)$. If no two of the paths share a vertex, we say that the family is {\em vertex-disjoint}.

\begin{prop}[Lindstr\"{o}m Lemma, \cite{Lind}]
Let $N$ be a planar network with $r$ sources and $s$ sinks, and let $I \subset [r], J \subset [s]$ be two subsets of the same cardinality. Then the minor of $M(N)$ using rows $I$ and columns $J$ is given by
\[
\Delta_{I,J}(M(N)) = \sum_{\mathcal{F} = (p_a; \sigma) \, : \, I \rightarrow J} \sgn(\sigma) \wt(\mathcal{F}),
\]
where the sum is over \underline{vertex-disjoint} families of paths from $I$ to $J$.
\end{prop}

\begin{ex}
\label{ex_Lind}
Here is a planar network with 5 sources and 3 sinks, and its associated matrix. Unlabeled edges are assumed to have weight 1.

\begin{center}
\begin{tikzpicture}[scale=0.65]
\draw (-1.5,2) node {$N=$};

\pgfmathtruncatemacro{\n}{5};
\pgfmathtruncatemacro{\k}{3};
\pgfmathtruncatemacro{\nminusk}{\n-\k};
\foreach \x in {1,...,\k}
{
\filldraw (3*\x,2*\nminusk) circle [radius = .05] node[anchor=south] {$\x'$};
\pgfmathtruncatemacro{\source}{\x + \nminusk};
\filldraw (3*\x,0) circle [radius = .05] node[anchor=north] {$\source$};
\foreach \z in {1,...,\nminusk}
{
\pgfmathtruncatemacro{\y}{\nminusk - \z};
\draw[thick,->] (3*\x,2*\y) -- (3*\x, 2*\y + 1);
\draw[thick] (3*\x,2*\y+1) -- (3*\x, 2*\y + 2);
\pgfmathtruncatemacro{\weight}{\z + \x - 1};
\draw[thick,->] (3*\x-3,2*\y) -- (3*\x-1.5,2*\y+1);
\draw[thick] (3*\x-1.5,2*\y+1) -- (3*\x,2*\y+2);
\draw (3*\x-1.7,2*\y+1.35) node {$x_{\z \weight}$};
}
}
\foreach \z in {1,...,\nminusk}
{
\pgfmathtruncatemacro{\y}{\nminusk - \z +1};
\filldraw (0,2*\z-2) circle [radius = .05] node[anchor=east] {$\y$};
}

\draw (16,2) node {$M(N) = \left(
\begin{array}{ccc}
x_{11} & 0 & 0 \\
x_{22} & x_{12}x_{22} & 0 \\
1 & x_{12} + x_{23} & x_{13}x_{23} \\
0 & 1 & x_{13} + x_{24} \\
0 & 0 & 1
\end{array}
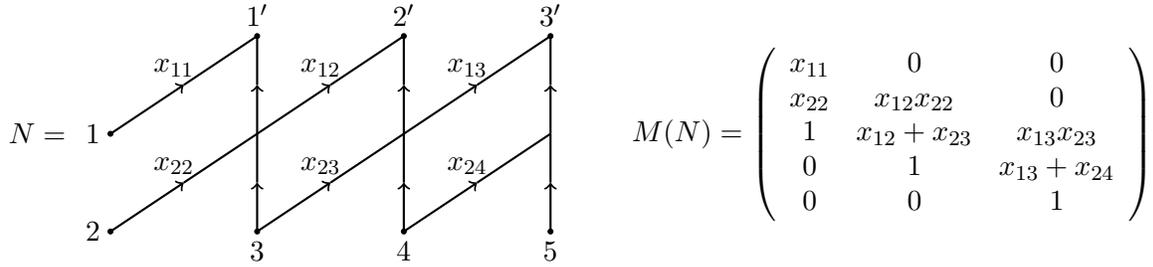
\right)$};

\end{tikzpicture}
\end{center}

\noindent There are three vertex-disjoint families of paths from $\{3,4\}$ to $\{2',3'\}$. The weights of these families are $x_{12}x_{13}, x_{12}x_{24},$ and $x_{23}x_{24}$, and in all three cases $\sigma$ is the identity permutation. From the matrix, one computes
\[
\Delta_{34, 23}(M(N)) = x_{12}x_{13} + x_{12}x_{24} + x_{23}x_{24},
\]
in agreement with the Lindstr\"{o}m Lemma.
\end{ex}

Every planar network appearing in this article has the property that only the identity permutation can appear in a vertex-disjoint family of paths, so the minors of the associated matrix are polynomials in the edge weights with non-negative coefficients.

\subsubsection{Network description of $\Theta_k$}

\begin{defn}
Suppose $(X_{ij},t) \in \bT_{n-k}$, and let $x_{ij} = X_{ij}/X_{i,j-1}$ as in \eqref{eq_x_ij}. Let $N_k = N_k(X_{ij},t)$ be the planar network on the vertex set $\bbZ^2$ with
\begin{itemize}
\item $k$ sinks labeled $1', \ldots, k'$, with the $j^{th}$ sink located at $(0,j)$;
\item $n$ sources labeled $1, \ldots, n$. The $i^{th}$ source is located at $(i,0)$ if $i \leq n-k$, and at $(n-k,i-n+k)$ if $i > n-k$;
\item an arrow pointing from $(i,j)$ to $(i,j+1)$ for $i = 1, \ldots, n-k$ and $j = 1, \ldots, k$. The weight of this edge is 1;
\item an arrow pointing from $(i,j)$ to $(i-1,j+1)$ for $i = 1, \ldots, n-k$ and $j = 0, \ldots, k-1$. The weight of this edge is $x_{i,i+j}$.
\end{itemize}
We will always draw the network $N_k$ using the convention for matrix indices, i.e., the first coordinate gives the vertical position, and increases from top to bottom; the second coordinate gives the horizontal position, and increases from left to right.
\end{defn}

\begin{ex}
\label{ex_net_proj_2_5}
Here is the network $N_k$ for $n=5$ and $k=2$. The network for $n=5, k=3$ is shown in Example \ref{ex_Lind}, and the network for $n=9, k=5$ appears in Figure \ref{fig_paths_Jtab} below.

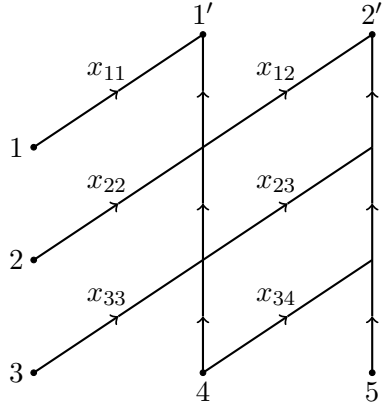
\begin{center}
\begin{tikzpicture}[scale=0.75]

\pgfmathtruncatemacro{\n}{5};
\pgfmathtruncatemacro{\k}{2};
\pgfmathtruncatemacro{\nminusk}{\n-\k};
\foreach \x in {1,...,\k}
{
\filldraw (3*\x,2*\nminusk) circle [radius = .05] node[anchor=south] {$\x'$};
\pgfmathtruncatemacro{\source}{\x + \nminusk};
\filldraw (3*\x,0) circle [radius = .05] node[anchor=north] {$\source$};
\foreach \z in {1,...,\nminusk}
{
\pgfmathtruncatemacro{\y}{\nminusk - \z};
\draw[thick,->] (3*\x,2*\y) -- (3*\x, 2*\y + 1);
\draw[thick] (3*\x,2*\y+1) -- (3*\x, 2*\y + 2);
\pgfmathtruncatemacro{\weight}{\z + \x - 1};
\draw[thick,->] (3*\x-3,2*\y) -- (3*\x-1.5,2*\y+1);
\draw[thick] (3*\x-1.5,2*\y+1) -- (3*\x,2*\y+2);
\draw (3*\x-1.7,2*\y+1.35) node {$x_{\z \weight}$};
}
}
\foreach \z in {1,...,\nminusk}
{
\pgfmathtruncatemacro{\y}{\nminusk - \z +1};
\filldraw (0,2*\z-2) circle [radius = .05] node[anchor=east] {$\y$};
}

\end{tikzpicture}
\end{center}

\end{ex}

\begin{lem}
\label{lem_Nk_represents}
If $(M,t) = \Theta_k(X_{ij},t)$, then the $n \times k$ matrix associated to the network $N_k(X_{ij},t)$ is a representative of the subspace $M$.
\end{lem}

\begin{proof}
Start by constructing a network for the $n \times n$ matrix $\Phi_{n-k}(X_{ij},t)$ by gluing together networks for matrices of the form $M_{[i,j]}(c_i, \ldots, c_j)$. For example, if $n=5$ and $k=2$, then the matrix $M_{[2,4]}(c_2,c_3,c_4)$ is represented by the network

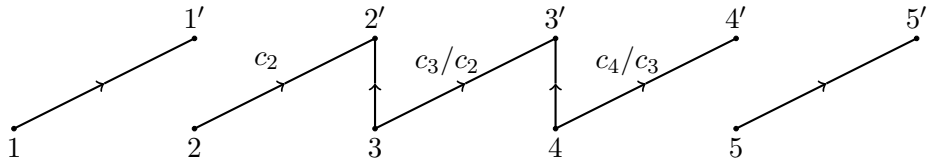
\begin{center}
\begin{tikzpicture}[scale=0.6]

\pgfmathtruncatemacro{\n}{5};
\pgfmathtruncatemacro{\k}{4};
\pgfmathtruncatemacro{\nminusk}{\n-\k};
\pgfmathtruncatemacro{\rows}{1};
\pgfmathtruncatemacro{\xscale}{4};
\pgfmathtruncatemacro{\yscale}{2};
\foreach \x in {1,...,\n}
{
\pgfmathtruncatemacro{\xx}{\xscale*\x};
\foreach \y in {1,...,\rows}
{
\pgfmathtruncatemacro{\yy}{\yscale*\y};
\pgfmathtruncatemacro{\xy}{\xscale*\y};
\pgfmathtruncatemacro{\yx}{\yscale*\x};

\ifthenelse{\x > 2 \AND \x < 5}{
\draw[thick,->] (\xx,\yy) -- (\xx,\yy + 0.5*\yscale);
\draw[thick] (\xx,\yy+0.5*\yscale) -- (\xx,\yy+\yscale);
}{}

\pgfmathtruncatemacro{\aa}{\nminusk - 1}
\pgfmathtruncatemacro{\yweight}{-\y+1+\x}
\pgfmathtruncatemacro{\xweight}{\nminusk - \y + 1}
\ifthenelse{\x = 2}{
\draw (\xx+0.4*\xscale,\yy+0.75*\yscale) node {$c_2$};
}{}
\ifthenelse{\x = 3}{
\draw (\xx+0.4*\xscale,\yy+0.75*\yscale) node {$c_{3}/c_{2}$};
}{}
\ifthenelse{\x = 4}{
\draw (\xx+0.4*\xscale,\yy+0.75*\yscale) node {$c_{4}/c_{3}$};
}{}

\draw[thick,->] (\xx + \xy - \xscale,\yy) -- (\xx+\xy-0.5*\xscale,\yy+0.5*\yscale);
\draw[thick] (\xx + \xy-0.5*\xscale,\yy+0.5*\yscale) -- (\xx+\xy, \yy+\yscale);
}

\filldraw (\xx,\yscale) circle [radius=.05] node[below] {$\x$};
\filldraw (\xx+ \xscale*\rows, \yscale*\rows+\yscale) circle [radius=.05] node[above] {$\x'$};
}

\end{tikzpicture}
\end{center}
and the matrix $\Phi_3(X_{ij},t)$ is represented by the network

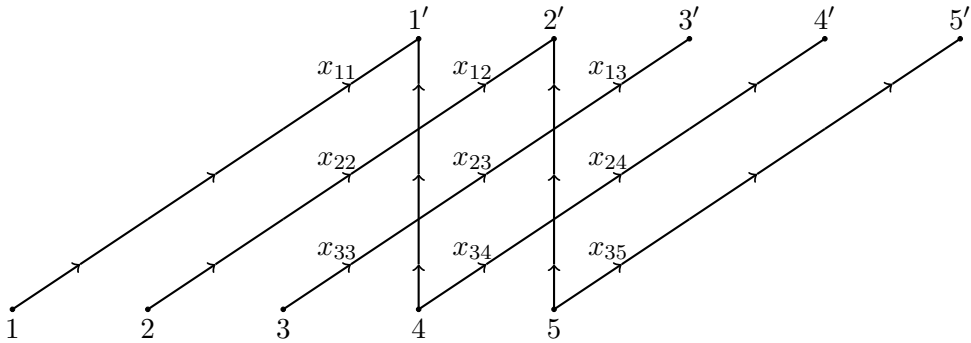
\begin{center}
\begin{tikzpicture}[scale=0.6]

\pgfmathtruncatemacro{\n}{5};
\pgfmathtruncatemacro{\k}{2};
\pgfmathtruncatemacro{\nminusk}{\n-\k};
\pgfmathtruncatemacro{\rows}{\nminusk};
\pgfmathtruncatemacro{\xscale}{3};
\pgfmathtruncatemacro{\yscale}{2};
\foreach \x in {1,...,\n}
{
\pgfmathtruncatemacro{\xx}{\xscale*\x};
\foreach \y in {1,...,\rows}
{
\pgfmathtruncatemacro{\yy}{\yscale*\y};
\pgfmathtruncatemacro{\xy}{\xscale*\y};
\pgfmathtruncatemacro{\yx}{\yscale*\x};

\ifthenelse{\x > \nminusk}{
\draw[thick,->] (\xx,\yy) -- (\xx,\yy + 0.5*\yscale);
\draw[thick] (\xx,\yy+0.5*\yscale) -- (\xx,\yy+\yscale);
}{}

\pgfmathtruncatemacro{\aa}{\nminusk - 1}
\pgfmathtruncatemacro{\yweight}{-\y+1+\x}
\pgfmathtruncatemacro{\xweight}{\nminusk - \y + 1}
\ifthenelse{\x > \aa}{
\draw (\xx+0.4*\xscale,\yy+0.65*\yscale) node {$x_{\xweight \yweight}$};
}{}

\draw[thick,->] (\xx + \xy - \xscale,\yy) -- (\xx+\xy-0.5*\xscale,\yy+0.5*\yscale);
\draw[thick] (\xx + \xy-0.5*\xscale,\yy+0.5*\yscale) -- (\xx+\xy, \yy+\yscale);
}

\filldraw (\xx,\yscale) circle [radius=.05] node[below] {$\x$};
\filldraw (\xx+ \xscale*\rows, \yscale*\rows+\yscale) circle [radius=.05] node[above] {$\x'$};
}

\end{tikzpicture}
\end{center}
As usual, unlabeled edges are assumed to have weight 1. (The reader may verify that the matrices \eqref{eq_chev_gens} and \eqref{eq_2_5} are indeed the weight matrices of these two networks.)

To get a network representative for the first $k$ columns of $\Phi_{n-k}(X_{ij},t)$, simply erase everything to the right of the sink $k'$. Finally, contract the diagonal edges of weight 1 leaving the sources $1, \ldots, n-k$ (this doesn't change the associated matrix) to obtain the network $N_k(X_{ij},t)$.
\end{proof}

\begin{cor}
\label{cor_M}
Suppose $(X_{ij},t) \in \bT_{n-k}$.
\begin{enumerate}
\item The $k \times k$ submatrix in the bottom left corner of $\Phi_{n-k}(X_{ij}, t)$ is upper-triangular with 1's on the main diagonal.
\item If $(M,t) = \Theta_k(X_{ij}, t)$, then the subspace $M$ is independent of the parameter $t$.
\end{enumerate}
\end{cor}

\begin{proof}
For $i = n-k+1, \ldots, n$, there are no paths in $N_k(X_{ij},t)$ from source $i$ to sink $j'$ if $j < i-n+k$, and there is a unique path of weight 1 from source $i$ to sink $(i-n+k)'$. This proves (1).

Part (2) follows from the fact that $x_{ij}$ only depends on $t$ if $j = i+k$, and $x_{i,i+k}$ does not appear as an edge weight in $N_k(X_{ij},t)$.
\end{proof}

\subsubsection{$J$-tableaux}
\label{sec_J_tab}

Let $M$ be the $n \times k$ matrix associated to the network $N_k(X_{ij},t)$. The Lindstr\"{o}m Lemma expresses the Pl\"{u}cker coordinates of $M$ in terms of the quantities $x_{ij} = X_{ij}/X_{i,j-1}$ by adding up the weights of vertex-disjoint families of paths. Here we give a more explicit formula for these minors in terms of a combinatorial object that we call a $J$-tableau.

For $k \in [1,n-1]$, let
\[
D_k = \{(a,b) \in \bbZ^2 \,|\, 1 \leq a \leq b \leq k\}
\]
be the shifted staircase of size $k$. We identify $D_k$ with its ``Young diagram,'' so that each point $(a,b) \in D_k$ corresponds to a box in row $a$ and column $b$ of the diagram. Given a subset $J = \{j_1 < j_2 < \cdots < j_k\} \in {[n] \choose k}$, let $D_{J,k}$ be the subset of $D_k$ obtained by removing $j_r - n + k$ boxes from the bottom of column $r$, for each $r$ such that $j_r > n-k$. For example, if $n = 8$, then
\[
D_3 =
\ytableaushort{{}{}{},\none{}{},\none\none{}}
\quad\quad \text{and} \quad\quad
D_{\{4,5,7\},3} = \ytableaushort{{}{}{},\none{}\none,\none\none\none} \, .
\]

\begin{defn}
\label{defn_J_tab}
Let $J = \{j_1 < \cdots < j_k\}$. A {\em $(J,k)$-tableau} is a map $T : D_{J,k} \rightarrow [n-k]$ satisfying the following three properties:
\begin{enumerate}
\item[(a)] $T(a,b) \leq T(a,b+1)$ whenever $(a,b), (a,b+1) \in D_{J,k}$;
\item[(b)] $T(a,b) < T(a+1,b)$ whenever $(a,b), (a+1,b) \in D_{J,k}$;
\item[(c)] $T(a,a) = j_a$ if $j_a \leq n-k$.
\end{enumerate}
We will often write $J$-tableau instead of $(J,k)$-tableau when $k$ is understood.

Let $\{z_{ij} \, | \, (i,j) \in R_{n-k}\}$ (see \eqref{eq_R_k}) be an $(n-k) \times k$ array of indeterminates. Define the {\em weight} of a $J$-tableau $T$ (with respect to $z_{ij}$) by
\[
\wt(T) = \prod_{(a,b) \in D_{J,k}} z_{T(a,b), T(a,b) + b-a}.
\]
If $D_{J,k}$ is empty (i.e., if $J = [n-k+1,n]$), we define the weight of the unique (empty) $J$-tableau to be 1.
\end{defn}

Note that properties (a) and (b) require the rows of $T$ to weakly increase, and the columns to strictly increase.

\begin{ex}
Let $n=8, k=3$. There are two $\{4,5,7\}$-tableaux, shown here with their weights:
\[
\ytableaushort{{4}{4}{4},\none{5}\none,\none\none\none} \quad z_{44}z_{45}z_{46}z_{55}
\quad\quad\quad\quad\quad\quad\quad
\ytableaushort{{4}{4}{5},\none{5}\none,\none\none\none} \quad z_{44}z_{45}z_{55}z_{56} \,
\]
\end{ex}

\begin{lem}
\label{lem_pluc_formula}
Suppose $(X_{ij},t) \in \bT_{n-k}$, and $(M,t) = \Theta_k(X_{ij}, t)$. Then we have
\begin{equation}
\label{eq_pluc_formula}
P_J(M) = \sum_T \wt(T)
\end{equation}
where the sum runs over all $(J,k)$-tableaux $T$, and the weight is taken with respect to $x_{ij} = X_{ij}/X_{i,j-1}$.
\end{lem}

\begin{proof}
By the Lindstr\"{o}m Lemma, $P_J(M)$ is equal to the weighted sum over vertex-disjoint families of paths $p_1, \ldots, p_k$ in $N_k(X_{ij},t)$, where $p_r$ starts at the source $j_r$ and ends at the sink $r'$. Let $p_1, \ldots, p_r$ be a (not necessarily vertex-disjoint) family of paths such that $p_r$ goes from source $j_r$ to sink $r'$. The number of diagonal edges in $p_r$ is, by definition, equal to the number of boxes in the $r^{th}$ column of the diagram $D_{J,k}$. Define $T : D_{J,k} \rightarrow [n-k]$ by filling the $r^{th}$ column of $D_{J,k}$ with the heights of the diagonal edges in the $r^{th}$ path (in increasing order), where the height of the edge from $(i,j)$ to $(i-1,j+1)$ is $i$. See Figure \ref{fig_paths_Jtab} for an example of a family of paths in $N_k(X_{ij},t)$ and the associated filling of $D_{J,k}$.

It's clear that the association $(p_r)_{1 \leq r \leq k} \mapsto T$ is a bijection between (not necessarily vertex-disjoint) families of paths and fillings of $D_{J,k}$ satisfying properties (b) and (c) of Definition \ref{defn_J_tab}. It's also not hard to see that the rows of $T$ are weakly increasing if and only if the family of paths is vertex-disjoint, and that the association is weight-preserving. This completes the proof.
\end{proof}

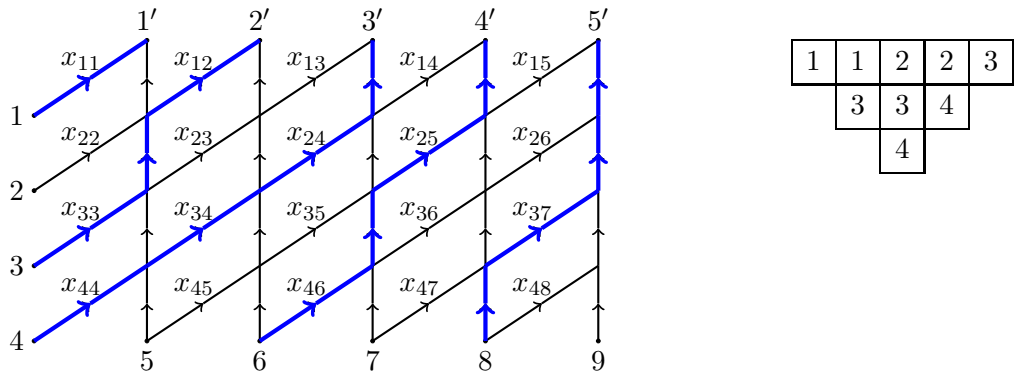
\begin{figure}

\begin{multicols}{2}

\begin{tikzpicture}[scale=0.5]

\pgfmathtruncatemacro{\n}{9};
\pgfmathtruncatemacro{\k}{5};
\pgfmathtruncatemacro{\nminusk}{\n-\k};
\foreach \x in {1,...,\k}
{
\filldraw (3*\x,2*\nminusk) circle [radius = .05] node[anchor=south] {$\x'$};
\pgfmathtruncatemacro{\source}{\x + \nminusk};
\filldraw (3*\x,0) circle [radius = .05] node[anchor=north] {$\source$};
\foreach \z in {1,...,\nminusk}
{
\pgfmathtruncatemacro{\y}{\nminusk - \z};
\draw[thick,->] (3*\x,2*\y) -- (3*\x, 2*\y + 1);
\draw[thick] (3*\x,2*\y+1) -- (3*\x, 2*\y + 2);

\pgfmathtruncatemacro{\weight}{\z + \x - 1};
\draw[thick,->] (3*\x-3,2*\y) -- (3*\x-1.5,2*\y+1);
\draw[thick] (3*\x-1.5,2*\y+1) -- (3*\x,2*\y+2);
\draw (3*\x-1.75,2*\y+1.45) node {$x_{\z \weight}$};
}
}
\foreach \z in {1,...,\nminusk}
{
\pgfmathtruncatemacro{\y}{\nminusk - \z +1};
\filldraw (0,2*\z-2) circle [radius = .05] node[anchor=east] {$\y$};
}

\pgfmathtruncatemacro{\x}{1}
\pgfmathtruncatemacro{\y}{2}
\draw[ultra thick,blue,->] (3*\x,2*\y) -- (3*\x, 2*\y + 1);
\draw[ultra thick,blue] (3*\x,2*\y+1) -- (3*\x, 2*\y + 2);

\pgfmathtruncatemacro{\x}{3}
\pgfmathtruncatemacro{\y}{3}
\draw[ultra thick,blue,->] (3*\x,2*\y) -- (3*\x, 2*\y + 1);
\draw[ultra thick,blue] (3*\x,2*\y+1) -- (3*\x, 2*\y + 2);

\pgfmathtruncatemacro{\x}{3}
\pgfmathtruncatemacro{\y}{1}
\draw[ultra thick,blue,->] (3*\x,2*\y) -- (3*\x, 2*\y + 1);
\draw[ultra thick,blue] (3*\x,2*\y+1) -- (3*\x, 2*\y + 2);

\pgfmathtruncatemacro{\x}{4}
\pgfmathtruncatemacro{\y}{3}
\draw[ultra thick,blue,->] (3*\x,2*\y) -- (3*\x, 2*\y + 1);
\draw[ultra thick,blue] (3*\x,2*\y+1) -- (3*\x, 2*\y + 2);

\pgfmathtruncatemacro{\x}{5}
\pgfmathtruncatemacro{\y}{3}
\draw[ultra thick,blue,->] (3*\x,2*\y) -- (3*\x, 2*\y + 1);
\draw[ultra thick,blue] (3*\x,2*\y+1) -- (3*\x, 2*\y + 2);

\pgfmathtruncatemacro{\x}{5}
\pgfmathtruncatemacro{\y}{2}
\draw[ultra thick,blue,->] (3*\x,2*\y) -- (3*\x, 2*\y + 1);
\draw[ultra thick,blue] (3*\x,2*\y+1) -- (3*\x, 2*\y + 2);

\pgfmathtruncatemacro{\x}{4}
\pgfmathtruncatemacro{\y}{0}
\draw[ultra thick,blue,->] (3*\x,2*\y) -- (3*\x, 2*\y + 1);
\draw[ultra thick,blue] (3*\x,2*\y+1) -- (3*\x, 2*\y + 2);

\pgfmathtruncatemacro{\x}{1}
\pgfmathtruncatemacro{\z}{1}
\pgfmathtruncatemacro{\weight}{\z + \x - 1}
\pgfmathtruncatemacro{\y}{\nminusk - \z}
\draw[ultra thick,blue,->] (3*\x-3,2*\y) -- (3*\x-1.5,2*\y+1);
\draw[ultra thick,blue] (3*\x-1.5,2*\y+1) -- (3*\x,2*\y+2);

\pgfmathtruncatemacro{\x}{1}
\pgfmathtruncatemacro{\z}{3}
\pgfmathtruncatemacro{\weight}{\z + \x - 1}
\pgfmathtruncatemacro{\y}{\nminusk - \z}
\draw[ultra thick,blue,->] (3*\x-3,2*\y) -- (3*\x-1.5,2*\y+1);
\draw[ultra thick,blue] (3*\x-1.5,2*\y+1) -- (3*\x,2*\y+2);

\pgfmathtruncatemacro{\x}{1}
\pgfmathtruncatemacro{\z}{4}
\pgfmathtruncatemacro{\weight}{\z + \x - 1}
\pgfmathtruncatemacro{\y}{\nminusk - \z}
\draw[ultra thick,blue,->] (3*\x-3,2*\y) -- (3*\x-1.5,2*\y+1);
\draw[ultra thick,blue] (3*\x-1.5,2*\y+1) -- (3*\x,2*\y+2);

\pgfmathtruncatemacro{\x}{2}
\pgfmathtruncatemacro{\z}{1}
\pgfmathtruncatemacro{\weight}{\z + \x - 1}
\pgfmathtruncatemacro{\y}{\nminusk - \z}
\draw[ultra thick,blue,->] (3*\x-3,2*\y) -- (3*\x-1.5,2*\y+1);
\draw[ultra thick,blue] (3*\x-1.5,2*\y+1) -- (3*\x,2*\y+2);

\pgfmathtruncatemacro{\x}{2}
\pgfmathtruncatemacro{\z}{3}
\pgfmathtruncatemacro{\weight}{\z + \x - 1}
\pgfmathtruncatemacro{\y}{\nminusk - \z}
\draw[ultra thick,blue,->] (3*\x-3,2*\y) -- (3*\x-1.5,2*\y+1);
\draw[ultra thick,blue] (3*\x-1.5,2*\y+1) -- (3*\x,2*\y+2);

\pgfmathtruncatemacro{\x}{3}
\pgfmathtruncatemacro{\z}{2}
\pgfmathtruncatemacro{\weight}{\z + \x - 1}
\pgfmathtruncatemacro{\y}{\nminusk - \z}
\draw[ultra thick,blue,->] (3*\x-3,2*\y) -- (3*\x-1.5,2*\y+1);
\draw[ultra thick,blue] (3*\x-1.5,2*\y+1) -- (3*\x,2*\y+2);

\pgfmathtruncatemacro{\x}{3}
\pgfmathtruncatemacro{\z}{4}
\pgfmathtruncatemacro{\weight}{\z + \x - 1}
\pgfmathtruncatemacro{\y}{\nminusk - \z}
\draw[ultra thick,blue,->] (3*\x-3,2*\y) -- (3*\x-1.5,2*\y+1);
\draw[ultra thick,blue] (3*\x-1.5,2*\y+1) -- (3*\x,2*\y+2);

\pgfmathtruncatemacro{\x}{4}
\pgfmathtruncatemacro{\z}{2}
\pgfmathtruncatemacro{\weight}{\z + \x - 1}
\pgfmathtruncatemacro{\y}{\nminusk - \z}
\draw[ultra thick,blue,->] (3*\x-3,2*\y) -- (3*\x-1.5,2*\y+1);
\draw[ultra thick,blue] (3*\x-1.5,2*\y+1) -- (3*\x,2*\y+2);

\pgfmathtruncatemacro{\x}{5}
\pgfmathtruncatemacro{\z}{3}
\pgfmathtruncatemacro{\weight}{\z + \x - 1}
\pgfmathtruncatemacro{\y}{\nminusk - \z}
\draw[ultra thick,blue,->] (3*\x-3,2*\y) -- (3*\x-1.5,2*\y+1);
\draw[ultra thick,blue] (3*\x-1.5,2*\y+1) -- (3*\x,2*\y+2);

\end{tikzpicture}

\columnbreak

\[
\ytableaushort{{1}{1}{2}{2}{3},\none{3}{3}{4},\none\none{4}}
\]
\end{multicols}

\caption{A collection of disjoint paths in $N_{5}(x_{ij})$ (with $n=9$) contributing to the Pl\"{u}cker coordinate $P_{13468}$, and the corresponding $\{1,3,4,6,8\}$-tableau.}
\label{fig_paths_Jtab}
\end{figure}

\begin{remark}
A similar result, also using an object called $J$-tableau to record vertex-disjoint families of paths, appeared previously (\cite[Proposition 2.6.7]{BFZ}). In that setting, $J$-tableaux are related to flag minors of an $n \times n$ matrix, rather than maximal minors of an $n \times k$ matrix.
\end{remark}

\begin{cor}
\label{cor_pluc_formulae}
If $(M,t) = \Theta_k(X_{ij},t)$, then
\begin{enumerate}
\item For all $J \in {[n] \choose k}$, $P_J(M)$ is a non-zero homogeneous polynomial of degree $|D_{J,k}|$ in the quantities $x_{ij} = X_{ij}/X_{i,j-1}$, with positive coefficients.
\item Suppose $i \in [n-k+1]$ and $0 \leq j-i \leq k-1$. Set $c = n-k+j-i+2$. Then we have
\begin{equation}
\label{eq_gen_cyc_pluc}
P_{[i,j] \cup [c,n]}(M) = \prod_{a \in [i,j] \cap [n-k]} X_{aj}.
\end{equation}
\end{enumerate}
\end{cor}

\begin{proof}
For any $J$, there is at least one $J$-tableau, namely, the tableau with all entries in row $r$ equal to $j_r$. The weight of every $J$-tableau is a monomial of degree $|D_{J,k}|$, so (1) is proved.

To prove (2), set $J = [i,j] \cup [c,n]$. We claim that there is only one $J$-tableau. Indeed, the first entry in the $a^{th}$ row must be $i+a-1$, and the lengths of the columns of $D_{J,k}$ are (in this case) weakly increasing, so every entry in the $a^{th}$ row must be $i+a-1$. The weight of this unique $J$-tableau is
\[
\prod_{a = 1}^m \prod_{b = a}^{j-i+1} x_{i+a-1,i+b-1} = \prod_{a' = i}^{m+i-1} X_{a'j}
\]
where $m = \min(j-i+1, n-k-i+1)$. Since $[i,m+i-1] = [i,j] \cap [n-k]$, we obtain \eqref{eq_gen_cyc_pluc}.
\end{proof}

\subsection{Basic Pl\"{u}cker coordinates}
\label{sec_basic_pluc}

We now focus on the special type of Pl\"{u}cker coordinates that appear in the formula for $\Theta_k^{-1}$ given in Proposition \ref{prop_GT_pl}. We first need a simple observation from linear algebra. Say that an $n \times k$ matrix $M$ has {\em diagonal form} if its first $k$ rows are lower triangular with nonzero entries on the main diagonal, and its last $k$ rows are upper triangular with $1$'s on the main diagonal. For example, if $n = 7$ and $k = 3$, then a matrix of diagonal form looks like
\begin{equation*}
\left(
\begin{array}{ccc}
a_1 & 0 & 0 \\
* & a_2 & 0 \\
* & * & a_3 \\
* & * & * \\
1 & * & * \\
0 & 1 & * \\
0 & 0 & 1
\end{array}
\right)
\end{equation*}
where $a_1, a_2, a_3$ are nonzero, and the $*$'s are arbitrary.

Say that a Pl\"{u}cker coordinate $P_J$ is {\em cyclic} if the elements of $J$ are consecutive mod $n$, and let $X_k^\circ$ denote the {\em open positroid cell}, the open subset of $\Gr(k,n)$ where the cyclic Pl\"{u}cker coordinates do not vanish.

\begin{lem}
\label{lem_diag_form}
Every subspace in $X_k^\circ$ is the column span of a unique $n \times k$ matrix of diagonal form.
\end{lem}

\begin{proof}
Suppose $M \in X_k^\circ$. Since $P_{[n-k+1,n]}(M) \neq 0$, $M$ can represented by an $n \times k$ matrix $M^\circ$ whose bottom $k$ rows are the identity matrix. Clearly we have $\Delta_{[1,i],[1,i]}(M^\circ) = \dfrac{P_{[1,i] \cup [n-k+i+1,n]}(M)}{P_{[n-k+1,n]}(M)}$ for $i \leq k$, so the principal minors $\Delta_{[1,i],[1,i]}(M^\circ)$ are nonzero for $i \leq k$. We may therefore use Gaussian elimination on the columns of $M^\circ$ to make the first $k$ rows lower triangular with nonzero entries on the main diagonal. The last $k$ rows will still have 1's on the main diagonal and 0's beneath the main diagonal, so we obtain a diagonal form representative of the subspace $M$. Uniqueness is clear.
\end{proof}

\begin{defn}
\label{defn_basic}
For $1 \leq i \leq n-k+1$ and $i-1 \leq j \leq i+k-1$, define the $k$-subset
\[
J_{i,j} = [i,j] \cup [n-k+j-i+2,n].
\]
We will call a subset of this form a {\em basic subset}, and we will refer to $P_{J_{i,j}}$ as a {\em basic Pl\"{u}cker coordinate}.
Define $U_k$ to be the open subset of $Gr(k,n)$ consisting of subspaces whose basic Pl\"{u}cker coordinates are all nonzero.
\end{defn}

Note that there is some redundancy in the definition of basic subsets: if $i = n-k+1$ or $j = i-1$, then $J_{i,j} = [n-k+1,n]$. Note also that cyclic Pl\"{u}cker coordinates are basic, so every element of $U_k$ has a diagonal form representative by Lemma \ref{lem_diag_form}. If $M^\circ$ is the diagonal form representative of $M \in U_k$, then we have
\begin{equation}
\label{eq_ij_minor}
\Delta_{[i,j],[1,j-i+1]}(M^\circ) = \dfrac{\Delta_{J_{i,j},[k]}(M^\circ)}{\Delta_{[n-k+1,n],[k]}(M^\circ)} = \dfrac{P_{J_{i,j}}(M)}{P_{[n-k+1,n]}(M)}
\end{equation}
for $1 \leq i \leq n-k+1$ and $i-1 \leq j \leq i+k-1$. This observation leads to the following result.

\begin{lem}
\label{lem_U_k_determines}
Every element of $U_k$ is uniquely determined by its basic Pl\"{u}cker coordinates.
\end{lem}

\begin{proof}
Suppose $M \in U_k$, and let $M^\circ$ be its diagonal form representative. We inductively show that all the entries of $M^\circ$ are determined by the basic Pl\"{u}cker coordinates of $M$. Consider an entry $M^\circ_{ab}$ which is not automatically 0 or 1, and assume that $M^\circ_{a'b'}$ is known for $a' < a$, and for $a' = a, b'<b$. Expand the determinant $\Delta_{[a-b+1,a],[1,b]}(M^\circ)$ along its last column. This gives an equation
\begin{equation}
\label{eq_M_ab}
M^\circ_{ab} \cdot \Delta_{[a-b+1,a-1],[1,b-1]}(M^\circ) = \Delta_{[a-b+1,a],[1,b]}(M^\circ) \; + \text{ a polynomial in known entries of } M^\circ.
\end{equation}
By \eqref{eq_ij_minor}, the two determinants in \eqref{eq_M_ab} are ratios of basic Pl\"{u}cker coordinates of $M$, and since these are nonzero, the entry $M^\circ_{ab}$ is determined.
\end{proof}

\begin{proof}[Proof of Proposition \ref{prop_GT_pl}]
Define $\Psi_k : U_k \times \Cx \rightarrow \bT_{n-k}$ by $\Psi_k(M,t) = (Y_{ij},t)$, where
\[
Y_{ij} = \dfrac{P_{J_{i,j}}(M)}{P_{J_{i+1,j}}(M)}.
\]
Suppose $(X_{ij}, t) \in \bT_{n-k}$, and let $(M,t) = \Theta_k(X_{ij},t)$. By part (2) of Corollary \ref{cor_pluc_formulae}, the basic Pl\"{u}cker coordinates of $M$ are monomials in the $X_{ij}$ (so they are nonzero), and we have
\[
\dfrac{P_{J_{i,j}}(M)}{P_{J_{i+1,j}}(M)} = \dfrac{\ds \prod_{a \in [i,j] \cap [1,n-k]} X_{aj}}{\ds \prod_{a \in [i+1,j] \cap [1,n-k]} X_{aj}} = X_{ij},
\]
so $\Psi_k \circ \Theta_k = \Id$.

Now suppose $M \in U_k$. Set $(X'_{ij}, t) = \Psi_k(M,t)$, and $(M',t) = \Theta_k(X'_{ij}, t)$. By part (2) of Corollary \ref{cor_pluc_formulae}, we have
\begin{align*}
\dfrac{P_{J_{i,j}}(M')}{P_{[n-k+1,n]}(M')} = \prod_{a \in [i,j] \cap [1,n-k]} X'_{aj} &= \prod_{a \in [i,j] \cap [1,n-k]} \dfrac{P_{J_{a,j}}(M)}{P_{J_{a+1,j}}(M)} \\
&= \dfrac{P_{J_{i,j}}(M)}{P_{J_{j+1,j}}(M)} = \dfrac{P_{J_{i,j}}(M)}{P_{[n-k+1,n]}(M)}.
\end{align*}
This shows that $M$ and $M'$ have the same basic Pl\"{u}cker coordinates, so $M = M'$ by Lemma \ref{lem_U_k_determines}. Thus, $\Theta_k \circ \Psi_k = \Id$, and we are done.
\end{proof}

\begin{cor}
\label{cor_basic=basis}
Every Pl\"{u}cker coordinate can be written as a Laurent polynomial in the basic Pl\"{u}cker coordinates with non-negative integer coefficients.
\end{cor}

\begin{proof}
Suppose $M \in U_k$. By Proposition \ref{prop_GT_pl}, we have $(M,t) = \Theta_k(X_{ij},t)$, where $X_{ij} = \dfrac{P_{J_{i,j}}(M)}{P_{J_{i+1,j}}(M)}$. Let $M^\circ$ be the diagonal form representative of $M$. By Lemma \ref{lem_Nk_represents}, $M^\circ$ is the matrix associated to a planar network whose weights are ratios of the $X_{ij}$, which are themselves ratios of the basic Pl\"{u}cker coordinates of $M$. Thus, by the Lindstr\"{o}m Lemma, every minor of $M^\circ$ is a Laurent polynomial in the basic Pl\"{u}cker coordinates of $M$ with non-negative integer coefficients. Since $U_k$ is a dense subset of $\Gr(k,n)$, the result follows.
\end{proof}

\begin{remark}
\label{rmk_cluster_alg}
Corollary \ref{cor_basic=basis} is a special case of the (positive) Laurent phenomenon in the theory of cluster algebras. Indeed, the $k(n-k)+1$ basic Pl\"{u}cker coordinates are a cluster in the coordinate ring of the affine cone over $\Gr(k,n)$ (see \cite[Figure 18]{MarSco}).
\end{remark}

\section{Tropicalization}
\label{sec_trop}

\subsection{Definition of tropicalization}
\label{sec_trop_defn}

There are multiple ways to precisely define tropicalization. Here we follow the approach of \cite[\S 2.1]{BFZ} (for a more general framework, see \cite[\S 4]{BKII}).

Following \cite{BFZ}, a {\em semifield} is a set $K$ endowed with two operations, addition and multiplication, such that addition is commutative and associative, $K$ is an abelian group under multiplication, and multiplication distributes over addition. Note that division (the inverse of multiplication) is defined in a semifield, but subtraction (the inverse of addition) is not.

The {\em tropical semifield} $\bbZ_{\Trop}$ is the set of integers endowed with the following addition and multiplication:
\[
a \oplus b := \min(a,b) \quad\quad\quad a \odot b := a+b.
\]
The {\em universal semifield} $\bbQ_{>0}(z_1, \ldots, z_r)$ is the set of nonzero rational functions in $r$ variables which can be written as a ratio of two polynomials whose coefficients are non-negative integers. The operations are the usual addition and multiplication of rational functions. For example, $f = z_1^2 - z_1z_2 + z_2^2$ is an element of $\bbQ_{>0}(z_1,z_2)$, since $f = \dfrac{z_1^3 + z_2^3}{z_1 + z_2}$. We will call elements of $\bbQ_{>0}(z_1, \ldots, z_r)$ {\em subtraction-free rational functions}.

The universal semifield is ``universal'' in the sense that, given a semifield $K$ and elements $a_1, \ldots, a_r \in K$, there is a unique homomorphism of semifields $\bbQ_{>0}(z_1, \ldots, z_r) \rightarrow K$ which sends $z_i \mapsto a_i$ \cite[Lemma 2.1.6]{BFZ}. We denote the image of a rational function $f(z_1, \ldots, z_r)$ under this homomorphism by $f_K(a_1, \ldots, a_r)$. Note that $f_K(a_1, \ldots, a_r)$ is computed by replacing the indeterminates $z_i$ with the semifield elements $a_i$, and replacing addition, multiplication, and division with the corresponding operations in $K$.

\begin{defn}
\label{defn_trop}
Given a subtraction-free rational function $f \in \bbQ_{>0}(z_1, \ldots, z_r)$, define $\Trop(f)$ to be the piecewise-linear function $\bbZ_{\Trop}^r \rightarrow \bbZ_{\Trop}$ which maps $(a_1, \ldots, a_r) \in \bbZ \mapsto f_{\bbZ_{\Trop}}(a_1, \ldots, a_r)$.

More generally, given a subtraction-free rational map $f = (f_1, \ldots, f_s) \in (\bbQ_{>0}(z_1, \ldots, z_r))^s$, define $\Trop(f)$ to be the piecewise-linear map $(\Trop(f_1), \ldots, \Trop(f_s)) : (\bbZ_{\Trop})^r \rightarrow (\bbZ_{\Trop})^s$.

We call $\Trop(f)$ the {\em tropicalization} of $f$.\footnote{Sometimes what we call ``tropicalization'' is called ``ultradiscretization,'' and ``tropicalization'' (or ``tropification'') refers to the inverse procedure of finding a rational lift. Also, $\max$ is sometimes used in place of $\min$.}
\end{defn}

We will implicitly use the fact that tropicalization respects addition, multiplication, and composition of subtraction-free rational maps. This implies, for instance, that for $m \in \bbZ_{>0}$ and $f \in \bbQ_{>0}(z_1, \ldots, z_r)$ we have
\[
\Trop(f^m) = m\Trop(f) \quad\quad \Trop(mf) = \Trop(f) \quad\quad \Trop(m) = \Trop(mf/f) = 0.
\]
For example, if $f = \dfrac{z_1^2z_2 + z_3}{z_2^5 + 8z_1z_3 + 4}$, then $\Trop(f) = \min(2a_1 + a_2, a_3) - \min(5a_2, a_1+a_3)$.

\subsection{Recovering the combinatorial crystals}
\label{sec_recover}

By tropicalizing the rational maps associated to the geometric crystal $\X{n-k} = \Y{n-k}{n}$, we obtain piecewise-linear maps on $\tw{\bT}_k$. We will show that these piecewise-linear maps, when restricted to the set of $k$-rectangles inside $\tw{\bT}_k$ (Definition \ref{defn_k_rect}), give formulas for the combinatorial crystal structure on $k$-row rectangular tableaux. We informally summarize the results of this section by saying that
\begin{equation}
\label{eq_trop_informal}
\textit{the geometric crystal on } \X{n-k} \textit{ tropicalizes to the disjoint union } \ds \bigsqcup_{L \geq 0} B^{k,L}.
\end{equation}

\subsubsection{Preliminaries}
Recall the Gelfand-Tsetlin parametrization $\Theta_{n-k} : \bT_k \rightarrow \X{n-k}$ from \S \ref{sec_GT_param}. If $F : \X{n-k} \rightarrow (\Cx)^r$ is a dominant rational map, set
\[
\Theta F = F \circ \Theta_{n-k} : \bT_k \rightarrow (\Cx)^r.
\]
If $\Theta F$ is subtraction-free, set $\wh{F} = \Trop(\Theta F) : \tw{\bT}_k \rightarrow \mathbb{Z}^r$. This is the piecewise-linear map on elements $(B_{ij}, L) \in \tw{\bT}_k$ which is obtained from the rational map $\Theta F$ by tropicalizing the operations and replacing $X_{ij}$ with $B_{ij}$, and $t$ with $L$.

Similarly, if $G : \X{n-k} \rightarrow \X{n-\ell}$ is a dominant rational map, set
\[
\Theta G = \Theta_{n-\ell}^{-1} \circ G \circ \Theta_{n-k} : \bT_k \rightarrow \bT_\ell,
\]
and if $\Theta G$ is positive, define $\wh{G} = \Trop(\Theta G) : \tw{\bT}_k \rightarrow \tw{\bT}_\ell$.

There is a useful sufficient condition for $\Theta F$ (resp., $\Theta G$) to be subtraction-free.

\begin{lem}
\label{lem_sub_free}
\
\begin{enumerate}
\item Let $F : \X{n-k} \rightarrow (\Cx)^r$ be a rational map sending $(M,t) \mapsto (x_1, \ldots, x_r)$. If each component function $x_i$ can be written as a (nonzero) subtraction-free rational function in the Pl\"{u}cker coordinates of $M$ and the variable $t$, then $\Theta F$ is subtraction-free.
\item Let $G : \X{n-k} \rightarrow \X{n-\ell}$ be a rational map sending $(M,t) \mapsto (M',t)$, and suppose that for each \underline{basic} $(n-\ell)$-subset $J$ (Definition \ref{defn_basic}), $P_J(M')$ can be written as a (nonzero) subtraction-free rational function in the Pl\"{u}cker coordinates of $M$ and the variable $t$. Then $\Theta G$ is subtraction-free.
\end{enumerate}
\end{lem}

\begin{proof}
Suppose $(M,t) = \Theta_k(X_{ij},t)$. By Corollary \ref{cor_basic=basis}, the Pl\"{u}cker coordinates of $M$ are given by Laurent polynomials in the variables $X_{ij}$ with non-negative integer coefficients, and these Laurent polynomials are nonzero by Corollary \ref{cor_pluc_formulae}(1). Now part (1) follows immediately, and part (2) follows from the fact that the component functions of $\Theta_{n-k}^{-1}$ are ratios of basic Pl\"{u}cker coordinates (Proposition \ref{prop_GT_pl}).
\end{proof}

\subsubsection{Cutting out the set of $k$-rectangles}
\label{sec_trop_dec}

The first step in ``proving'' \eqref{eq_trop_informal} is to show that the tropicalization of the decoration $f$ is able to identify the set of $k$-rectangles inside $\tw{\bT}_k$. Recall from \S \ref{sec_geom_Gr} that $f : \X{n-k} \rightarrow \Cx$ is defined by
\begin{equation}
\label{eq_dec_n-k}
f(M,t) = \sum_{i \neq n-k} \frac{P_{\{i-n+k\} \cup [i-n+k+2,i]}(M)}{P_{[i-n+k+1,i]}(M)} + t \frac{P_{[2,n-k] \cup \{n\}}(M)}{P_{[1,n-k]}(M)}.
\end{equation}
Clearly $f$ is subtraction-free in the sense of Lemma \ref{lem_sub_free}(1), so $\Theta f = f \circ \Theta_{n-k}$ is subtraction-free, and we may define
\[
\wh{f} = \Trop(\Theta f) : \tw{\bT}_k \rightarrow \mathbb{Z}.
\]

\begin{prop}
\label{prop_trop_dec}
Suppose $b = (B_{ij}, L) \in \tw{\bT}_k$. Then $\wh{f}(b) \geq 0$ if and only if $b$ is a $k$-rectangle.
\end{prop}

\begin{proof}
From the defining inequalities of a Gelfand-Tsetlin pattern, it's clear that $b$ is a $k$-rectangle if and only if the following inequalities are satisfied:
\begin{enumerate}
\item $B_{kk} \geq 0$
\item $L \geq B_{1,n-k}$
\item $B_{ij} \geq B_{i,j-1} \text{ for } i \in [k] \text{ and } j \in [i+1,i+n-k-1]$
\item $B_{ij} \geq B_{i+1,j+1} \text{ for } i \in [k-1] \text{ and } j \in [i,i+n-k-1]$.
\end{enumerate}

We will show that for $(X_{ij},t) \in \bT_k$, we have
\begin{equation}
\label{eq_f_coords}
f \circ \Theta_{n-k}(X_{ij},t) =  X_{kk} + \dfrac{t}{X_{1,n-k}} + \sum_{i \in [k] \atop j \in [i+1,i+n-k-1]} \dfrac{X_{ij}}{X_{i,j-1}} + \sum_{i \in [k-1] \atop j \in [i,i+n-k-1]} \dfrac{X_{ij}}{X_{i+1,j+1}}.
\end{equation}
Given this equation, it follows that $\wh{f}(b) \geq 0$ if and only if $b$ satisfies the inequalities listed above.

Let $(M,t) = \Theta_{n-k}(X_{ij},t)$, and let $P_J = P_J(M)$ denote the Pl\"{u}cker coordinates of $M$. To prove \eqref{eq_f_coords}, we will show the following:
\begin{enumerate}
\item $\dfrac{P_{\{k\} \cup [k+2,n]}}{P_{[k+1,n]}} = X_{kk}$
\item $t \dfrac{P_{[2,n-k] \cup \{n\}}}{P_{[1,n-k]}} = \dfrac{t}{X_{1,n-k}}$
\item $\dfrac{P_{[1,r-k] \cup \{r\} \cup [r+2,n]}}{P_{[1,r-k] \cup [r+1,n]}} = \ds \sum_{i = 1}^k \dfrac{X_{i,i+r-k}}{X_{i,i+r-k-1}} \quad \text{ for } r = k+1, \ldots, n-1$
\item $\dfrac{P_{\{r\} \cup [r+2,r+n-k]}}{P_{[r+1,r+n-k]}} = \ds \sum_{j=1}^{n-k} \dfrac{X_{r,r+j-1}}{X_{r+1,r+j}} \quad \text{ for } r = 1, \ldots, k-1$.
\end{enumerate}

By Corollary \ref{cor_pluc_formulae}(2), we have
\[
\ds \dfrac{P_{\{k\} \cup [k+2,n]}}{P_{[k+1,n]}} = \dfrac{X_{kk}}{1} \quad\quad \text{ and } \quad\quad \ds \dfrac{P_{[2,n-k] \cup \{n\}}}{P_{[1,n-k]}} = \dfrac{\ds \prod_{a \in [2,n-k] \cap [k]} X_{a,n-k}}{\ds \prod_{a \in [1,n-k] \cap [k]} X_{a,n-k}} = \dfrac{1}{X_{1,n-k}}
\]
which gives (1) and (2).

For (3), let $J = [1,r-k] \cup \{r\} \cup [r+2,n]$, and let $T$ be a $J$-tableau (see \S \ref{sec_J_tab}). The diagram $D_{J,n-k}$ has $r-k+1$ columns and $\min(r-k,k)$ rows, and the lengths of the first $r-k$ columns are weakly increasing. Since the first entry in the $a^{th}$ row of $T$ must be $a$, the first $r-k$ columns are completely determined. It remains to consider column $r-k+1$, which consists of a single box in the top row. The first $r-k$ boxes in the top row of $T$ are filled with 1, so we may choose any element of $[k]$ for the last column. If we choose $i$, then the weight of $T$ is $x_{i,i+r-k} \ds \prod_{a \in [r-k] \cap [k]} X_{a,r-k}$. By Corollary \ref{cor_pluc_formulae}(2), we have $\ds P_{[1,r-k] \cup [r+1,n]} = \prod_{a \in [r-k] \cap [k]} X_{a,r-k}$. Thus, Lemma \ref{lem_pluc_formula} gives
\[
\dfrac{P_{[1,r-k] \cup \{r\} \cup [r+2,n]}}{P_{[1,r-k] \cup [r+1,n]}} = \sum_{i=1}^k x_{i,i+r-k} = \sum_{i=1}^k \dfrac{X_{i,i+r-k}}{X_{i,i+r-k-1}}.
\]

For (4), let $J = \{r\} \cup [r+2,r+n-k]$, and let $T$ be a $J$-tableau. The diagram $D_{J,n-k}$ has $n-k$ columns and $\min(n-k,k-r)$ rows, and the column lengths are weakly increasing. For $a \geq 2$, the condition $T(a,a) = j_a = r+a$ implies that every entry in the $a^{th}$ row of $T$ must be $r+a$. There is some choice for the first row. The first entry must be $r$, but the other $n-k-1$ entries can be any weakly increasing sequence of $r$'s and $r+1$'s. If the first row of $T$ consists of $r$ repeated $b$ times and $r+1$ repeated $n-k-b$ times (for $1 \leq b \leq n-k$), then
\begin{align*}
\wt(T) &= x_{rr} x_{r,r+1} \cdots x_{r,r+b-1} x_{r+1,r+b+1} \cdots x_{r+1,r+n-k} \prod_{a \in [r+2,r+n-k] \cap [k]} X_{a,r+n-k} \\
&= X_{r,r+b-1} \dfrac{X_{r+1,r+n-k}}{X_{r+1,r+b}} \prod_{a \in [r+2,r+n-k] \cap [k]} X_{a,r+n-k}.
\end{align*}
Thus, using Lemma \ref{lem_pluc_formula} for the numerator and Corollary \ref{cor_pluc_formulae}(2) for the denominator, we have
\[
\dfrac{P_{\{r\} \cup [r+2,r+n-k]}}{P_{[r+1,r+n-k]}} = \dfrac{\ds \sum_{b = 1}^{n-k} X_{r,r+b-1} \dfrac{X_{r+1,r+n-k}}{X_{r+1,r+b}} \ds \prod_{a \in [r+2,r+n-k] \cap [k]} X_{a,r+n-k}}{\ds \prod_{a \in [r+1,r+n-k] \cap [k]} X_{a,r+n-k}} = \ds \sum_{b=1}^{n-k} \dfrac{X_{r,r+b-1}}{X_{r+1,r+b}}.
\]
\end{proof}

\subsubsection{Promotion and the crystal operators}

Recall the cyclic shift map $\PR : \X{n-k} \rightarrow \X{n-k}$ from \S \ref{sec_cyclic_shift}. By definition, every Pl\"{u}cker coordinate of $\PR_t(M)$ is equal to a power of $t$ times a Pl\"{u}cker coordinate of $M$, so $\Theta \PR = \Theta_{n-k}^{-1} \circ \PR \circ \, \Theta_{n-k}$ is subtraction-free, and we may define the piecewise-linear map
\[
\wh{\PR} = \Trop(\Theta \PR) : \tw{\bT}_k \rightarrow \tw{\bT}_k.
\]
The map $\Theta \PR^{-1}$ is also subtraction-free, so we may define $\wh{\PR^{-1}}  = \wh{\PR}^{-1}$ as well.

The following result, which is proved in \S \ref{sec_pf_trop_PR}, is the key tool in this section.

\begin{thm}
\label{thm_trop_PR}
If $b$ is a $k$-rectangle, then $\wh{\PR}(b) = \tw{\pr}(b)$.
\end{thm}

\begin{remark}
The map $\PR$ clearly has order $n$, so Theorem \ref{thm_trop_PR} gives a ``birational'' proof that $\tw{\pr}$ has order $n$ on rectangular tableaux. Grinberg and Roby \cite{GrinRob} used a similar birational technique to prove an equivalent result.
\end{remark}

Now recall the definitions of the rational maps $\gamma, \vp_i, \ve_i,$ and $e_i^c$ from \ref{sec_geom_Gr}. It's clear that $\gamma : \X{n-k} \rightarrow (\Cx)^n$ and $\vp_i, \ve_i : \X{n-k} \rightarrow \Cx$ are subtraction-free in the sense of Lemma \ref{lem_sub_free}(1), so we may define piecewise-linear maps $\wh{\gamma} = \Trop(\Theta \gamma), \wh{\vp}_i = \Trop(\Theta \vp_i),$ and $\wh{\ve}_i = \Trop(\Theta \ve_i)$.

It requires a bit more work to show that $e_i^c$ is subtraction-free.

\begin{lem}
\label{lem_ei_pos}
If $e_i^c(M,t) = (M',t)$, then the Pl\"{u}cker coordinates of $M'$ are given by nonzero subtraction-free rational functions in the Pl\"{u}cker coordinates of $M$ and the variables $c$ and $t$.
\end{lem}

\begin{proof}
By Lemma \ref{lem_pr_conj}(b) and the fact that $\PR$ and $\PR^{-1}$ are subtraction-free, it suffices to prove this for $i = 1$. By definition, $M'$ is obtained from $M$ by adding a scalar multiple of the second row to the first row, so $P_J(M') = P_J(M)$ unless $1 \in J$ and $2 \not \in J$. The only basic $(n-k)$-subset which contains 1 but not 2 is $J_{1,1} = \{1\} \cup [k+2,n]$, and we compute
\begin{align*}
P_{\{1\} \cup [k+2,n]}(M') &= P_{\{1\} \cup [k+2,n]}(M) + \dfrac{c-1}{\vp_1(M,t)} \cdot P_{\{2\} \cup [k+2,n]}(M) \\
&= P_{\{1\} \cup [k+2,n]}(M) + (c-1) \cdot P_{\{1\} \cup [k+2,n]}(M) = cP_{\{1\} \cup [k+2,n]}(M).
\end{align*}
By Corollary \ref{cor_basic=basis}, every Pl\"{u}cker coordinate is a subtraction-free rational function (in fact, Laurent polynomial) in the basic Pl\"{u}cker coordinates, so we are done.
\end{proof}

In light of this Lemma, we may define a piecewise-linear action $\wh{e}_i : \bbZ \times \tw{\bT}_k \rightarrow \tw{\bT}_k$ by
\[
\wh{e}_i(m,b) = (\Trop(\Theta_{n-k}^{-1} \circ e_i^c \circ \Theta_{n-k}))(m,b),
\]
where we replace $c$ with $m$ in the tropicalization.

\begin{thm}
\label{thm_trop_cryst}
Let $b = (B_{ij},L)$ be a $k$-rectangle. Then for $i \in \Zn$, we have
\begin{enumerate}
\item $\wh{\gamma}_i(b) = \tw{\gamma}_i(b)$.
\item $\wh{\vp}_i(b) = -\tw{\vp}_i(b)$ and $\wh{\ve}_i(b) = -\tw{\ve}_i(b)$.
\item $\tw{e}_i(b)$ is defined if and only if $\wh{f}(\wh{e}_i(1,b)) \geq 0$; in that case, $\wh{e}_i(1,b) = \tw{e}_i(b)$.
\item $\tw{f}_i(b)$ is defined if and only if $\wh{f}(\wh{e}_i(-1,b)) \geq 0$; in that case, $\wh{e}_i(-1,b) = \tw{f}_i(b)$.
\end{enumerate}
Note that in (1), $\wh{\gamma}_i(b)$ is the $i$th component of $\wh{\gamma}(b)$, and similarly for $\tw{\gamma}_i(b)$; in (4), $\wh{f}$ is the tropicalization of the decoration, whereas $\tw{f}_i$ is a combinatorial crystal operator.
\end{thm}

\begin{proof} We prove each of these statements for $i = 1$, and then Proposition \ref{prop_cryst_ops_pr}, Lemma \ref{lem_pr_conj}, and Theorem \ref{thm_trop_PR} allow us to conjugate by $\PR$ at the geometric level and $\tw{pr}$ at the combinatorial level to obtain the statements for all values of $i$.

Since $\tw{\gamma}_i(b)$ is the number of $i$'s in the tableau corresponding to $b$, we have $\tw{\gamma}_1 = B_{11}$. Suppose $\gamma \circ \Theta_{n-k}(X_{ij},t) = (\gamma_1, \ldots, \gamma_n)$. By Corollary \ref{cor_pluc_formulae}(2) and the definition of $\gamma$, we have $\gamma_1 = X_{11}$, so $\wh{\gamma}_1(b) = \tw{\gamma}_1(b)$, proving (1).

For (2)-(4), we assume that $k \neq 1, n-1$ to avoid ``boundary effects'' (we leave the cases $k = 1, n-1$ to the reader). Let $(M,t) = \Theta_{n-k}(X_{ij},t)$. By Corollary \ref{cor_pluc_formulae}(2), we have
\begin{equation*}
\vp_1 \circ \Theta_{n-k}(X_{ij},t) = \frac{P_{\{2\} \cup [k+2,n]}(M)}{P_{\{1\} \cup [k+2,n]}(M)} = \dfrac{X_{22}}{X_{11}}
\end{equation*}
and
\begin{equation*}
\ve_1 \circ \Theta_{n-k}(X_{ij},t) = \frac{P_{\{2\} \cup [k+2,n]}(M) P_{\{1\} \cup [k+2,n]}(M)}{P_{[k+1,n]}(M) P_{\{1,2\} \cup [k+3,n]}(M)} = \dfrac{X_{11}}{X_{12}}.
\end{equation*}
Thus, we have $-\wh{\vp}_1(b) = B_{11} - B_{22}$ and $-\wh{\ve}_1(b) = B_{12} - B_{11}$, and (2) follows from comparison with Example \ref{ex_e1}.

Now set $e_1^c(M,t) = (M',t)$ and $(X'_{ij},t) = \Theta_{n-k}^{-1}(M',t)$. By Proposition \ref{prop_GT_pl}, the $X'_{ij}$ depend only on the basic Pl\"{u}cker coordinates of $M'$, and it was shown in the proof of Lemma \ref{lem_ei_pos} that $P_{\{1\} \cup [k+2,n]}(M') = cP_{\{1\} \cup [k+2,n]}(M)$, and all other basic Pl\"{u}cker coordinates of $M$ and $M'$ are the same. Thus, the effect of $\Theta_{n-k}^{-1} \circ e_i^c \circ \Theta_{n-k}$ on $(X_{ij},t)$ is to replace $X_{11}$ with $cX_{11}$, and to leave the other $X_{ij}$ unchanged. This means that $\wh{e}_1(m,b)$ adds $m$ to $B_{11}$. Furthermore, using formula \eqref{eq_f_coords} for $f(M,t)$, we see that if $\wh{f}(b) \geq 0$, then $\wh{f}(\wh{e}_1(1,b)) \geq 0$ if and only if $B_{12} > B_{11}$, and $\wh{f}(\wh{e}_1(-1,b)) \geq 0$ if and only if $B_{11} > B_{22}$.

We saw in Example \ref{ex_e1} that $\tw{e}_1(b)$ is not defined when $B_{12} = B_{11}$, and otherwise $\tw{e}_1(b)$ increases $B_{11}$ by 1; similarly, $\tw{f}_1(b)$ is not defined when $B_{11} = B_{22}$, and otherwise $\tw{f}_1(b)$ decreases $B_{11}$ by 1. This agrees with the description of $\wh{e}_1(\pm 1,b)$ in the previous paragraph, so (3) and (4) are proved.
\end{proof}

\subsection{Examples}
\label{sec_trop_ex}

\subsubsection{One-row tableaux}
\label{sec_one_row}

Let $(X_{11}, \ldots, X_{1,n-1}, t)$ be an element of $\bT_{1}$, and set $x_1 = X_{11}, x_j = X_{1j}/X_{1,j-1}$ for $j = 2, \ldots, n-1$, and $x_n = t/X_{1,n-1}$. We have
\[
\Phi_1(X_{1j},t) = \left(
\begin{array}{ccccccc}
x_1 & \\
1 & x_2 & \\
 & 1 & x_3 & \\
 & & & \ddots \\
 & & & & x_{n-1} \\
 & & & & 1 & x_n
\end{array}
\right).
\]
By definition, $\Theta_{n-1}(X_{1j},t) = (M,t)$, where $M$ is the $(n-1)$-dimensional subspace spanned by the first $n-1$ columns of $\Phi_1(X_{1j},t)$. One easily computes
\[
\dfrac{P_{[1,j] \cup [j+2,n]}(M)}{P_{[2,n]}(M)} = x_1 x_2 \cdots x_j = X_{1j}
\]
for $i = 1, \ldots, n-1$, in agreement with Proposition \ref{prop_GT_pl}.

Set $(M',t) = \PR(M,t)$, and $(X'_{1j},t) = \Theta_{n-1}^{-1}(M',t) = \Theta \PR(X_{1j},t)$. We have
\[
X'_{1j} = \dfrac{P_{[1,j] \cup [j+2,n]}(M')}{P_{[2,n]}(M')} = t \dfrac{P_{[1,j-1] \cup [j+1,n]}(M)}{P_{[1,n-1]}(M)} = t \dfrac{X_{1,j-1}}{X_{1,n-1}} = x_n x_1 \cdots x_{j-1},
\]
and thus in terms of the variables $x_i$, we have
\begin{equation}
\label{eq_PR_one_row}
\Theta \PR(x_1, \ldots, x_n) = (x_n, x_1, \ldots, x_{n-1}).
\end{equation}

Now we compute $\Theta e_0^c$. Set $(M'',t) = e_0^c(M,t)$, and $(X''_{1j},t) = \Theta^{-1}_{n-1}(M'',t) = \Theta e_0^c(X_{1j},t)$. We have $\phi_0(M,t) = \dfrac{x_1 \cdots x_{n-1}}{t}$, so
\[
M'' = x_0\left( \dfrac{(-1)^{n-2}}{t} \dfrac{c-1}{\phi_0(M,t)} \right) \cdot M = x_0\left( (-1)^n \dfrac{c-1}{x_1 \cdots x_{n-1}} \right) \cdot M.
\]
Left-multiplication by $x_0(a)$ means adding $a$ times row 1 to row $n$, so we have
\[
P_{[2,n]}(M'') = P_{[2,n]}(M) + \dfrac{c-1}{x_1 \cdots x_{n-1}} P_{[1,n-1]}(M) = c,
\]
and the other maximal minors of $M''$ are equal to those of $M$. Thus, $X''_{1j} = c^{-1}X_{1j}$ for all $j$, so
\begin{equation}
\label{eq_e0_one_row}
\Theta e_0^c(x_1, \ldots, x_n) = (c^{-1}x_1, x_2, \ldots, x_{n-1}, cx_n).
\end{equation}
Since conjugation by $\PR$ sends $e_i^c$ to $e_{i-1}^c$, \eqref{eq_PR_one_row} and \eqref{eq_e0_one_row} imply that
\begin{equation}
\label{eq_ei_one_row}
\Theta e_i^c(x_1, \ldots, x_n) = (x_1, \ldots, cx_i, c^{-1}x_{i+1}, \ldots, x_n)
\end{equation}
for all $i \in \Zn$ (this can also be computed directly, of course). Thus, we recover the well-known affine geometric crystal structure on $(\Cx)^n$ (\cite{KOTY}). Note that the actions of $\tpr$ and $\te_i$ on a one-row tableau are indeed given by the tropicalizations of \eqref{eq_PR_one_row} and \eqref{eq_ei_one_row}, where $x_i$ is replaced with the number of $i$'s in the tableau (and $c$ is replaced with 1).

\subsubsection{The case $n=4$, $k=2$}
\label{sec_n4_k2}

Let $(X_{11}, X_{12}, X_{22}, X_{23}, t)$ be a rational 2-rectangle. Set $(M,t) = \Theta_2(X_{ij}, t)$, $(M',t) = \PR(M,t)$, and $(X'_{ij},t) = \Theta \PR(X_{ij},t) = \Theta_2^{-1}(M',t)$. We have
\[
M = \left(
\begin{array}{cc}
X_{11} & 0 \\ \smallskip
X_{22} & \dfrac{X_{12}}{X_{11}}X_{22} \\ \smallskip
1 & \dfrac{X_{12}}{X_{11}} + \dfrac{X_{23}}{X_{22}} \\ \smallskip
0 & 1
\end{array}
\right)
\quad\quad \text{ and } \quad\quad
M' = \left(
\begin{array}{cc}
0 & -t \\
X_{11} & 0 \\
X_{22} & \dfrac{X_{12}}{X_{11}}X_{22} \\
1 & \dfrac{X_{12}}{X_{11}} + \dfrac{X_{23}}{X_{22}}
\end{array}
\right),
\]
so Proposition \ref{prop_GT_pl} gives
\begin{equation}
\label{eq_PR_coords}
\begin{array}{ll}
X'_{11} = \dfrac{P_{14}(M')}{P_{34}(M')} = \dfrac{t}{X_{23}} & X'_{12} = \dfrac{P_{12}(M')}{P_{24}(M')} = \dfrac{t X_{11} X_{22}}{X_{11}X_{23} + X_{12}X_{22}} \\
X'_{22} = \dfrac{P_{24}(M')}{P_{34}(M')} = \dfrac{X_{11}X_{23} + X_{12}X_{22}}{X_{22}X_{23}} & X'_{23} = \dfrac{P_{23}(M')}{P_{34}(M')} = \dfrac{X_{12}X_{22}}{X_{23}}. \\
\end{array}
\end{equation}

Now suppose $(B_{11}, B_{12}, B_{22}, B_{23}, L) \in \tw{\bT}_{2}$. Tropicalizing \eqref{eq_PR_coords}, we obtain $\wh{\PR}(B_{ij},L) = (B'_{ij},L)$, where
\begin{equation}
\label{eq_wh_PR}
\begin{array}{l}
B'_{11} = L - B_{23} \\
B'_{12} = L + B_{11} + B_{22} - \min(B_{11} + B_{23}, B_{12} + B_{22}) \\
B'_{22} = \min(B_{11} + B_{23}, B_{12} + B_{22}) - B_{22} - B_{23} \\
B'_{23} = B_{12} + B_{22} - B_{23}.
\end{array}
\end{equation}

We verify that these piecewise-linear formulas agree with the combinatorial rule for $\tw{\pr}$ for a particular tableau. Consider the following 2-row tableau $T$, and its corresponding 2-rectangle:

\begin{equation}
\label{eq_T}
T = \ytableaushort{{1}{1}{2}{2}{2}{3},{2}{3}{3}{4}{4}{4}}
\quad\quad
\longleftrightarrow
\quad\quad
\begin{array}{ccccccccccc}
&&& 2 & \\
&& 5 && 1\\
&6 && 3
\end{array}
\end{equation}
Using either Bender-Knuth involutions or jeu-de-taquin, one computes

\begin{equation}
\label{eq_pr_T}
\tw{\pr}(T) =
\ytableaushort{{1}{1}{1}{2}{3}{3},{2}{3}{3}{4}{4}{4}}
\quad\quad
\longleftrightarrow
\quad\quad
\begin{array}{ccccccccccc}
&&& 3 & \\
&& 4 && 1\\
&6 && 3
\end{array}
\end{equation}
and the reader may verify that the 2-rectangle corresponding to $\tw{\pr}(T)$ agrees with the output of the piecewise-linear formulas \eqref{eq_wh_PR}, in accordance with Theorem \ref{thm_trop_PR}.

\subsection{Proof of Theorem \ref{thm_trop_PR}}
\label{sec_pf_trop_PR}

Recall from \S \ref{sec_prom_evac} that promotion is defined as the composition
\[
\tw{pr} = \tw{\sigma}_1 \cdots \tw{\sigma}_{n-1},
\]
where $\tw{\sigma}_r$ is the $r^{th}$ Bender-Knuth involution. Recall also the piecewise-linear formula for the action of a Bender-Knuth involution on a Gelfand-Tsetlin pattern from Lemma \ref{lem_PLBK}. Our strategy is to ``detropicalize'' this piecewise-linear formula to obtain ``geometric Bender-Knuth involutions,'' and then to show that applying a sequence of these involutions to an element $(X_{ij},t) \in \bT_k$ has the same effect as applying $\Theta_{n-k}^{-1} \circ \PR \circ \, \Theta_{n-k}$ to $(X_{ij},t)$.

Let $(B_{ij}, L) \in \tw{\bT}_k$ be a $k$-rectangle, and let $(B'_{ij}, L) = \tsi_r(B_{ij},L)$. By combining Lemma \ref{lem_PLBK} with the ``embedding'' of a $k$-rectangle into its associated Gelfand-Tsetlin pattern (c.f. \eqref{eq_rect_GT}), we see that
\begin{align*}
B'_{ij} &= \begin{cases}
\twf_{ir}(B_{ij},L) + \twg_{ir}(B_{ij},L) - B_{ir} & \text{ if } j = r \\
B_{ij} &\text{ if } j \neq r
\end{cases} \\
\text{ where } \quad\quad \twf_{ij}(B_{ij},L) &= \begin{cases}
\min(B_{i-1,j-1}, B_{i,j+1}) &\text{ if } i \neq 1 \text{ and } j \neq n-k-1+i \\
B_{i-1,j-1} & \text{ if } i \neq 1 \text{ and } j = n-k-1+i \\
B_{i,j+1} & \text{ if } i = 1 \text{ and } j \neq n-k \\
L & \text{ if } i = 1 \text{ and } j = n-k
\end{cases} \\
\text{ and } \quad\quad \twg_{ij}(B_{ij},L) &= \begin{cases}
\max(B_{i,j-1}, B_{i+1,j+1}) &\text{ if } i \neq k \text{ and } j \neq i \\
B_{i+1,j+1} & \text{ if } i \neq k \text{ and } j = i\\
B_{i,j-1} & \text{ if } i = k \text{ and } j \neq k \\
0 & \text{ if } i = k \text{ and } j = k.
\end{cases}
\end{align*}

Now we naively lift formula this piecewise-linear formula for $\tsi_r$ to a rational map $\sigma_r : \bT_k \rightarrow \bT_k$. That is, given $(X_{ij}, t) \in \bT_k$, define $\sigma_r(X_{ij},t) = (X'_{ij},t)$ by
\begin{align*}
X'_{ij} &= \begin{cases}
f_{ir}(X_{ij},t) \cdot g_{ir}(X_{ij},t) \cdot \dfrac{1}{X_{ir}} &\text{ if } j = r \\
X_{ij} &\text{ if } j \neq r
\end{cases} \\
\text{ where } \quad\quad f_{ij}(X_{ij},t) &= \begin{cases}
X_{i-1,j-1} + X_{i,j+1} &\text{ if } i \neq 1 \text{ and } j \neq n-k-1+i \\
X_{i-1,j-1} & \text{ if } i \neq 1 \text{ and } j = n-k-1+i \\
X_{i,j+1} & \text{ if } i = 1 \text{ and } j \neq n-k \\
t & \text{ if } i = 1 \text{ and } j = n-k
\end{cases} \\
\text{ and } \quad\quad g_{ij}(X_{ij},t) &= \begin{cases}
\dfrac{X_{i,j-1}X_{i+1,j+1}}{X_{i,j-1} + X_{i+1,j+1}} &\text{ if } i \neq k \text{ and } j \neq i \\
X_{i+1,j+1} & \text{ if } i \neq k \text{ and } j = i\\
X_{i,j-1} & \text{ if } i = k \text{ and } j \neq k \\
1 & \text{ if } i = k \text{ and } j = k.
\end{cases}
\end{align*}

Define $\pr : \bT_k \rightarrow \bT_k$ by
\[
\pr = \sigma_1 \circ \sigma_2 \circ \cdots \circ \sigma_{n-1}.
\]
Clearly $\Trop(\pr) = \tw{\pr}$, so to prove Theorem \ref{thm_trop_PR}, it suffices to show that 
\begin{equation}
\label{eq_PR_toprove}
\pr \, \circ \, \Theta_{n-k}^{-1} = \Theta_{n-k}^{-1} \circ \PR
\end{equation}
as rational maps from $\Y{n-k}{n}$ to $\bT_k$. Given $(M,t) \in \Y{n-k}{n}$, define $X_{ij}$ by $(X_{ij}, t) = \Theta_{n-k}^{-1}(M,t)$, and define $X'_{ij}$ by $(X'_{ij}, t) = \Theta_{n-k}^{-1} \circ \PR(M,t)$. Write $P_J = P_J(M)$ for the Pl\"{u}cker coordinates of $M$. By Proposition \ref{prop_GT_pl} and the definition of $\PR$, we have
\[
X_{ij} = \frac{P_{[i,j] \cup [k+j-i+2,n]}}{P_{[i+1,j] \cup [k+j-i+1,n]}}
\quad\quad \text{ and } \quad\quad
X'_{ij} = t^{\delta_{i,1}} \dfrac{P_{[i-1,j-1] \cup [k+j-i+1,n-1]}}{P_{[i,j-1] \cup [k+j-i,n-1]}}.
\]

Set $X^{(n)}_{ij} = X_{ij}$, and for $r = 1, \ldots, n-1$, define $X^{(r)}_{ij}$ by
\begin{equation*}
(X^{(r)}_{ij}, t) = \sigma_r(X^{(r+1)}_{ij},t) = \sigma_r \circ \cdots \circ \sigma_{n-1}(X_{ij},t).
\end{equation*}
In this notation, \eqref{eq_PR_toprove} is the equality $X^{(1)}_{ij} = X'_{ij}$ for all $i,j$. To prove this, we will show by descending induction on $r$ that
\begin{equation}
\label{eq_pr_to_prove}
X^{(r)}_{ij} = X'_{ij} \quad\quad\quad \text{ for } j = r, r+1, \ldots, n-1.
\end{equation}

If $r = n$, then \eqref{eq_pr_to_prove} is vacuously true. So suppose $1 \leq r \leq n-1$. Since $\sigma_a$ only changes entries in the $a$th row of the GT pattern, \eqref{eq_pr_to_prove} holds for $j > r$ by induction, and we need only show that for each $i$, we have
\begin{equation}
\label{eq_fg}
f_{ir}(X^{(r+1)}_{ij},t) \cdot g_{ir}(X^{(r+1)}_{ij},t) \cdot \dfrac{1}{X_{ir}} = X'_{ir}.
\end{equation}
By the induction hypothesis, the ``neighborhood'' of $X_{ir}$ in the GT pattern $X^{(r+1)}_{ij}$ looks like
\[
\begin{array}{ccccccccccc}
 X_{i-1,r-1} && X_{i,r-1}\\
& X_{ir}\\
X'_{i,r+1} && X'_{i+1,r+1} \\
\end{array}.
\]
Note that some or all of the NW, NE, and SE neighbors may be ``missing,'' and the SW neighbor may be $t$. For instance, when $r = n-1$, the SW neighbor is $t$ and the SE neighbor is missing.

We claim that
\begin{equation}
\label{eq_f_ir}
f_{ir}(X^{(r+1)}_{ij},t) = t^{\delta_{i,1}} \frac{P_{[i-1,r-1] \cup [k+r-i+1,n-1]}P_{[i,r] \cup [k+r-i+2,n]}}{P_{[i,r-1] \cup [k+r-i+1,n]}P_{[i,r] \cup [k+r-i+1,n-1]}}
\end{equation}
and
\begin{equation}
\label{eq_g_ir}
g_{ir}(X^{(r+1)}_{ij},t) = \frac{P_{[i,r-1] \cup [k+r-i+1,n]}P_{[i,r] \cup [k+r-i+1,n-1]}}{P_{[i,r-1] \cup [k+r-i,n-1]}P_{[i+1,r] \cup [k+r-i+1,n]}}.
\end{equation}

First we prove \eqref{eq_f_ir}. If $1 < i \leq k$ and $i \leq r < n-k-1+i$, then the NW and SW neighbors of $X_{ir}$ both exist, and we have
\begin{align*}
f_{ir}(X^{(r+1)}_{ij},t) &= X_{i-1,r-1} + X'_{i,r+1} \\
&= \frac{P_{[i-1,r-1] \cup [k+r-i+2,n]}}{P_{[i,r-1] \cup [k+r-i+1,n]}} + \frac{P_{[i-1,r] \cup [k+r-i+2,n-1]}}{P_{[i,r] \cup [k+r-i+1,n-1]}} \\
&= \frac{P_{[i-1,r-1] \cup [k+r-i+2,n]}P_{[i,r] \cup [k+r-i+1,n-1]} + P_{[i,r-1] \cup [k+r-i+1,n]}P_{[i-1,r] \cup [k+r-i+2,n-1]}}{P_{[i,r-1] \cup [k+r-i+1,n]}P_{[i,r] \cup [k+r-i+1,n-1]}} \\
&= \frac{P_{[i-1,r-1] \cup [k+r-i+1,n-1]}P_{[i,r] \cup [k+r-i+2,n]}}{P_{[i,r-1] \cup [k+r-i+1,n]}P_{[i,r] \cup [k+r-i+1,n-1]}}
\end{align*}
where in the last step we apply a three-term Pl\"{u}cker relation (Corollary \ref{cor_3term}) to simplify the numerator. We have verified the ``general case'' of \eqref{eq_f_ir}.

The three ``boundary cases'' of \eqref{eq_f_ir} are straightforward to verify: for instance, if $i = 1$ and $r < n-k$, then
\[
f_{1r}(X^{(r+1)}_{ij},t) = X'_{1,r+1} = t \dfrac{P_{[0,r] \cup [k+r+1,n-1]}}{P_{[1,r] \cup [k+r,n-1]}},
\]
which agrees with the right-hand side of \eqref{eq_f_ir} (recall Convention \ref{conv_pluc}). The other two boundary cases are similar, and are left to the reader.

Now we prove \eqref{eq_g_ir}. If $1 \leq i < k$ and $i < r \leq n-k-1+i$, then the NE and SE neighbors of $X_{ir}$ both exist, and we have

\begin{align*}
g_{ir}(X^{(r+1)}_{ij},t) &= \dfrac{X_{i,j-1}X'_{i+1,j+1}}{X_{i,j-1} + X'_{i+1,j+1}} \\
&= \frac{P_{[i,r-1] \cup [k+r-i+1,n]}P_{[i,r] \cup [k+r-i+1,n-1]}}{P_{[i,r-1] \cup [k+r-i+1,n]}P_{[i+1,r] \cup [k+r-i,n-1]} + P_{[i+1,r-1] \cup [k+r-i,n]}P_{[i,r] \cup [k+r-i+1,n-1]}} \\
&= \frac{P_{[i,r-1] \cup [k+r-i+1,n]}P_{[i,r] \cup [k+r-i+1,n-1]}}{P_{[i,r-1] \cup [k+r-i,n-1]}P_{[i+1,r] \cup [k+r-i+1,n]}}
\end{align*}
where in the last step we apply a three-term Pl\"{u}cker relation (Corollary \ref{cor_3term}) to simplify the denominator. This verifies the ``general case'' of \eqref{eq_g_ir}; we leave the three ``boundary cases'' to the reader.

Finally, observe that the denominator of \eqref{eq_f_ir} is equal to the numerator of \eqref{eq_g_ir}, so we have

\begin{align*}
\quad f_{ir}(X^{(r+1)}_{ij},t) \cdot g_{ir}(X^{(r+1)}_{ij},t) \cdot \dfrac{1}{X_{ir}} &= t^{\delta_{i,1}} \dfrac{P_{[i-1,r-1] \cup [k+r-i+1,n-1]}P_{[i,r] \cup [k+r-i+2,n]}}{P_{[i,r-1] \cup [k+r-i,n-1]}P_{[i+1,r] \cup [k+r-i+1,n]}} \cdot \dfrac{P_{[i+1,r] \cup [k+r-i+1,n]}}{P_{[i,r] \cup [k+r-i+2,n]}} \\
&= t^{\delta_{i,1}} \dfrac{P_{[i-1,r-1] \cup [k+r-i+1,n-1]}}{P_{[i,r-1] \cup [k+r-i,n-1]}} \\
&= X'_{ir}.
\end{align*}
This verifies \eqref{eq_fg} and completes the induction, proving Theorem \ref{thm_trop_PR}.

\section{Unipotent crystals}
\label{sec_unip_cryst}

\subsection{Infinite periodic matrices}
\label{sec_inf_period}

Let $\mathbb{C}((\lp))$ be the field of formal Laurent series in the indeterminate $\lp$, that is, expressions of the form
\[
\sum_{m = m_0}^{\infty} a_m \lp^m
\]
where $m_0$ is an integer, and each $a_m$ is an element of $\mathbb{C}$. Let $M_n[\mathbb{C}((\lp))]$ denote the ring of $n \times n$ matrices with entries in this field. Borrowing the perspective of \cite{LPwhirl}, we will often view such matrices as infinite periodic arrays of complex numbers, as we now explain.

An {\em $n$-periodic matrix} (over $\mathbb{C}$) is a $\mathbb{Z} \times \mathbb{Z}$ array of complex numbers $(X_{ij})_{(i,j) \in \mathbb{Z}}$ such that $X_{ij} = 0$ if $j-i$ is sufficiently large, and $X_{ij} = X_{i+n,j+n}$ for all $i,j$. Say that the entries $X_{ij}$ with $i-j = k$ lie on the {\em $k^{th}$ diagonal} of $X$, or that $k$ indexes this diagonal. Thus, the main diagonal of $X$ is indexed by 0, and higher numbers index lower diagonals. We add these matrices entry-wise, and multiply them using the usual matrix product: if $X = (X_{ij})$ and $Y = (Y_{ij})$, then
\[
(XY)_{ij} = \sum_{k \in \mathbb{Z}} X_{ik}Y_{kj}.
\]
The hypothesis that $X_{ij} = 0$ for $j-i$ sufficiently large ensures that each of these sums is finite, and it's clear that the product of two $n$-periodic matrices is $n$-periodic. Denote the ring of $n$-periodic matrices by $M_n^{\infty}(\mathbb{C})$.

Given a matrix $A = (A_{ij}) \in M_n[\mathbb{C}((\lp))]$, where $A_{ij} = \sum a_m^{i,j} \lp^m$, define an $n$-periodic matrix $X = (X_{ij})$ by\footnote{The definition in \cite{LPwhirl} uses $s-r$ instead of $r-s$. This is equivalent to interchanging $\lp$ and $\lp^{-1}$.}
\[
X_{r n + i, s n + j} = a^{i,j}_{r - s}
\]
for $r,s \in \bbZ$ and $i,j \in [n]$. For example, if $n=2$ and
\[
A = \left(
\begin{array}{cc}
2\lp^{-1} + 3 + 4\lp + 5\lp^2 & \lp^{-1} + 7 + 8\lp \\
-3\lp^{-1} + 1 + \lp^2 & -2\lp^{-1} + 5 + 6\lp
\end{array}
\right)
\]
then
\[
X = \left(
\begin{array}{c|cc|cc|cc|c}
\ddots &&&&&&& \iddots \\ \hline
&3 & 7 & 2 & 1 & 0 & 0 \\
&1 & 5 & -3 & -2 & 0 & 0 \\ \hline
&4 & 8 & 3 & 7 & 2 & 1 \\
&0 & 6 & 1 & 5 & -3 & -2 \\ \hline
&5 & 0 & 4 & 8 & 3 & 7 \\
&1 & 0 & 0 & 6 & 1 & 5 \\ \hline
\iddots &&&&&&& \ddots
\end{array}
\right)
\]
where the row (resp., column) indexed by 1 is the upper-most row (resp., left-most column) whose entries are shown. The vertical and horizontal lines partition the matrix into $2 \times 2$ blocks whose entries are the $m^{th}$ coefficients of the entries of $A$, for some $m$.

It is straightforward to check that the map $A \mapsto X$ is an isomorphism of rings. We will refer to the $n \times n$ matrix $A$ as a {\em folded matrix}, and the $n$-periodic matrix $X$ as an {\em unfolded matrix}. We call $X$ the {\em unfolding} of $A$, and $A$ the {\em folding} of $X$. When it is important to distinguish between folded and unfolded matrices, we will use letters near the beginning of the alphabet for folded matrices, and letters near the end of the alphabet for unfolded matrices.

\subsection{Definition of unipotent crystal}
\label{sec_defn_unip}

Write $\GL_n(\mathbb{C}(\lp))$ for the group of $n \times n$ invertible matrices with entries in $\mathbb{C}(\lp)$, the field of rational functions in the indeterminate $\lp$. Since $\mathbb{C}(\lp)$ is a field, the condition of invertibility is equivalent to the requirement that the determinant be a nonzero rational function. We will refer to $\GL_n(\bbC(\lp))$ as the {\em loop group}, even though this term does not have a fixed meaning in the literature. We call the indeterminate $\lp$ the {\em loop parameter}.

Every rational function in $\lp$ has a Laurent series expansion, so $\GL_n(\bbC(\lp))$ is a subset of $M_n[\mathbb{C}((\lp))]$, and we may talk about the unfoldings of its elements.

In this article, we will work with the submonoid of $\GL_n(\bbC(\lp))$ consisting of matrices whose entries are Laurent polynomials in $\lp$, and whose determinant is a nonzero Laurent polynomial in $\lp$. Call this monoid $G$. The purpose of restricting to this monoid is that it is an ind-variety, so we may talk about rational maps to and from this space. It will be necessary to allow non-constant determinants, so the (minimal) Kac-Moody group used in Nakashima's \cite{Nak} definition of type $A_{n-1}^{(1)}$ unipotent crystal is too small (even the maximal Kac-Moody group only has elements of determinant 1).

Let $B^- \subset G$ be the submonoid of matrices whose unfolding is lower triangular with nonzero entries on the main diagonal. In terms of folded matrices, this means that all entries are (non-Laurent) polynomials, with the entries on the diagonal having nonzero constant term, and the entries above the diagonal having no constant term. $B^-$ is naturally an ind-variety, where the $m^{th}$ piece consists of unfolded matrices which are supported on diagonals $0, \ldots, m$.

For $a \in \bbC$, define
\[
\wh{x}_i(a) = Id + aE_{i,i+1} \quad \text{ for } i \in [n-1], \quad\quad\quad \text{ and } \quad \wh{x}_0(a) = Id + a\lp^{-1}E_{n1},
\]
where $E_{ij}$ is an $n \times n$ matrix unit. For $i \in \bbZ$, set $\wh{x}_i(a) = \wh{x}_{\ov{i}}(a)$, where $\ov{i}$ is the residue of $i$ mod $n$ (in $\{0, \ldots, n-1\}$). Let $U \subset G$ be the subgroup generated by the elements $\wh{x}_i(a)$. Note that the unfoldings of the elements of $U$ are upper triangular with ones on the main diagonal.

The usual definition of unipotent crystals (\cite{BKI}, \cite{Nak}) is based on rational actions of $U$. We work here with a slightly weaker notion.

\begin{defn}
\label{defn_pseudo}
Let $V$ be a complex algebraic (ind-)variety, and let $\alpha : U \times V \rightarrow V$ be a partially-defined map. Let $u.v := \alpha(u,v)$. We will say that $\alpha$ is a {\em pseudo-rational $U$-action} if it satisfies the following properties:
\begin{enumerate}
\item $1.v = v$ for all $v \in V$;
\item If $u.v$ and $u'.(u.v)$ are defined, then $(u'u).v =u'.(u.v)$;
\item For each $i \in \Zn$, the partially defined map from $\bb{C} \times V \rightarrow V$ given by $(a,v) \mapsto \wh{x}_i(a).v$ is rational.
\end{enumerate}
\end{defn}

\begin{remark}
We suspect that it is possible to give $U$ an ind-variety structure so that a pseudo-rational $U$-action is actually a rational $U$-action. The difficulty is that $U$ is \underline{not} the full set of upper triangular unfolded matrices with 1's on the diagonal and folded determinant equal to 1 (it is not possible to generate all the one-parameter subgroups corresponding to positive real roots using only the $\wh{x}_i(a)$). Fortunately, pseudo-rational $U$-actions suffice for our purposes.
\end{remark}

\begin{defn}
Define $\alpha_{B^-} : U \times B^- \rightarrow B^-$ by $u.b = b'$ if $ub = b'u'$, with $b' \in B^-, u' \in U$. If $ub$ does not have such a factorization, then $u.b$ is undefined.
\end{defn}

Note that if $b_1u_1 = b_2u_2$, then $b_2^{-1}b_1 = u_2u_1^{-1}$ is both lower triangular and upper triangular with 1's on the main diagonal (as an unfolded matrix), so it must be the identity matrix, and thus $b_1 = b_2$ and $u_1 = u_2$. This shows that $\alpha_{B^-}$ is well-defined (as a partial map). Observe that if $X \in B^-$ is an unfolded matrix, then
\begin{equation*}
\wh{x}_i(a) \cdot X \cdot \wh{x}_i\left(\frac{-aX_{i+1,i+1}}{X_{ii} + aX_{i+1,i}}\right) \in B^-,
\end{equation*}
so we have
\begin{equation}
\label{eq_pseudo_action}
\wh{x}_i(a).X = \wh{x}_i(a) \cdot X \cdot \wh{x}_i(\tau_i(a,X)) \quad\quad \text{ where } \quad\quad \tau_i(a,X) = \frac{-aX_{i+1,i+1}}{X_{ii} + aX_{i+1,i}}.
\end{equation}
This shows that $\alpha_{B^-}$ satisfies property (3) of Definition \ref{defn_pseudo}. It's clear the first two properties are satisfied as well, so $\alpha_{B^-}$ is a pseudo-rational $U$-action.

\begin{defn}
\label{defn_U_variety}
A {\em $U$-variety} is an irreducible complex algebraic (ind-)variety $X$ together with a pseudo-rational $U$-action $\alpha : U \times X \rightarrow X$. A morphism of $U$-varieties is a rational map which commutes with the $U$-actions (when they are defined).
\end{defn}

\begin{defn}
\label{defn_unip_cryst}
A {\em unipotent crystal} (of type $A_{n-1}^{(1)}$) is a pair $(V,g)$, where $V$ is a $U$-variety, and $g : V \rightarrow B^-$ is a morphism of $U$-varieties, such that for each $i \in \Zn$, the rational function $v \mapsto g(v)_{i+1,i}$ is not identically zero (here $g(v)$ is an unfolded matrix).
\end{defn}

For example, the ind-variety $B^-$ with the pseudo-rational $U$-action $\alpha_{B^-}$ is a $U$-variety, and the pair $(B^-,\Id)$ is (trivially) a unipotent crystal.

\begin{thm}
\label{thm_induces}
Let $(V,g)$ be a unipotent crystal. Suppose $v \in V$, and let $X = g(v)$ be an unfolded matrix. Define
\begin{itemize}
\item $\gamma(v) = (X_{11}, \ldots, X_{nn})$;
\item $\vp_i(v) = \dfrac{X_{i+1,i}}{X_{ii}} \quad \text{ and } \quad \ve_i(x) = \dfrac{X_{i+1,i}}{X_{i+1,i+1}}$;
\item $e_i^c(v) = \wh{x}_i\left(\dfrac{c-1}{\vp_i(v)}\right).v$ (here $.$ is the pseudo-rational action of $U$ on $X$).
\end{itemize}
Then $(V, \gamma, \vp_i, \ve_i, e_i)$ is a type $A_{n-1}^{(1)}$ geometric crystal (Definition \ref{defn_geom_cryst}). We say that this geometric crystal is {\em induced} from the unipotent crystal $(V,g)$.
\end{thm}

\begin{proof}
The rational functions $\ve_i$ and $\vp_i$ are not identically zero due to the assumption about $g(v)_{i+1,i}$ in Definition \ref{defn_unip_cryst}. Property (1) of a geometric pre-crystal (Definition \ref{defn_geom_precryst}) is immediate. Given $v \in V$, set $X = g(v)$, $v' = \wh{x}_i\left(\frac{c-1}{\vp_i(v)}\right).v$, and $X' = g(v')$ (view $X$ and $X'$ as unfolded matrices). Since $g$ is a morphism of $U$-varieties, we have
\[
X' = \wh{x}_i(a) \cdot X \cdot \wh{x}_i(\tau_i(a,X)),
\]
where $a = \frac{c-1}{\vp_i(v)}$. A short computation shows that the principal two-by-two submatrix of $X'$ using rows and columns $i$ and $i+1$ is
\[
\left(
\begin{array}{cc}
cX_{i,i} & 0 \\
X_{i+1,i} & c^{-1}X_{i+1,i+1}
\end{array}
\right),
\]
and the other entries on the main diagonal of $X'$ are equal to those of $X$. This implies properties (2) and (3) of a geometric pre-crystal.

To see that $e_i$ is an action, compute
\begin{align*}
e_i^{c_1}(e_i^{c_2}(v)) &= \wh{x}_i\left(\dfrac{c_1-1}{\vp_i(e_i^{c_2}(v))}\right).\wh{x}_i\left(\dfrac{c_2 - 1}{\vp_i(v)}\right).v \\
&= \wh{x}_i\left(\dfrac{c_1 - 1}{c_2^{-1}\vp_i(v)} + \dfrac{c_2-1}{\vp_i(v)}\right).v = \wh{x}_i\left(\dfrac{c_1c_2 - 1}{\vp_i(v)}\right).v = e_i^{c_1c_2}(v)
\end{align*}
where the second equality uses property (3) of a geometric pre-crystal.

The geometric Serre relations in the case $|i-j| > 1$ follow from the fact that $\wh{x}_i(a)$ and $\wh{x}_j(b)$ commute when $|i-j| > 1$. The proof of the geometric Serre relations in the case $|i-j| = 1$ is a computation inside $GL_3$, which is worked out in \cite[\S 5.2, Proof of Theorem 3.8]{BKI}. 
\end{proof}

The unipotent crystal $(B^-, \Id)$ induces a geometric crystal $(B^-, \gamma_{B^-}, \vp_{i,B^-}, \ve_{i,B^-}, e_{i,B^-})$. If $X \in B^-$ is an unfolded matrix, then a short computation using \eqref{eq_pseudo_action} shows that
\begin{equation}
\label{eq_U_action_B-}
e_{i,B^-}^c(X) = \wh{x}_i\left( \dfrac{c-1}{\vp_i(X)} \right) \cdot X \cdot \wh{x}_i\left( \dfrac{c^{-1} - 1}{\ve_i(X)} \right).
\end{equation}
Since we don't consider any other geometric crystal structures on $B^-$, we will usually drop the subscript $B^-$. Note that if $(V,g)$ is a unipotent crystal with induced geometric crystal $(V,\gamma,\vp_i,\ve_i,e_i)$, then by definition, we have
\begin{equation}
\label{eq_commutes_with_g}
\gamma(v) = \gamma(g(v)) \quad\quad \vp_i(v) = \vp(g(v)) \quad\quad \ve_i(v) = \ve_i(g(v)) \quad\quad g(e_i^c(v)) = e_i^c(g(v)).
\end{equation}

\subsection{Unipotent crystal on the Grassmannian}
\label{sec_unip_Gr}

Given a folded matrix $A \in GL_n(\mathbb{C}(\lp))$ and $z \in \bbC$, let $A|_{\lp = z}$ denote the matrix obtained by evaluating the loop parameter $\lp$ at $z$. This is defined as long as $z$ is not a pole of any entry of $A$, and the resulting matrix is invertible if $z$ is not a root of the determinant of $A$.

Recall that $\X{k} = \Y{k}{n}$. Define a $U$-action $\alpha_k : U \times \X{k} \rightarrow \X{k}$ by
\[
u.(M,t) = (u|_{\lp = (-1)^{k-1} t} \cdot M, t).
\]
Note that $u.(M,t)$ is always defined, since every element of $U$ has Laurent polynomial entries and determinant 1. This action makes $\X{k}$ into a $U$-variety.

\begin{defn}
\label{defn_g}
Define a rational map $g : \X{k} \rightarrow B^-$ by $g(M, t) = A$, where $A$ is the folded matrix defined by
\[
A_{ij} = c_{ij} \dfrac{P_{[j-k+1,j-1] \cup \{i\}}(M)}{P_{[j-k,j-1]}(M)}, \quad\quad\quad
c_{ij}  = 
\begin{cases} 1 & \text{ if } j \leq k \\
t & \text{ if } j > k \text{ and } i \geq j \\
\lp  & \text{ if } j > k \text { and } i < j
\end{cases}
\]
(recall Convention \ref{conv_pluc}).
\end{defn}

For example, if $(M,t) \in \Y{2}{5}$, then setting $P_J = P_J(M)$, we have
\begin{equation}
\label{eq_ex_2_5}
g(M,t) = \left(
\begin{array}{ccccc}
\dfrac{P_{15}}{P_{45}} & 0 & \lp & \lp \dfrac{P_{13}}{P_{23}} & \lp \dfrac{P_{14}}{P_{34}} \smallskip \\
\dfrac{P_{25}}{P_{45}} & \dfrac{P_{12}}{P_{15}} & 0 & \lp & \lp \dfrac{P_{24}}{P_{34}} \smallskip \\
\dfrac{P_{35}}{P_{45}} & \dfrac{P_{13}}{P_{15}} & t \dfrac{P_{23}}{P_{12}} & 0 & \lp \smallskip \\
1 & \dfrac{P_{14}}{P_{15}} & t \dfrac{P_{24}}{P_{12}} & t \dfrac{P_{34}}{P_{23}} & 0 \smallskip \\
0 & 1 & t \dfrac{P_{25}}{P_{12}} & t \dfrac{P_{35}}{P_{23}} & t \dfrac{P_{45}}{P_{34}}
\end{array}
\right).
\end{equation}

If $(M,t) \in \Y{3}{4}$, then setting $P_J = P_J(M)$, we have
\begin{equation}
\label{eq_ex_3_4}
g(M,t) = \left(
\begin{array}{cccc}
\dfrac{P_{134}}{P_{234}} & 0 & 0 & \lp \\
1 & \dfrac{P_{124}}{P_{134}} & 0 & 0 \\
0 & 1 & \dfrac{P_{123}}{P_{124}} & 0 \\
0 & 0 & 1 & t \dfrac{P_{234}}{P_{123}}
\end{array}
\right).
\end{equation}

\begin{remark}
The matrix \eqref{eq_ex_3_4} is a shifted version of the ``whirl'' in \cite{LPwhirl}. It is also related to the type $A_{n-1}^{(1)}$ ``$M$-matrix'' in \cite{KNO2}.
\end{remark}

Define the {\em shift map} $\sh : B^- \rightarrow B^-$ on an unfolded matrix $X$ by
\begin{equation}
\label{eq_sh_defn}
\sh(X)_{ij} = X_{i-1,j-1}.
\end{equation}
By $n$-periodicity, this map has order $n$, and it is the ``unipotent analogue'' of the cyclic shift map $\PR : \X{k} \rightarrow \X{k}$ in the following sense.

\begin{lem}
\label{lem_PR_sh}
We have the identities
\begin{enumerate}
\item $\PR^{-1}(\wh{x}_i(a).\PR(M,t)) = \wh{x}_{i-1}(a).(M,t)$ for $(M,t) \in \X{k}$
\item $\sh^{-1}(\wh{x}_i(a).\sh(X)) = \wh{x}_{i-1}(a).X$ for $X \in B^-$
\item $g \circ \PR = \sh \circ \, g$.
\end{enumerate}
\end{lem}

\begin{proof}
Part (1) is a reformulation of Lemma \ref{lem_pr_u_action} in terms of the $U$-action on $\X{k}$. For part (2), let $X \in B^-$ be an unfolded matrix, and let $X' = \sh(X)$. Since $\sh$ is multiplicative, we have
\begin{align*}
\sh^{-1}(\wh{x}_i(a).X') &= \sh^{-1}(\wh{x}_i(a)) \cdot \sh^{-1}(X') \cdot \sh^{-1}(\wh{x}_i(\tau_i(a,X'))) \\
&= \wh{x}_{i-1}(a) \cdot X \cdot \wh{x}_{i-1}(\tau_{i-1}(a,X)) \\
&= \wh{x}_{i-1}(a).X,
\end{align*}
where the second equality follows from
\[
\tau_i(a,X') = \dfrac{-aX'_{i+1,i+1}}{X'_{ii} + aX'_{i+1,i}} = \dfrac{-aX_{ii}}{X_{i-1,i-1} + aX_{i,i-1}} = \tau_{i-1}(a,X).
\]

Part (3) follows from comparing the entries of the unfolded matrices $X = g(M,t)$ and $X' = g \circ \PR(M,t)$.
\end{proof}

\begin{thm}
\label{thm_Gr_induces}
The pair $(\X{k}, g)$ is a unipotent crystal. This unipotent crystal induces the geometric crystal from Definition \ref{defn_maps}.
\end{thm}

\begin{proof}
It's clear that the maps $\gamma, \vp_i, \ve_i,$ and $e_i$ in Definition \ref{defn_maps} are induced from $(\X{k}, g)$ as in Theorem \ref{thm_induces}. The difficulty is to show that $g$ commutes with the $U$-actions. Since $U$ is generated by $\wh{x}_i(a)$, we need only show that
\begin{equation}
\label{eq_U_morphism}
g(\wh{x}_i(a).x) = \wh{x}_i(a).g(x)
\end{equation}
for all $i$. If we know that \eqref{eq_U_morphism} holds for a particular value of $i$, then Lemma \ref{lem_PR_sh} allows us to deduce that it holds for all $i$, so if suffices to consider the case $i=1$.

Suppose $(M,t) \in \X{k}$ and $a \in \bbC$. Set $(M',t) = \wh{x}_1(a).(M,t)$, and write $P_J = P_J(M)$ and $P'_J = P_J(M')$. By definition, $M' = x_1(a) \cdot M$ is obtained from $M$ by adding $a$ times row 2 to row 1, so we have
\begin{equation}
\label{eq_pluc_change}
P'_J = \begin{cases}
P_J + aP_{(J \backslash \{1\}) \cup \{2\}} &\text{ if } 1 \in J \text{ and } 2 \not \in J \\
P_J &\text{ otherwise}
\end{cases}.
\end{equation}

Set $A = g(M,t)$, $A' = g(M',t)$, and $A'' = \wh{x}_1(a).A$ (view these as folded matrices). We must show that $A' = A''$. By \eqref{eq_pseudo_action}, we have
\[
A'' = \wh{x}_1(a) \cdot A \cdot \wh{x}_1(\tau_1(a,A)).
\]
In words, $A''$ is obtained from $A$ by adding $a$ times row 2 to row 1, and then adding $\tau_1(a,A)$ times column 1 to column 2. Thus, $A''$ and $A$ differ only in the first row and the second column. There are four cases to consider.

\smallskip
\noindent \textbf{Case 1: $i \neq 1, j \neq 2$.}
In this case, $A''_{ij} = A_{ij}$, and by \eqref{eq_pluc_change} and the definition of $g$, we see that $A'_{ij} = A_{ij}$ as well.

\smallskip
\noindent \textbf{Case 2: $i = 1, j = 2$.}
By definition, $A_{12}$ and $A'_{12}$ are equal to $\lp$ if $k = 1$, and 0 otherwise. The quantity $\tau_1(a,A)$ is defined so that $A''_{12}$ has no constant term, so $A''_{12} = \delta_{k,1} \lp$ as well.

\smallskip
\noindent \textbf{Case 3: $i = 1, j \neq 2$.} In this case, we have
\begin{align*}
A''_{1j} = A_{1j} + aA_{2j} = \lp^{1-\delta_{j,1}} \dfrac{P_{\{1\} \cup [j-k+1,j-1]} + aP_{\{2\} \cup [j-k+1,j-1]}}{P_{[j-k,j-1]}} = \lp^{1-\delta_{j,1}} \dfrac{P'_{\{1\} \cup [j-k+1,j-1]}}{P'_{[j-k,j-1]}} = A'_{1j}.
\end{align*}

\noindent \textbf{Case 4: $i \neq 1, j = 2$.} Since the matrix entries $A_{11}, A_{12},$ and $A_{22}$ are constant polynomials, we have $\tau_1(a,A) = \frac{-aA_{22}}{A_{11} + aA_{21}}$. We compute
\begin{align*}
A''_{i2} &= A_{i2} + \tau_1(a,A)A_{i1} \\
&= t^{\delta_{k,1}} \dfrac{P_{\{1,i\} \cup [n-k+3,n]}}{P_{\{1\} \cup [n-k+2,n]}} + \frac{-at^{\delta_{k,1}} \dfrac{P_{\{1,2\} \cup [n-k+3,n]}}{P_{\{1\} \cup [n-k+2,n]}}}{\dfrac{P_{\{1\} \cup [n-k+2,n]}}{P_{[n-k+1,n]}} + a \dfrac{P_{\{2\} \cup [n-k+2,n]}}{P_{[n-k+1,n]}}}\dfrac{P_{\{i\} \cup [n-k+2,n]}}{P_{[n-k+1,n]}} \\
&= t^{\delta_{k,1}} \dfrac{P_{\{1,i\} \cup [n-k+3,n]}(P_{\{1\} \cup [n-k+2,n]} + aP_{\{2\} \cup [n-k+2,n]}) - aP_{\{1,2\} \cup [n-k+3,n]} P_{\{i\} \cup [n-k+2,n]}}{P_{\{1\} \cup [n-k+2,n]}(P_{\{1\} \cup [n-k+2,n]} + aP_{\{2\} \cup [n-k+2,n]})}.
\end{align*}
If $k > 1$ and $i \leq n-k+2$, apply a three-term Pl\"{u}cker relation (Corollary \ref{cor_3term}) to the terms in the numerator containing $a$ to obtain
\[
A''_{i2} = \dfrac{P_{\{1,i\} \cup [n-k+3,n]} + aP_{\{2,i\} \cup [n-k+3,n]}}{P_{\{1\} \cup [n-k+2,n]} + aP_{\{2\} \cup [n-k+2,n]}} = \dfrac{P'_{\{1,i\} \cup [n-k+3,n]}}{P'_{\{1\} \cup [n-k+2,n]}} = A'_{i2}.
\]
If $i > n-k+2$, then $A_{i1} = A_{i2} = A''_{i2} = A'_{i2} = 0$, and if $k = 1$, then we have
\[
A''_{i2} = t\dfrac{P_i(P_1 + aP_2) - aP_2P_i}{P_1(P_1+aP_2)} = t\dfrac{P_i}{P_1 + aP_2} = t\dfrac{P'_i}{P'_1} = A'_{i2}.
\]

We have shown that $A' = A''$, so we are done.
\end{proof}

To complete the proof of Theorem \ref{thm_is_geom_cryst}, we must show that the function $f$ (Definition \ref{defn_maps}(4)) is a decoration. Say that an unfolded $n$-periodic matrix $X$ is {\em $\ell$-shifted unipotent} if $X_{ij} = 0$ when $i-j > \ell$, and $X_{ij} = 1$ when $i-j = \ell$. If $X$ is $\ell$-shifted unipotent, define
\[
\chi(X) = \sum_{j = 1}^n X_{j+\ell-1,j}.
\]
It's easy to see that if $X$ is $\ell$-shifted unipotent and $Y$ is $\ell'$-shifted unipotent, then $XY$ is $(\ell+\ell')$-shifted unipotent, and $\chi(XY) = \chi(X)\chi(Y)$.

If $(M,t) \in \X{k}$, then $g(M,t)$ is $(n-k)$-shifted unipotent, and comparing the definitions of $g(M,t)$ and $f(M,t)$, we see that
\[
f(M,t) = \chi(g(M,t)).
\]
For example, if $(M,t) \in \Y{2}{5}$, then the folded version of $g(M,t)$ is shown above in \eqref{eq_ex_2_5}, and we have
\[
\chi(g(M,t)) = \frac{P_{35}}{P_{45}} + \frac{P_{14}}{P_{15}} + t \frac{P_{25}}{P_{12}} + \frac{P_{13}}{P_{23}} + \frac{P_{24}}{P_{34}} = f(M,t).
\]

Using \eqref{eq_U_action_B-}, \eqref{eq_commutes_with_g}, and the fact that $\wh{x}_i(a)$ is $0$-shifted unipotent with $\chi(\wh{x}_i(a)) = a$, we compute
\begin{align*}
f(e_i^c(M,t)) = \chi(g(e_i^c(M,t))) &= \chi(e_i^c(g(M,t))) \\
&= \chi\left(\wh{x}_i\left(\dfrac{c-1}{\vp_i(M,t)}\right) \cdot g(M,t) \cdot \wh{x}_i\left(\dfrac{c^{-1}-1}{\ve_i(M,t)}\right) \right) \\
&= \dfrac{c-1}{\vp_i(M,t)} + f(M,t) + \dfrac{c^{-1} - 1}{\ve_i(M,t)},
\end{align*}
so $f$ is indeed a decoration.

\subsection{Properties of the matrix $g(M,t)$}
\label{sec_g_props}

Recall from \S \ref{sec_basic_pluc} the open positroid cell $X_k^\circ \subset \Gr(k,n)$ where the cyclic Pl\"{u}cker coordinates are nonzero. Note that $g(M,t)$ is defined if and only if $M \in X_k^\circ$.

\begin{prop}
\label{prop_g_props}
Suppose $(M,t) \in \X{k}$, with $M \in X_k^\circ$. Let $A = g(M,t)$ (viewed as a folded matrix).
\begin{enumerate}
\item \label{itm:pi_g}
The first $k$ columns of $A$ span the subspace $M$.
\item \label{itm:rank_k}
The matrix $A|_{\lp = (-1)^{k-1}t}$ obtained from $A$ by evaluating $\lp$ at $(-1)^{k-1}t$ has rank $k$.
\item \label{itm:det_g}
The determinant of $A$ is $(t + (-1)^k \lp)^{n-k}$.
\item \label{itm:g_0}
If $(M,t) = \Theta_k(X_{ij},t)$, then $A|_{\lp = 0} = \Phi_{n-k}(X_{ij},t)$ (see Definition \ref{defn_Phi}).
\end{enumerate}
\end{prop}

\begin{proof}
By Lemma \ref{lem_diag_form}, the subspace $M$ has a diagonal form representative $M^\circ$. It follows from the definition of diagonal form that
\[
M^\circ_{ij} = \frac{P_{[1,j-1] \cup \{i\} \cup [n-k+j+1,n]}(M)}{P_{[1,j-1] \cup [n-k+j,n]}(M)} = \frac{P_{[j-k+1,j-1] \cup \{i\}}(M)}{P_{[j-k,j-1]}(M)}
\]
for $j = 1, \ldots, k$ and $i \in [j, j+n-k]$, and $M^\circ_{ij} = 0$ otherwise. Comparing with the definition of the map $g$, we see that $M^\circ$ is equal to the the first $k$ columns of $A$, which proves (1).

For (2), set $A_t = A|_{\lp = (-1)^{k-1}t}$. We claim that $\Delta_{I, [1,k] \cup \{j\}}(A_t) = 0$ for all $(k+1)$-subsets $I \subset [n]$, and $j \in [k+1,n]$. To see this, suppose $I = \{i_1 < \cdots < i_{k+1}\}$, and expand the determinant along column $j$:
\begin{equation}
\label{eq_Delta_I_k_j}
\Delta_{I, [1,k] \cup \{j\}}(A_t) = \sum_{r = 1}^{k+1} (-1)^{k+1+r} (A_t)_{i_r,j} \Delta_{I \backslash \{i_r\},[1,k]}(A_t).
\end{equation}
By part (1) (and the fact that $\Delta_{[n-k+1,n],[1,k]}(A_t) = 1$), we have
\[
\Delta_{I \backslash \{i_r\},[1,k]}(A_t) = \dfrac{P_{I \backslash \{i_r\}}(M)}{P_{[n-k+1,n]}(M)}.
\]
By the definition of $g$, we have
\begin{align*}
(A_t)_{i_r,j} &= \begin{cases}
t \dfrac{P_{[j-k+1,j-1] \cup \{i_r\}}(M)}{P_{[j-k,j-1]}(M)} & \text{ if } i_r \geq j \smallskip \\
(-1)^{k-1} t \dfrac{P_{[j-k+1,j-1] \cup \{i_r\}}(M)}{P_{[j-k,j-1]}(M)} & \text{ if } i_r < j
\end{cases} \\
&= t \dfrac{P_{<j-k+1,j-k+2, \ldots, j-1,i_r>}(M)}{P_{[j-k,j-1]}(M)}
\end{align*}
where in the last line, the angle brackets indicate that we are taking the columns inside the brackets in the order in which they appear in the sequence, rather than sorting them in increasing order (see Convention \ref{conv_pluc}). Now \eqref{eq_Delta_I_k_j} becomes
\begin{equation*}
\Delta_{I, [1,k] \cup \{j\}}(A_t) = \sum_{r=1}^{k+1} (-1)^{k+1+r} t \dfrac{P_{<j-k+1,j-k+2, \ldots, j-1,i_r>}(M)}{P_{[j-k,j-1]}(M)} \dfrac{P_{I \backslash \{i_r\}}(M)}{P_{[n-k+1,n]}(M)} = 0
\end{equation*}
by the Grassmann-Pl\"{u}cker relations (Proposition \ref{prop_Gr_Pl_rel}).

We have shown that each of the last $n-k$ columns of $A_t$ is in the span of the first $k$, and since the first $k$ columns have rank $k$ by part (1), this proves (2).

For (3), let $A_t$ be as above. By part (2), it is possible to add linear combinations of the first $k$ columns of $A_t$ to the last $n-k$ columns to obtain a matrix with zeroes in the last $n-k$ columns. Let $A'$ be the matrix obtained by adding the same linear combinations of the first $k$ columns of $A$ (which are equal to the first $k$ columns of $A_t$) to the last $n-k$ columns of $A$. Then we have
\[
(A')_{ij} = c'_{ij} \dfrac{P_{[j-k+1,j-1] \cup \{i\}}(M)}{P_{[j-k,j-1]}(M)}, \quad\quad\quad
c'_{ij} = \begin{cases}
1 & \text{ if } j \leq k \\
(-1)^k t + \lp & \text{ if } i \leq n-k \text{ and } j > i \\
0 & \text{ otherwise.}
\end{cases}
\]
For example, if $n = 5$ and $k = 2$, then $A'$ is of the form
\[
\left(
\begin{array}{ccccc}
* & 0 & t + \lp & (t + \lp)* & (t + \lp)* \\
* & * & 0 & t+\lp & (t + \lp)* \\
* & * & 0 & 0 & t+\lp \\
1 & * & 0 & 0 & 0 \\
0 & 1 & 0 & 0 & 0
\end{array}
\right)
\]
where the $*$'s are certain ratios of Pl\"{u}cker coordinates. Thus, we have
\[
\det(A) = \det(A') = (-1)^{k(n-k)}((-1)^k t + \lp)^{n-k} = (t + (-1)^k \lp)^{n-k},
\]
proving (3).

Now we prove (4). By definition, the first $k$ columns of $\Phi_{n-k}(X_{ij},t)$ span the subspace $M$. By (1), the first $k$ columns of $A = g(M,t)$, (which are independent of $\lp$) also span $M$. Since the first $k$ columns of both matrices are diagonal form, they are equal.

Both $A|_{\lp = 0}$ and $\Phi_{n-k}(X_{ij},t)$ are lower triangular, so it remains to consider the entries in positions $(i,j)$, with $k < j \leq i$. First suppose $j = i$. Consider the network for $\Phi_{n-k}(X_{ij},t)$ (the case $n=5, k=2$ is shown in the proof of Lemma \ref{lem_Nk_represents}). There is a single path from source $j$ to sink $j'$, and this path has weight
\[
\dfrac{t}{X_{j-k,j-1}} \prod_{i \in [j-k+1,j] \cap [1,n-k]} \dfrac{X_{ij}}{X_{i,j-1}} = t \dfrac{P_{[j-k+1,j]}(M)}{P_{[j-k,j-1]}(M)}
\]
by Corollary \ref{cor_pluc_formulae}(2). This shows that
\begin{equation}
\label{eq_diag_entries}
\Phi_{n-k}(X_{ij},t)_{jj} = t \frac{P_{[j-k+1,j]}(M)}{P_{[j-k,j-1]}(M)} = A_{jj}.
\end{equation}

Now suppose $j < i$. We claim that
\begin{equation}
\label{eq_k+1_minor}
\Delta_{[j-k+1,j] \cup \{i\}, [1,k] \cup \{j\}}(\Phi_{n-k}(X_{ij},t)) = 0.
\end{equation}
To see this, again consider the network. There is exactly one vertex-disjoint family of paths from $[j-k+1,j]$ to $[1,k]$, and the $k$th path in this family ``blocks off'' the only access to the sink $j'$, so for any $i > j$, there is no way to add a path from $i$ to $j'$ which is vertex-disjoint from the other $k$ paths. Thus, the determinant is zero by the Lindstr\"{o}m Lemma.

The only nonzero entries in the $j$th column of the submatrix $\Phi_{n-k}(X_{ij},t)_{[j-k+1,j] \cup \{i\}, [1,k] \cup \{j\}}$ are in rows $j$ and $i$. Expand the determinant of this submatrix along the $j$th column and use \eqref{eq_diag_entries}, \eqref{eq_k+1_minor}, and the fact that the first $k$ columns of $\Phi_{n-k}(X_{ij},t)$ are the diagonal form representative of $M$ to get
\begin{align*}
0 &= \Phi_{n-k}(X_{ij},t)_{ij} \frac{P_{[j-k+1,j]}(M)}{P_{[n-k+1,n]}(M)} - \Phi_{n-k}(X_{ij},t)_{jj} \frac{P_{[j-k+1,j-1] \cup \{i\}}(M)}{P_{[n-k+1,n]}(M)} \\
&= \Phi_{n-k}(X_{ij},t)_{ij} \frac{P_{[j-k+1,j]}(M)}{P_{[n-k+1,n]}(M)} - t \frac{P_{[j-k+1,j]}(M)}{P_{[j-k,j-1]}(M)} \frac{P_{[j-k+1,j-1] \cup \{i\}}(M)}{P_{[n-k+1,n]}(M)}.
\end{align*}
This shows that $\Phi_{n-k}(X_{ij},t)_{ij} = t \dfrac{P_{[j-k+1,j-1] \cup \{i\}}(M)}{P_{[j-k,j-1]}(M)} = A_{ij}$, completing the proof.
\end{proof}

\section{Symmetries}
\label{sec_symm}

Throughout this section we write $g_t(M) = g(M,t)$.

\subsection{Geometric Sch\"{u}tzenberger involution}
\label{sec_schutz}

For $s \in \bbC$, define $\pi^k_s : B^- \rightarrow \Gr(k,n)$ by $\pi^k_s(A) = M$, where $M$ is the subspace spanned by the first $k$ columns of the folded matrix $A_s = A|_{\lp = (-1)^{k-1} s}$. This map is undefined if the first $k$ columns of $A_s$ do not have full rank. Proposition \ref{prop_g_props}\eqref{itm:pi_g} states that for $M \in X_k^\circ$ and $t \in \Cx$, we have
\begin{equation}
\label{eq_pi_g}
\pi^k_t \circ g_t(M) = M.
\end{equation}
This shows that the matrix $A = g_t(M)$ is determined by the subspace spanned by its first $k$ columns (and the value of $t$). Now we consider what happens if we ``project'' onto the last $k$ rows instead of the first $k$ columns.

Define the map $\fl : B^- \rightarrow B^-$ on a folded matrix $A$ by
\begin{equation}
\label{eq_fl_defn}
\fl(A)_{ij} = A_{n-j+1,n-i+1}.
\end{equation}
In words, $\fl$ reflects the folded matrix over the anti-diagonal. It's easy to see that $\fl$ is an anti-automorphism, and that it satisfies
\begin{equation}
\label{eq_sh_fl}
\fl^2 = \Id \quad\quad\quad\quad\quad \fl \circ \sh = \sh^{-1} \circ \fl,
\end{equation}
where $\sh$ is the shift map defined by \eqref{eq_sh_defn}.

\begin{defn}
\label{defn_S}
Define the {\em geometric Sch\"{u}tzenberger involution} $S : \X{k} \rightarrow \X{k}$ by $S(M,t) = (M',t)$, where
\[
M' = \pi^k_t \circ \fl \circ \, g_t(M).
\]
This is a rational map which is defined for $M$ in the open positroid cell $X_k^\circ$. Continuing the notation of previous sections, we write $S_t$ to denote the map $M \mapsto M'$. Note that the Pl\"{u}cker coordinates of $M'$ are given by
\begin{equation}
\label{eq_pluc_M'}
\frac{P_J(M')}{P_{[n-k+1,n]}(M')} = \Delta_{[n-k+1,n],w_0(J)}(g_t(M)),
\end{equation}
where $w_0(J)$ is the subset obtained from $J$ by replacing each $i \in J$ with $n-i+1$.
\end{defn}

For example, if $(M,t) \in \Y{2}{5}$, then setting $P_J = P_J(M)$, we have (c.f. \eqref{eq_ex_2_5})
\[
S_t(M) = \left(
\begin{array}{cc}
t \dfrac{P_{45}}{P_{34}} & 0 \smallskip \\
t \dfrac{P_{35}}{P_{23}} & t \dfrac{P_{34}}{P_{23}} \smallskip \\
t \dfrac{P_{25}}{P_{12}} & t \dfrac{P_{24}}{P_{12}} \smallskip \\
1 & \dfrac{P_{14}}{P_{15}} \smallskip \\
0 & 1
\end{array}
\right).
\]

Recall from \S \ref{sec_basic_pluc} the basic subsets $J_{i,j} = [i,j] \cup [n-k+j-i+2,n]$, and the open subset $U_k \subset \Gr(k,n)$ where the basic Pl\"{u}cker coordinates don't vanish. There is a simple expression for the basic Pl\"{u}cker coordinates of $S_t(M)$ in terms of those of $M$.

\begin{lem}
\label{lem_basic_pluc_S}
Suppose $M \in \Gr(k,n)$, and $M' = S_t(M)$. If $M \in U_k$, then so is $M'$, and the basic Pl\"{u}cker coordinates of $M'$ are given by
\begin{equation}
\label{eq_basic_pluc_S}
\dfrac{P_{J_{i,j}}(M')}{P_{[n-k+1,n]}(M')} = t^{\min(j,n-k) - i+1} \dfrac{P_{J_{n-k-i+2,n-j}}(M)}{P_{[n-j-k+1,n-j]}(M)}.
\end{equation}
\end{lem}

\begin{proof}
Set $A = g_t(M)$ and $A' = \fl(A)$. By the definition of $S$ and the fact that $\Delta_{[n-k+1,n],[k]}(A) = 1$, we have
\[
\dfrac{P_{J_{i,j}}(M')}{P_{[n-k+1,n]}(M')} = \Delta_{J_{i,j}, [k]}(A') = \Delta_{[n-k+1,n], w_0(J_{i,j})}(A).
\]
Set $a = k-j+i-1$ and $b = n-i-k+1$, so that $w_0(J_{i,j}) = [1,a] \cup [b+a+1,b+k]$. Consider the $k \times k$ submatrix of $A$ using the rows $[b+1,b+k]$ and the columns $[1,a] \cup [b+a+1,b+k]$. The last $k-a$ columns of this submatrix consist of $a$ rows of zeroes followed by a lower triangular $(k-a) \times (k-a)$ block, so we have
\begin{align*}
\Delta_{[b+1,b+k], [1,a] \cup [b+a+1,b+k]}(A) &= \Delta_{[b+1,b+a],[1,a]}(A) \prod_{r = b+a+1}^{b+k} A_{rr} \\
&= \Delta_{[b+1,b+a],[1,a]}(A) \prod_{r = b+a+1}^{b+k} t^{c_r} \dfrac{P_{[r-k+1,r]}(M)}{P_{[r-k,r-1]}(M)}
\end{align*}
where $c_r = 0$ if $r \leq k$, and $c_r = 1$ if $r > k$. Using \eqref{eq_ij_minor} and canceling terms in the product, we obtain
\begin{equation}
\label{eq_minor_of_A}
\Delta_{[b+1,b+k],[1,a] \cup [b+a+1,b+k]}(A) = t^{\min(k-a,b)} \dfrac{P_{J_{b+1,b+a}}(M)}{P_{[n-k+1,n]}(M)} \dfrac{P_{[b+1,b+k]}(M)}{P_{[b+a+1-k,b+a]}(M)}.
\end{equation}
Note that the Pl\"{u}cker coordinates appearing here are nonzero because $M \in U_k$.

Let $A_t = A|_{\lp = (-1)^{k-1} t}$. By Proposition \ref{prop_g_props}, all the columns of $A_t$ are in the span of the first $k$ columns (which is $M$), so if a single minor using a given set of $k$ columns is nonzero, then those $k$ columns also span the subspace $M$. Thus, we have
\begin{equation}
\label{eq_pluc_ratios}
\dfrac{\Delta_{[n-k+1,n], [1,a] \cup [b+a+1,b+k]}(A_t)}{\Delta_{[b+1,b+k], [1,a] \cup [b+a+1,b+k]}(A_t)} = \dfrac{P_{[n-k+1,n]}(M)}{P_{[b+1,b+k]}(M)}.
\end{equation}
The minors appearing in the left-hand side of \eqref{eq_pluc_ratios} don't depend on $\lp$, so this equation still holds if we replace $A_t$ with $A$ (and \eqref{eq_minor_of_A} shows that the denominator of the left-hand side is nonzero). The lemma follows from combining \eqref{eq_minor_of_A} and \eqref{eq_pluc_ratios}, and replacing $a,b$ with $k-j+i-1,n-i-k+1$, respectively.
\end{proof}

Specializing \eqref{eq_basic_pluc_S} to the case of cyclic Pl\"{u}cker coordinates, we get
\begin{equation}
\label{eq_cyclic_pluc_S}
\dfrac{P_{[i,i+k-1]}(M')}{P_{[n-k+1,n]}(M')} = t^{|[i,i+k-1] \cap [n-k]|} \dfrac{P_{[n-k+1,n]}(M)}{P_{[n-i-2k+2,n-i-k+1]}(M)}
\end{equation}
for $i \in \Zn$.
This shows that when $M \in X_k^\circ$, $S_t(M) \in X_k^\circ$ as well, so $g_t \circ S_t(M)$ is defined.

\begin{prop}
\label{prop_pi_g_S}
For $M \in X_k^\circ$ and $t \in \Cx$, we have
\[
g_t \circ S_t(M) = \fl \circ \, g_t(M).
\]
\end{prop}

\begin{proof}
Set $A = g_t(M), A' = \fl(A)$, and $M' = \pi^k_t(A') = S_t(M)$ (view $A$ and $A'$ as folded matrices). We must show that
\begin{equation}
\label{eq_A'_x'}
g_t(M') = A'.
\end{equation}
By definition, the first $k$ columns of $A'$ are the diagonal form representative of $M'$ (note that the first $k$ columns of $A'$ do not depend on $\lp$), so arguing as in the proof of Proposition \ref{prop_g_props}\eqref{itm:pi_g}, we see that the first $k$ columns of $g_t(M')$ are equal to the first $k$ columns of $A'$. It remains to consider the last $n-k$ columns.

We claim that for $i \leq n-k$, we have
\begin{equation}
\label{eq_to_prove_pi_g_S}
A_{ij} = d_{ij} \dfrac{\Delta_{[n-k+1,n], \{j\} \cup [i+1, i+k-1]}(A)}{\Delta_{[n-k+1,n], [i+1,i+k]}(A)}, \quad\quad\quad
d_{ij} = \begin{cases}
t & \text{ if } j \leq i \\
\lp & \text{ if } j > i.
\end{cases}
\end{equation}
This is clearly true when $j \in [i+1,i+k]$. Let $A_t = g_t(M)|_{\lp = (-1)^{k-1} t}$. By Proposition \ref{prop_g_props}\eqref{itm:rank_k}, all $(k+1)$-minors of $A_t$ vanish. For $j \leq i$, expand the minor $\Delta_{\{i\} \cup [n-k+1,n], \{j\} \cup [i+1,i+k]}(A_t)$ along row $i$ and use the fact that $(A_t)_{ir} = 0$ for $r = i+1, \ldots, i+k-1$, and $(A_t)_{i,i+k} = (-1)^{k-1} t$ to obtain
\begin{equation}
\label{eq_A_t_k+1_small_j}
(A_t)_{ij} \Delta_{[n-k+1,n],[i+1,i+k]}(A_t) - t \Delta_{[n-k+1,n], \{j\} \cup [i+1, i+k-1]}(A_t) = 0.
\end{equation}
There are no $\lp$'s in the last $k$ rows of $A$, and $(A_t)_{ij} = A_{ij}$ for $j \leq i$, so we may replace $A_t$ by $A$ in \eqref{eq_A_t_k+1_small_j}. By \eqref{eq_pluc_M'} and \eqref{eq_cyclic_pluc_S}, the minor $\Delta_{[n-k+1,n],[i+1,i+k]}(A)$ is nonzero, so \eqref{eq_A_t_k+1_small_j} implies the $j \leq i$ case of \eqref{eq_to_prove_pi_g_S}.

For $j > i+k$, the same reasoning gives
\begin{equation*}
\label{eq_A_t_k+1_big_j}
t \Delta_{[n-k+1,n], [i+1, i+k-1] \cup \{j\}}(A_t) + (-1)^k (A_t)_{ij} \Delta_{[n-k+1,n],[i+1,i+k]}(A_t) = 0,
\end{equation*}
and since $(A_t)_{ij} = \dfrac{(-1)^{k-1} t}{\lp} A_{ij}$ when $j > i+k$, \eqref{eq_to_prove_pi_g_S} holds in this case as well.

Now \eqref{eq_pluc_M'} and \eqref{eq_to_prove_pi_g_S} imply that
\[
A'_{ij} = A_{n-j+1,n-i+1} = d_{n-j+1,n-i+1} \dfrac{P_{\{i\} \cup [j-k+1, j-1]}(M')}{P_{[j-k,j-1]}(M')} = g_t(M')_{ij}
\]
for $j \geq k+1$, which completes the proof.
\end{proof}

\begin{cor}
\label{cor_S}
The map $S$ has the following properties:
\begin{enumerate}
\item $S^2 = \Id$
\item $S \circ \PR = \PR^{-1} \circ \, S$
\item $\vp_i \circ S = \ve_{n-i}$ and $\ve_i \circ S = \vp_{n-i}$
\item $S \circ e_i^c = e_{n-i}^{c^{-1}} \circ S$.
\end{enumerate}
\end{cor}

\begin{proof}
Throughout the proof, fix $M \in X_k^\circ$ and $t \in \Cx$. By \eqref{eq_pi_g}, \eqref{eq_sh_fl}, and Proposition \ref{prop_pi_g_S}, we have
\[
S_t^2 = \pi^k_t \circ \fl \circ \, g_t \circ S_t = \pi^k_t \circ \fl \circ \fl \circ \, g_t = \pi^k_t \circ \, g_t = \Id,
\]
which proves (1).

For (2), use Lemma \ref{lem_PR_sh}(3), \eqref{eq_sh_fl}, and Proposition \ref{prop_pi_g_S} to compute
\begin{align*}
g_t \circ S_t \circ \PR_t &= \fl \circ \, g_t \circ \PR_t = \fl \circ \sh \circ g_t \\
&= \sh^{-1} \circ \fl \circ \, g_t = \sh^{-1} \circ \, g_t \circ S_t = g_t \circ \PR^{-1} \circ \, S_t.
\end{align*}
Applying $\pi^k_t$ to both sides and using \eqref{eq_pi_g} gives (2).

For (3), suppose $X \in B^-$ is an unfolded matrix. Due to $n$-periodicity, $\fl$ acts on unfolded matrices by $\fl(X)_{ij} = X_{n-j+1,n-i+1}$, so we have
\begin{equation}
\label{eq_vp_fl}
\vp_i(\fl(X)) = \dfrac{\fl(X)_{i+1,i}}{\fl(X)_{ii}} = \dfrac{X_{n-i+1,n-i}}{X_{n-i+1,n-i+1}} = \ve_{n-i}(X).
\end{equation}
Combining this with Proposition \ref{prop_pi_g_S} and \eqref{eq_commutes_with_g}, we obtain
\begin{align*}
\vp_i \circ S(M,t) = \vp_i \circ g \circ S(M,t) = \vp_i \circ \fl \circ \, g(M,t) = \ve_{n-i} \circ g(M,t) = \ve_{n-i}(M,t),
\end{align*}
proving the first half of (3). The second half of (3) is equivalent since $S$ is an involution.

For (4), suppose $X \in B^-$, and set $X' = \fl(X)$. By \eqref{eq_U_action_B-}, \eqref{eq_vp_fl}, and the fact that $\fl$ is an anti-automorphism which maps $\wh{x}_i(a)$ to $\wh{x}_{n-i}(a)$, we have
\begin{align*}
\fl(e_i^c (X)) &= \fl \left(\wh{x}_i\left(\frac{c-1}{\vp_i(X)}\right) \cdot X \cdot \wh{x}_i\left(\frac{c^{-1} - 1}{\ve_i(X)}\right) \right) \\
&= \wh{x}_{n-i}\left(\frac{c^{-1}-1}{\ve_i(X)}\right) \cdot X' \cdot \wh{x}_{n-i}\left(\frac{c - 1}{\vp_i(X)}\right) \\
&= \wh{x}_{n-i}\left(\frac{c^{-1}-1}{\vp_{n-i}(X')}\right) \cdot X' \cdot \wh{x}_{n-i}\left(\frac{c - 1}{\ve_{n-i}(X')}\right) = e_{n-i}^{c^{-1}}(\fl (X)).
\end{align*}
Since $e_i^c$ commutes with $g$ by \eqref{eq_commutes_with_g}, (4) is now proved in the same manner as (2).
\end{proof}

\subsection{The dual Grassmannian}
\label{sec_duality}

Given a subspace $W \subset \mathbb{C}^n$, let $W^\perp$ be the orthogonal complement of $W$ with respect to the non-degenerate bilinear form given by $\langle v_i, v_j \rangle = (-1)^{i+1} \delta_{i,j}$, where $v_1, \ldots, v_n$ is the standard basis.

\begin{lem}
\label{lem_perp=comp}
If $M \in \Gr(k,n)$, then for $J \in {[n] \choose k}$, we have
\[
P_J(M) = P_{\ov{J}}(M^\perp)
\]
(as projective coordinates), where $\ov{J}$ denotes the complement $[n] \setminus J$.
\end{lem}

The proof relies on a version of Jacobi's identity for complementary minors of inverse matrices (see, e.g., \cite{Jacobi}). Given a (complex-valued) matrix $X$, define $X^c$ to be the matrix obtained from $X$ by scaling the $i^{th}$ row and column by $(-1)^i$ (so $(X^c)_{ij} = (-1)^{i+j}X_{ij}$). If $X$ is invertible, define $X^{-c} = (X^{-1})^c = (X^c)^{-1}$. Jacobi's identity implies that for two subsets $I$ and $J$ of the same cardinality, we have
\begin{equation}
\label{eq_Jacobi}
\Delta_{I,J}(X^{-c}) = \dfrac{1}{\det(X)} \Delta_{\overline{J},\overline{I}}(X).
\end{equation}

\begin{proof}[Proof of Lemma \ref{lem_perp=comp}]
Let $M'$ be an $n \times k$ matrix whose column span is $M$. Choose a $k$-subset $I$ so that $P_I(M) \neq 0$, and suppose $\ov{I} = \{i_1 < \ldots < i_{n-k}\}$. Let $X$ be the $n \times n$ matrix whose $j$th column is the standard basis vector $e_{i_j}$ for $j = 1, \ldots, n-k$, and whose last $k$ columns are the matrix $M'$. Clearly $X$ is invertible. Let $N'$ be the $(n-k) \times n$ matrix consisting of the first $n-k$ rows of $X^{-c}$. Since $X^{-1}X = I$, we have
\[
0 = \sum_{r=1}^n (-1)^{i+r}N'_{ir}M'_{rj} = (-1)^{i-1} \sum_{r=1}^n (-1)^{r+1}N'_{ir}M'_{rj}
\]
for $i = 1, \ldots, n-k$ and $j = 1, \ldots, k$. Thus, every row of the matrix $N'$ is orthogonal to every column of the matrix $M'$ with respect to the bilinear form defined above, and since these rows are linearly independent, they span the $(n-k)$-dimensional subspace $M^\perp$.

By \eqref{eq_Jacobi}, we have $P_{\ov{I}}(M^\perp) = \Delta_{[n-k],\ov{I}}(X^{-c}) = \dfrac{1}{\det(X)} \Delta_{I,[n-k+1,n]}(X) \neq 0$. Combining this with another application of \eqref{eq_Jacobi}, we obtain
\[
\frac{P_J(M)}{P_I(M)} = \dfrac{\Delta_{J,[n-k+1,n]}(X)}{\Delta_{I,[n-k+1,n]}(X)} = \dfrac{\Delta_{[n-k],\ov{J}}(X^{-c})}{\Delta_{[n-k],\ov{I}}(X^{-c})} = \frac{P_{\ov{J}}(M^\perp)}{P_{\ov{I}}(M^\perp)}
\]
for all $J \in {[n] \choose k}$. Thus, $P_J(M) = P_{\ov{J}}(M^\perp)$ as projective coordinates, as claimed.
\end{proof}

Let $Q_{w_0}$ be the permutation matrix corresponding to the longest element of $S_n$. Define $T_{w_0} : \Gr(k,n) \rightarrow \Gr(k,n)$ by $T_{w_0}(M) = Q_{w_0} \cdot M$. Note that $\Delta_{J,[k]}(T_{w_0}(M)) = (-1)^{k(k-1)/2} \Delta_{w_0(J),[k]}(M)$, so $P_J(T_{w_0}(M)) = P_{w_0(J)}(M)$ (as projective coordinates).

\begin{defn}
\label{defn_D}
Define the {\em duality map} $D : \X{k} \rightarrow \X{n-k}$ by
\[
D(M,t) = S(T_{w_0}(M^\perp),t).
\]
As usual, we write $D_t$ to denote the map $M \mapsto S_t(T_{w_0}(M^\perp))$.
\end{defn}

To understand the interaction of $D$ with the geometric crystal structure, we will show that $g^{n-k}_t \circ D_t(M)$ is closely related to the inverse of $g^k_t(M)$ (we introduce superscripts on $g$ since there are multiple Grassmannians in this section). We start by explicitly computing the inverse of $g_t^k(M)$. Define $h^k_t : X_k^\circ \rightarrow B^-$ by $h^k_t(M) = B$, where $B$ is the folded matrix given by
\begin{equation}
\label{eq_h_defn}
B_{ij} = (-1)^{i+j} c'_{ij} \dfrac{P_{[i-k,i] \setminus \{j\}}(M)}{P_{[i-k+1,i]}(M)}, \quad\quad\quad
c'_{ij} =
\begin{cases}
1 & \text{ if } i > k \\
t & \text{ if } i \leq k \text{ and } i \geq j \\
(-1)^n \lp & \text{ if } i \leq k \text { and } i < j.
\end{cases}
\end{equation}
When $k$ is clear from context, we write $h_t$ instead of $h^k_t$. For example, if $n=5$ and $k=3$, then writing $P_J = P_J(M)$, we have
\begin{equation*}
h_t(M) = \left(
\begin{array}{ccccc}
t\dfrac{P_{345}}{P_{145}} & 0 & -\lp & \lp \dfrac{P_{135}}{P_{145}} & -\lp \dfrac{P_{134}}{P_{145}} \smallskip \\
-t\dfrac{P_{245}}{P_{125}} & t\dfrac{P_{145}}{P_{125}} & 0 & -\lp & \lp \dfrac{P_{124}}{P_{125}} \smallskip \\
t \dfrac{P_{235}}{P_{123}} & -t \dfrac{P_{135}}{P_{123}} & t \dfrac{P_{125}}{P_{123}} & 0 & -\lp \smallskip \\
-1 & \dfrac{P_{134}}{P_{234}} & -\dfrac{P_{124}}{P_{234}} & \dfrac{P_{123}}{P_{234}} & 0 \smallskip \\
0 & -1 & \dfrac{P_{245}}{P_{345}} & -\dfrac{P_{235}}{P_{345}} & \dfrac{P_{234}}{P_{345}}
\end{array}
\right).
\end{equation*}

\begin{lem}
\label{lem_g_inv}
For $t \in \Cx$ and $M \in X_k^\circ$, we have $h^k_t(M) \cdot g^k_t(M) = (t + (-1)^k \lp) \cdot Id$.
\end{lem}

\begin{proof}
All matrices in this proof are folded. It's easy to see that $h \circ \PR = \sh \circ \, h$ (c.f. Lemma \ref{lem_PR_sh}(3)). Thus, since $\sh$ is an automorphism, it suffices to prove that
\[
(h_t(M) \cdot g_t(M))_{i1} = \delta_{i,1} (t + (-1)^k \lp).
\]

Set $B = h_t(M)$ and $A = g_t(M)$, and write $P_J = P_J(M)$ for the Pl\"{u}cker coordinates of $M$. By definition,
\begin{equation}
\label{eq_BA_i1}
(BA)_{i1} = \sum_{\ell} B_{i \ell} A_{\ell 1} = \sum_{\ell} (-1)^{i+\ell} c'_{i \ell} \dfrac{P_{[i-k,i] \setminus \{\ell\}}}{P_{[i-k+1,i]}} \dfrac{P_{[n-k+2,n] \cup \{\ell\}}}{P_{[n-k+1,n]}}.
\end{equation}
If $i = 1$, then $B_{i \ell} A_{\ell 1} = 0$ unless $\ell \in \{1,n-k+1\}$, so we have
\[
(BA)_{11} = t + (-1)^k \lp.
\]
If $i > 1$, then $c'_{i \ell}$ has the same value for all nonzero terms appearing in \eqref{eq_BA_i1} (the value is $t$ if $i \in [2,k]$ and $1$ if $i > k$), so we have $BA_{i1} = 0$ by the Grasmmann-Pl\"{u}cker relations (Proposition \ref{prop_Gr_Pl_rel}).
%
\end{proof}

We previously defined $X^c$, for a complex-valued matrix $X$, to be the matrix obtained by replacing $X_{ij}$ with $(-1)^{i+j}X_{ij}$. If $A \in M_n[\bbC((\lp))]$, define $A^c$ to be the folding of $X^c$, where $X$ is the unfolding of $A$. It is important to note that $A^c$ is obtained from $A$ by multiplying the $(i,j)$-entry by $(-1)^{i+j}$ and replacing $\lp$ with $(-1)^n \lp$, so $A^c_{ij} \neq (-1)^{i+j}A_{ij}$ if $n$ is odd. Define $\inv : B^- \rightarrow B^-$ on a folded matrix $A$ by
\[
\inv(A) = \adj(A)^c,
\]
where $\adj(A)$ is the adjoint of $A$ (i.e., $\adj(A)_{ij} = (-1)^{i+j} \Delta_{[n] \setminus \{j\}, [n] \setminus \{i\}}(A)$).

\begin{prop}
\label{prop_D}
For $t \in \Cx$ and $M \in X_k^\circ$, we have
\[
\beta^{n-k-1} \cdot g^{n-k}_t \circ D_t(M) = \inv \circ \, g^k_t(M),
\]
where $\beta = (t + (-1)^{k+n}\lp)$.
\end{prop}

\begin{proof}
Set $A = g^k_t(M)$. By Proposition \ref{prop_g_props}\eqref{itm:det_g}, we have
\[
\adj(A) \cdot A = \det(A) \cdot Id = \beta_0^{n-k} \cdot Id,
\]
where $\beta_0 = (t + (-1)^k\lp)$. Comparing this with Lemma \ref{lem_g_inv}, we see that $\adj(A) = \beta_0^{n-k-1} \cdot h^k_t(M)$, so
\begin{equation}
\label{eq_inv_matrix}
\inv(A)_{ij} = \beta^{n-k-1} c'_{ij} \dfrac{P_{[i-k,i] \setminus \{j\}}(M)}{P_{[i-k+1,i]}(M)},
\quad\quad
c'_{ij} = \begin{cases}
1 & \text{ if } i > k \\
t & \text{ if } i \leq k \text{ and } i \geq j \\
\lp & \text{ if } i \leq k \text{ and } i < j
\end{cases}.
\end{equation}

Set $A' = g_t^{n-k} \circ D_t(M)$. By Proposition \ref{prop_pi_g_S}, we have $A' = \fl \circ \, g^{n-k}_t(T_{w_0}(M^\perp))$. Now unravel the definitions, apply Lemma \ref{lem_perp=comp}, and compare with \eqref{eq_inv_matrix} to obtain $\beta^{n-k-1} \cdot A' = \inv(A)$.
\end{proof}

\begin{cor}
\label{cor_D}
The map $D$ has the following properties:
\begin{enumerate}
\item $D^2 = \Id$
\item $S \circ D = D \circ S$
\item $\PR \circ \, D = D \circ \PR$
\item $\vp_i \circ D = \ve_i$ and $\ve_i \circ D = \vp_i$
\item $e_i^c \circ D = D \circ e_i^{c^{-1}}$.
\end{enumerate}
\end{cor}

\begin{proof}
Set $\beta = (t + (-1)^{k+n} \lp)$ as above. By Propositions \ref{prop_pi_g_S} and \ref{prop_D} and the fact that $\inv$ and $\fl$ commute, we have
\begin{align*}
\beta^{n-k-1} \cdot g^{n-k}_t \circ S_t \circ D_t &= \beta^{n-k-1} \cdot \fl \circ \, g^{n-k}_t \circ D_t = \fl \circ \inv \circ g^k_t \\
&= \inv \circ \fl \circ \, g^k_t = \inv \circ \, g^k_t \circ S_t = \beta^{n-k-1} \cdot g^{n-k}_t \circ D_t \circ S_t.
\end{align*}
Divide by $\beta^{n-k-1}$ and apply $\pi^{n-k}_t$ to both sides to obtain (2). Part (3) is proved in the same way, using Lemma \ref{lem_PR_sh}(3) and the fact that $\inv$ and $\sh$ commute. To prove (1), let $\mu$ denote the map $M \mapsto T_{w_0}(M^\perp)$. By (2), the definition of $D$, and the fact that $S$ and $\mu$ are involutions, we have
\[
D^2 = D \circ S \circ \mu = S \circ D \circ \mu = S \circ S \circ \mu \circ \mu = \Id.
\]

Suppose $A \in B^-$. We claim that
\begin{equation}
\label{eq_vp_inv}
\vp_i(\inv(A)) = \ve_i(A).
\end{equation}
Since $\sh$ commutes with $\inv$ and $\vp_{i-1} = \vp_i \circ \sh$, it suffices to prove \eqref{eq_vp_inv} for $i = 1$. Let $A' = \inv(A)$, and let $X,X'$ be the unfolded matrices corresponding to the folded matrices $A,A'$, respectively. For $i,j \in [n]$, $X'_{ij}$ is the constant coefficient of the $(i,j)$-entry in $A'$, so we have
\[
X'_{ij} = \Delta_{[n] \setminus \{j\}, [n] \setminus \{i\}}(A|_{\lp = 0}).
\]
Since $A|_{\lp = 0}$ is lower triangular and its $(i,j)$-entry is $X_{ij}$, we have
\[
\vp_1(X') = \dfrac{X'_{21}}{X'_{11}} = \dfrac{\Delta_{[2,n],\{1\} \cup [3,n]}(A|_{\lp = 0})}{\Delta_{[2,n],[2,n]}(A|_{\lp = 0}} = \dfrac{X_{21}X_{33} \cdots X_{nn}}{X_{22}X_{33} \cdots X_{nn}} = \ve_1(X),
\]
proving \eqref{eq_vp_inv}.

Using \eqref{eq_vp_inv}, \eqref{eq_U_action_B-}, the identity $\inv(XY) = \inv(Y)\inv(X)$, and the fact that $\inv$ fixes $\wh{x}_i(a)$ for each $i$, we obtain
\begin{equation}
\label{eq_e_inv}
\inv(e_i^c(A)) = e_i^{c^{-1}}(\inv(A)).
\end{equation}

Now (4) and (5) follow from Proposition \ref{prop_D}, \eqref{eq_vp_inv}, \eqref{eq_e_inv}, and the observation that if $p(\lp)$ is any polynomial in $\lp$ with nonzero constant term, then $\vp_i(p(\lp) \cdot A) = \vp_i(A)$. This is similar to the proof of parts (3) and (4) of Corollary \ref{cor_S}, so we omit the details.
%
\end{proof}

\subsection{Tropicalizing the symmetries}
\label{sec_symms_trop}

Suppose $(M,t) \in \X{n-k}$, and set $(M',t) = S(M,t)$. Lemma \ref{lem_basic_pluc_S} shows that the basic Pl\"{u}cker coordinates of $M'$ are equal to a power of $t$ times a ratio of Pl\"{u}cker coordinates of $M$, so by Lemma \ref{lem_sub_free}(2), $\Theta S = \Theta_{n-k}^{-1} \circ S \circ \Theta_{n-k}$ is subtraction-free, and we may define $\wh{S} = \Trop(\Theta S) : \tw{\bT}_k \rightarrow \tw{\bT}_k$. Since $P_J(T_{w_0}(M^\perp)) = P_{w_0(\ov{J})}(M)$, the map $\Theta D = \Theta_k^{-1} \circ D \circ \Theta_{n-k}$ is also subtraction-free, so we may define $\wh{D} = \Trop(\Theta D) : \tw{\bT}_k \rightarrow \tw{\bT}_{n-k}$.

\begin{thm}
\label{thm_trop_S_D}
On the set of $k$-rectangles $B^k$, we have $\wh{S} = \tw{\rot}$ and $\wh{D} = \tw{\refl}$ (see \S \ref{sec_comb_symmetries}).
\end{thm}

\begin{proof}
In this proof, we write $J^r_{i,j}$ for the basic subset $[i,j] \cup [n-r+j-i+2,n]$ of size $r$.

For $(X_{ij},t) \in \bT_k$, set $(M,t) = \Theta_{n-k}(X_{ij},t)$, $(M',t) = S(M,t)$, and $(X'_{ij},t) = \Theta_{n-k}^{-1}(M',t)$. By Proposition \ref{prop_GT_pl} and Lemma \ref{lem_basic_pluc_S}, we have
\[
X'_{ij} = \dfrac{P_{J^{n-k}_{i,j}}(M')}{P_{J^{n-k}_{i+1,j}}(M')} = \dfrac{t^{\min(j-i+1,k-i+1)} P_{J^{n-k}_{k-i+2,n-j}}(M)}{t^{\min(j-i,k-i)} P_{J^{n-k}_{k-i+1,n-j}}(M)} = \dfrac{t}{X_{k-i+1,n-j}}.
\]
Tropicalizing this equality and comparing with the definition of $\tw{\rot}$, we see that $\wh{S} = \tw{\rot}$.

Now set $(M'',t) = D(M,t)$ and $(X''_{ij},t) = \Theta_k^{-1}(M'',t)$. By definition, $M'' = S_t(T_{w_0}(M^\perp))$, so we have
\[
X''_{ij} = \dfrac{P_{J^k_{i,j}}(S_t(T_{w_0}(M^\perp)))}{P_{J^k_{i+1,j}}(S_t(T_{w_0}(M^\perp)))} = t \dfrac{P_{J^k_{n-k-i+2,n-j}}(T_{w_0}(M^\perp))}{P_{J^k_{n-k-i+1,n-j}}(T_{w_0}(M^\perp))} = t \dfrac{P_{J^{n-k}_{j-i+2,j}}(M)}{P_{J^{n-k}_{j-i+1,j}}(M)} = \dfrac{t}{X_{j-i+1,j}}.
\]
Tropicalizing and comparing with the definition of $\tw{\refl}$, we conclude that $\wh{D} = \tw{\refl}$.
\end{proof}

\begin{remark}
\label{rmk_symms_cryst_ops}
Parts (2) and (3) of Proposition \ref{prop_symms_cryst_ops} follow from combining this result with Theorem \ref{thm_trop_cryst}(3,4), Corollary \ref{cor_S}(4), and Corollary \ref{cor_D}(5).
\end{remark}

\bibliographystyle{plain}
\bibliography{affine_geom_cryst_paper.refs}

\end{document}